\documentclass[a4paper]{amsart}
\usepackage{amssymb,enumerate}
\usepackage[chatter]{rotating}
\usepackage{epsfig}
\usepackage{fancyvrb}
\newtheorem{lemma}{Lemma}[section]

\newtheorem{theorem}[lemma]{Theorem}
\newtheorem{remark}[lemma]{Remark}
\newtheorem{cor}[lemma]{Corollary}
\newtheorem{proposition}[lemma]{Proposition}

\newcommand{\N}{{\mathbb N}}

\newcommand{\Amo}{{\mathcal {\bf A}}}
\newcommand{\A}{ {\hat A}}

\newcommand{\Ro}{{\overline R}}

\newcommand{\Vo}{{\bf V}}

\newcommand{\So}{{\hat {\mathcal F}}}
\newcommand{\ad}{{\rm ad}}
\newcommand{\aV}{{\hat V}}

\newcommand{\aR}{{\hat R}}

\newcommand{\Var}{{\rm var}}

\newcommand{\Q}{{\mathbb Q}}

\newcommand{\F}{{\mathbb F}}

\newcommand{\Z}{{\mathbb Z}}

\newcommand{\im}{\rm im}

\newcommand{\lm}{\lambda}
\newcommand{\lmu}{\lambda_1}
\newcommand{\lmf}{\lambda_1^f}
\newcommand{\lmd}{\lambda_2}
\newcommand{\lmdf}{\lambda_2^f}
\newcommand{\lmt}{\lambda_3}
\newcommand{\al}{\alpha}
\newcommand{\bt}{\beta}
\newcommand{\gm}{\gamma}
\newcommand{\dl}{\delta}
\newcommand{\M}{\mathcal M}

\newcommand{\Mab}{\mathcal M(\al,\bt)}

\newcommand{\hW}{{\hat W}}
\newcommand{\hJ}{{\hat J}}
\newcommand{\ha}{{\hat a}}
\newcommand{\hw}{{\hat w}}
\newcommand{\hv}{{\hat v}}
\newcommand{\hu}{{\hat u}}

\newcommand{\bv}{{\bf v}}
\newcommand{\ba}{{\bf a}}
\newcommand{\bc}{{\bf c}}
\newcommand{\bw}{{\bf w}}
\newcommand{\bu}{{\bf u}}
\newcommand{\bs}{{\bf s}}
\newcommand{\bp}{{\bf p}}
\newcommand{\bq}{{\bf q}}

\usepackage{xcolor}

\title[]{$2$-generated axial algebras of Monster type}
\author{Clara Franchi, Mario Mainardis, Sergey Shpectorov}
\begin{document}
\maketitle
Clara Franchi, Dipartimento di Matematica e Fisica,
Universit\`a Cattolica del Sacro Cuore,
Via della Garzetta 48,
I-25133 Brescia, Italy,
clara.franchi@unicatt.it\\

Mario Mainardis, 
Dipartimento di Scienze Matematiche, Informatiche e Fisiche, 
Universit\`a degli Studi di Udine, via delle Scienze 206,
I-33100 Udine, Italy,\\
 mario.mainardis@uniud.it\\

Sergey Shpectorov, School of Mathematics,
University of Birmingham, 
Watson Building, Edgbaston,
Birmingham, B15 2TT, UK, s.shpectorov@bham.ac.hk\\

keywords: axial algebras, Griess algebra, Monster Group.

\begin{abstract}
We provide the basic setup for the project, initiated by Felix Rehren 
in~\cite{R}, aiming at classifying all $2$-generated axial algebras of 
Monster type $(\al,\bt)$ over a field $\F$.  Using this, we first show that every such algebra has dimension at most $8$, except for the case $(\al,\bt)=(2,\tfrac{1}{2})$, where the Highwater algebra provides examples of dimension $n$, for all $n\in \N\cup \{\infty\}$. 
We then classify all 2-generated axial 
algebras of Monster type $(\al,\bt)$ over $\Q(\al,\bt)$, for $\al$ and 
$\bt$ algebraically independent over $\Q$. Finally, we generalise the Norton-Sakuma Theorem to every primitive $2$-generated axial algebra of Monster type $(\frac{1}{4},\frac{1}{32})$ over a field of characteristic zero, dropping the hypothesis on the existence of a Frobenius form.  
\end{abstract}

\maketitle

\section{Introduction}

Primitive axial algebras of Monster type constitute a class of commutative non-associative algebras 
generated by idempotent elements called {\it axes}, whose adjoint action is 
semisimple and the eigenvectors satisfy a prescribed fusion law. Axial 
algebras were introduced by Hall, Rehren and Shpectorov~\cite{HRS,HRS1} in 
order to axiomatise key features of important classes of algebras, including 
the Majorana algebras~\cite{Mb}, weight-2 components of OZ-type 
vertex operator algebras, Jordan algebras, and Matsuo algebras.
%(see the introductions of \cite{HRS}, \cite{R} and~\cite{3A})
 They are of 
particular interest for finite group theorists as a majority of the finite 
simple groups, including many sporadic simple groups, can be represented as 
automorphism groups of such algebras. For a motivation and a more detailed introduction to this subject we refer to the survey papers~\cite{surveyI}, and~\cite{surveyS}.

In this paper we provide the basic setup for an ongoing project, initiated by Rehren~\cite{R}, aimed at classifying $2$-generated primitive axial algebras of Monster type over fields of characteristic other than $2$.\footnote{Recall that, by an argument of M. Stout, in every primitive axial algebra of Monster type over a field of characteristic $2$ the $\alpha$ and $\beta$ eigenspaces of the adjoint actions of the axes are trivial and the algebra is associative (see~\cite[Lemma 2.1]{HW}).} Such a classification is fundamental  for the understanding and application of this theory (see~\cite{Novosibirsk}). While this paper was in preparation, the  classification of $2$-generated primitive \textit{symmetric}\footnote{Recall that a $2$-generated axial algebra of Monster type is \textit{symmetric} if the map that swaps the two generating axes extends to an automorphism of the entire algebra (see~\cite{FMS2}).} axial algebras of Monster type  has been fully completed by Yabe~\cite{Yabe}, Franchi, Mainardis, and McInroy~\cite{HW5, FMM}. The present paper now provides foundations for a more general 
approach that does not distinguish between the symmetric and non-symmetric case.

We begin by revising Sections 2 and 3 of Rehren's paper~\cite{R}. In Section~\ref{nuovasec} we discuss some properties of endomorphisms of modules over a commutative associative ring with identity, which might be of independent interest. These results are used in  Sections~\ref{2}-\ref{universal}, where we define, for a fixed fusion law, the category of axial algebras with $k$-marked generators and give a formal construction of the initial object in this category, filling some gaps in Rehren's proof (in particular regarding torsion freeness of axes). In Section~\ref{table} we specialise the above construction to the case of primitive $2$-generated 
algebras of Monster type $(\al, \bt)$. These fall naturally into three cases ({\it Rehren's Trichotomy}): 
\medskip 
\begin{enumerate}
\item [-]  {\it  regular case $\al\not \in \{2\bt, 4\bt\}$}, considered by Rehren~\cite[Theorem~3.7]{R} and reproved here in a simpler way;
 \item [-] {\it  exceptional case $\al=2\bt$}, which has been completely solved in~\cite{FMS2}; 
\item [-] {\it  exceptional case $\al=4\bt$}, treated in this paper (Theorem~\ref{newth}).
\end{enumerate}
\medskip 

Combining the above results we obtain 
%the {\it regular} case, 
%where $\al\not \in \{2\bt, 4\bt\}$, and the two {\it 
%exceptional} cases $\al=2\bt$  and $\al=4\bt$. The regular case was originally considered by Rehren~\cite[Theorem~3.7]{R},  here we provide an alternative and simpler proof of it. The exceptional case $\al=2\bt$  has been completely solved in~\cite{FMS2}. {\red We then treat in some detail the case $\al=4\bt$, obtaining various new relations.} Combining these with~\cite[Proposition~3.7]{FMS2} we get
\begin{theorem}\label{Felix}
Every $2$-generated primitive axial algebra of Monster type $(\alpha, \beta)$ over a field $\F$ of characteristic other than $2$ has dimension at most $8$, provided $(\alpha, \beta)\neq (2, \tfrac{1}{2})$.
\end{theorem}

This bound is best possible: examples of $8$-dimensional algebras can be found in~\cite[Table 3]{surveyS} for each of the above three cases.
The case $(\alpha, \beta)=(2,\frac{1}{2})$ is truly exceptional, since the Highwater algebra~\cite{HW} and its characteristic $5$ cover~\cite{FMM} are examples of infinite dimensional $2$-generated symmetric primitive axial algebras of Monster type $(2,\frac{1}{2})$ and have quotients of dimension $n$, for every $n\in \N$ (see~\cite{HW5}). 

In the regular case, we provide explicit formulas (or a reference to a code for computing them, when too large to be displayed) for the structure constants of the initial object. These formulas, taken from the arXiv first version of this paper, have been used in~\cite{turner1, turner2}. Some of these also appeared, with an alternative formulation, in~\cite{R} (note that  a pair of them, namely those in Lemmas~3.2 and~3.3 of~\cite{R}, contain misprints that we have corrected). The analogous formulas in the case $\al=2\bt$ can be found in~\cite{FMS2}.

%In the case $\al=4\bt$, we prove\footnote{This result can be also obtained as a consequence of~\cite{Yabe} and \cite{HW5}. Here we give a more direct proof, independent of those results.} that the above bound holds for symmetric algebras, except for the case  $(\al, \beta)=(2,\frac{1}{2})$.
%\begin{theorem}\label{dim}
%Every primitive $2$-generated symmetric axial algebra of Monster type $(4\bt, \beta)$ over a field $\F$ of characteristic other than $2$ has dimension at most $8$, except possibly when $\beta=\frac{1}{2}$.
%\end{theorem} 
%The case $(\alpha, \beta)=(2,\frac{1}{2})$ is truly exceptional, since the Highwater algebra~\cite{HW} and its characteristic $5$ cover~\cite{HW5} are examples of infinite dimensional $2$-generated symmetric primitive axial algebras of Monster type $(2,\frac{1}{2})$. 

In Section~\ref{generic} we consider further in detail the regular case. Denote by $\F_0$ the prime subfield of $\F$ and let $\F_0(\alpha, \beta)[x, y, z, t]$ be the polynomial ring in $4$ variables over $\F_0(\alpha, \beta)$. Let $\mathbb A_\F^4$ be the $4$-dimensional affine space over $\F$. For a subset $X$ of $\F[x, y, z, t]$, denote  by $\Var(X)$ 
%the variety in $\mathbb A_\F^4$ corresponding to the ideal of $\F[x, y, z, t]$ generated by $T$ (i.e. 
the set of the common zeros in $\mathbb A_\F^4$ of the elements in $X$. Finally 
let ${\mathcal M_r}(2, \F)$ be a set of representatives of the isomorphism classes of  $2$-generated primitive axial algebras of Monster type $(\alpha, \beta)$ over $\F$ with  $\alpha\not \in \{2\beta, 4\beta\}$. We define four polynomials $p_1$, $p_2$, $p_3$, $p_4\in \F_0(\alpha, \beta)[x, y, z, t]$. These polynomials are much too large to be displayed here, but they can be computed with the code~\cite[genericsakuma.s]{code}.
We prove the following result.

\begin{theorem}\label{nec}
Assume $\F$ is a field of characteristic other than $2$. Then there is a map $$\xi\colon {\mathcal M_r}(2, \F)\to \mathbb A_\F^4$$ 
such that, 
\begin{enumerate}
\item ${\im} (\xi)\subseteq {\Var}(\{p_1, p_2, p_3, p_4\})$;
\item for every $P$ in $\im (\xi)$, there exists an element $V_P$ of  ${\mathcal M_r}(2, \F)$ such that every element in $\xi^{-1}(P)$ is a quotient of $V_P$.
\end{enumerate} 
\end{theorem}

This means that every $V\in   {\mathcal M_r}(2, \F)$ is determined, up to quotients, by a $4$-tuple of parameters $(\lmu, \lmf, \lmd, \lmdf)$ in $\F^4$ which must be a common root of the polynomials $p_1$, $p_2$, $p_3$, and $p_4$.  We leave open the question for which values of $\al$ and $\bt$ these polynomials are algebraically independent. An answer to this question could be interesting, since, in such a case, the set  ${\Var}(\{p_1, p_2, p_3, p_4\})$, hence the choices for $(\lmu, \lmf, \lmd, \lmdf)$,  would be finite.

Since every $2$-generated primitive axial algebra $V$ of Monster type~$(\al, \bt)$ contains two large symmetric subalgebras $V_e$ (the \textit{even subalgebra} of $V$) and $V_o$ (the \textit{odd subalgebra} of $V$) (see Remark~\ref{evenodd} for the definitions of $V_e$ and $V_o$),  combining the classification of the symmetric algebras with Theorem~1.1 in~\cite{axet} one can get a list of  possibile pairs $(V_e, V_o)$, which turns to be  finite except for the case $(\al, \bt)=(2,\frac{1}{2})$. As a consequence, under the hypotheses of Theorem~\ref{nec}, we get that $\xi^{-1}(P)$ can be nonempty only if the last two coordinates of $P$ vary within a finite set, which can be explicitly computed. This observation provides a useful tool towards the full classification of $2$-generated primitive axial algebras of Monster type (this strategy has been successfully employed in~\cite{FMS2}, using the corresponding version of Theorem~\ref{nec}, for the case $\al=2\bt$).

In the last two sections we apply the above strategy in two relevant cases. Namely, in Section 7 we classify  the $2$-generated primitive axial algebras of Monster type $(\alpha, \beta)$ defined over the field $\Q(\al, \bt)$, with $\al$ and $\bt$ algebraically independent indeterminates over $\Q$. Referring to~\cite{HRS1} for the definition of the algebras of type $1A$, $2B$, $3C(\eta)$, $\eta\in \F$ and denoting by $3A(\al,\bt)$ the algebra of dimension $4$ defined in~\cite[Table~9]{R} for $\al\neq \frac{1}{2}$, we prove the following result.

\begin{theorem}\label{symmetric} 
Let $V$ be a $2$-generated primitive axial algebra of Monster type $(\al, \bt)$ over the field $\Q(\al, \bt)$,  with $\al$ and $\bt$ algebraically independent indeterminates over $\Q$. Then we have one of  the following:
\begin{enumerate}
\item $V$ is the trivial algebra $\Q(\al, \bt)$ of type $1A$;
\item $V$ is an algebra of type $2B$;
\item $V$ is an algebra of Jordan type $\al$ of type $3C(\al)$;
\item $V$ is an algebra of Jordan type $\bt$ of type $3C(\bt)$;
\item $V$ is an algebra of dimension $4$ of type $3A(\al, \bt)$. 
\end{enumerate}
In particular, $V$ is symmetric.
\end{theorem}

Recall that the fusion law  $\M( \alpha, \beta)$ admits a $\Z_2$-grading $\M( \alpha, \beta)_+=\{1,0,\al\}$ and $\M( \alpha, \beta)_-=\{\bt\}$ and this implies that every axis $a$ induces an automorphism of the algebra $\tau_a$ called Miyamoto involution. The Miyamoto group is the  group generated by all Miyamoto involutions (see~\cite{KMS}). 

\begin{cor}\label{3-transposition}
Let $W$ be a primitive finitely generated axial algebra of Monster type $(\al, \bt)$ over the field $\Q(\al, \bt)$,  with $\al$ and $\bt$ independent indeterminates over $\Q$. Then the Miyamoto group of $W$ is a group of $3$-transpositions.
\end{cor}

Finally, in Section 8 we consider the case where $(\al, \bt)=(\frac{1}{4},\frac{1}{32})$, which includes the subalgebras of the Griess algebra (whose automorphism group is the Monster group $M$) that are generated by two axes. These subalgebras were first described by Norton (see~\cite[Theorem 3, Section 7]{N96}\footnote{Note that,  in order to to have only integers in the algebra rules,  Norton used a different scaling:  e.g. his 2A-axes are 8 times the Majorana axes (see~\cite[p. 230]{ATLAS}).}) and fall into nine $M$-orbits. These orbits are labelled  accordingly to the notation in the ATLAS for the conjugacy class in $M$ of the product of the $2A$-involutions associated to the generating axes. They are displayed here as the nodes of McKay's extended $E_8$ graph:

\begin{picture}(80,60)(-60,-38)
\put(-5,5){$1A$}
\put(0,0){\circle{4}}
\put(2,0){\line(1,0){26}} 
\put(25,5){$2A$}
\put(30,0){\circle{4}}
\put(32,0){\line(1,0){26}} 
\put(55,5){$3A$}
\put(60,0){\circle{4}}
\put(62,0){\line(1,0){26}} 
\put(85,5){$4A$}
\put(90,0){\circle{4}}
\put(92,0){\line(1,0){26}} 
\put(115,5){$5A$}
\put(120,0){\circle{4}}
\put(122,0){\line(1,0){26}}
\put(145,5){$6A$}
\put(150,0){\circle{4}}
\put(152,0){\line(1,0){26}}
\put(175,5){$4B$}
\put(180,0){\circle{4}}
\put(182,0){\line(1,0){26}}
\put(205,5){$2B$}
\put(210,0){\circle{4}}
\put(150,-2){\line(0,-1){26}}
\put(150,-30){\circle{4}}
\put(155,-32){$3C$}
\end{picture}

In~\cite{S} Sakuma proved that Norton's classification holds more generally for certain vertex operator algebras (including the Moonshine module). Because of that, these algebras are now known as Norton-Sakuma algebras. Later this classification was extended to $2$-generated primitive axial algebras of Monster type $(\frac{1}{4},\frac{1}{32})$ over a field of characteristic zero endowed with a Frobenius form (see~\cite{IPSS10, HRS}).
Here we prove that result in full generality, showing that the assumption on the existence of a Frobenius form can be dropped:
\begin{theorem}\label{thm}
Every $2$-generated primitive axial algebra of Monster type $(\frac{1}{4},\frac{1}{32})$ over a field of characteristic zero is a Norton-Sakuma algebra.
\end{theorem}
  
In order to prove Theorems~\ref{symmetric} and~\ref{thm} we made use of the computer algebra system Singular~\cite{Singular} to solve certain systems of polynomial equations. The codes are available in~\cite{code}. 
%%%%%%%%%%%%%%%%%%%%%%%%%%%%%%%%%%%%%%%%%%%%%%%
\section{Semisimple endomorphisms of $R$-modules} \label{nuovasec}
Let  $R$ be a commutative associative ring with identity. In this section we extend to $R$-modules some well known results about endomorphisms of vector spaces. The main difference is that,  when considering $R$-modules instead of vector spaces, it is no longer true, in general, that eigenvectors relative to different eigenvalues are linearly independent.

Let $V$ be an $R$-module.  For $\xi\in End(V)$, $\lambda \in R$, and $\Gamma \subseteq R$, define 
$$V_\lambda^\xi:=\{v\in V| \xi(v)=\lambda v\} \mbox{ and } V_{\Gamma}^\xi:=\sum_{\lambda\in \Gamma}V_\lambda^\xi.$$
If $V$ is also an $R$-algebra and $a\in V$,  denote by $\ad_a$ the endomorphism of $V$ induced by right multiplication by $a$: 
$$\begin{array}{rccc}
\ad_a:& V&\to &V\\
& x&\mapsto& xa
\end{array}.
$$
In this case, we'll write simply $V_\lambda^a$ and $V_\Gamma^a$ instead of $V_\lambda^{\ad_a}$ and $V_\Gamma^{\ad_a}$, respectively.

Two elements $\alpha$ and $\beta$ of $R$ are called \textit{distinguishable}  if $\alpha-\beta$ is a unit in $R$. In the remainder of this section we shall  assume that $\Gamma$ is a finite set of pairwise distinguishable elements of $R$ or $|\Gamma|=1$. 

\begin{remark} \label{homd}
Since, every nontrivial homomorphism of unitary rings maps invertible elements to invertible elements, it also maps pairs of distinguishable elements into pairs distinguishable elements. 
\end{remark}
Let $R[x]$ be the ring of polynomials over $R$ with indeterminate $x$. The following result is an obvious generalisation of the Remainder Theorem.
\begin{lemma}\label{Ruffini}
If $g\in R[x]$ satisfies  $g(\lambda)=0$ for every $\lambda \in \Gamma$ and $\deg g<|\Gamma|$, then $g=0$.
\end{lemma}
\begin{proof}
We proceed by induction on $|\Gamma|$. If $|\Gamma|=1$, then $g$ is a constant and since the value of $g$ is zero in at least one element (the element of $\Gamma$), it must be $g=0$. Suppose $|\Gamma|>1$.  Assume by contradiction that $g$ is not the zero polynomial and set $k:=\deg g$. Then $k<|\Gamma|$. Let $\lambda \in \Gamma$, then, since $x-\lambda$ is monic, there exist $q,r\in R[x]$ such that 
$$g=q\cdot (x-\lambda)+r,
$$
 with $r=g(\lambda)=0$.  Thus $g=q\cdot (x-\lambda)$, whence $\deg q=k-1<|\Gamma\setminus \{\lambda\}|$.
Also, for $\mu \in \Gamma\setminus \{\lambda\}$, we have
$$0=g(\mu)=q(\mu)\cdot (\mu-\lambda).
$$ 
Since $\mu$ and $\lambda$ are distinguishable, $(\mu-\lambda)$ is a unit and so $q(\mu)=0$. By the inductive hypothesis, it follows that $q=0$, whence also $g=0$, a contradiction.\end{proof}

For $\mu \in \Gamma$, define 
\begin{equation}\label{fmu}
f_\mu :=\prod_{\lambda\in {\Gamma}\setminus \{\mu\}} (x-\lambda), 
\end{equation}
and 
\begin{equation}\label{f}
f:=\prod_{\lambda\in {\Gamma}} (x-\lambda), 
\end{equation} 
clearly 
$$
f=(x-\mu)f_\mu 
$$
 for every $\mu \in \Gamma$.
Note that, since elements of $\Gamma$ are pairwise distinguishable, $f_\mu(\mu)$ is  a unit in $R$.  
\begin{cor}\label{poly1}
With the above notation
$$
\sum_{\mu \in \Gamma} \frac{1}{f_\mu(\mu)} f_\mu=1.
$$
\end{cor}
\begin{proof}
Define 
$$g:=\sum_{\mu \in \Gamma} \frac{1}{f_\mu(\mu)} f_\mu-1,
$$ then $\deg g<|\Gamma|$ and clearly $g(\lambda)=0$ for every $\lambda\in \Gamma$. Hence, by Lemma~\ref{Ruffini}, $g=0$.
\end{proof}

For $\xi\in End(V)$, $\mu\in \Gamma$, let 
\begin{equation}\label{proj}
\pi_\mu^\xi:=\frac{1}{f_\mu(\mu)}f_\mu(\xi).
\end{equation}
Obviously, $\pi_\mu^\xi\in End(V)$.

\begin{proposition}\label{equivalence}
Let $V$ be an $R$-module. For every $\xi\in End(V)$, the following statements are equivalent:
\begin{enumerate}
\item 
$
V=\bigoplus _{\lambda \in \Gamma} V_\lambda^\xi
$;
\item $V=V_{\Gamma}^\xi$;
\item  $f(\xi)(V)=0$.
\end{enumerate}
Moreover, if these hold, then, for every $v\in V$, we have $v=\sum_{\mu\in\Gamma}v_{\mu}$ where
$$v_\mu:=\pi_\mu^\xi(v)\in V_\mu^\xi.
$$
(In other words, $\pi_\mu^\xi$ is the projection of $V$ onto $V_\mu^\xi$ with respect to the decomposition in {\it (1)}, in particular for $w\in V_\mu^\xi$, $\pi_\mu^\xi(w)=w$).
\end{proposition}
\begin{proof}
{\it (1)} implies {\it (2)} by the definition of $V_{\Gamma}^\xi$. Since $f(\xi) (V_\lambda^\xi)=\{0\}$ for every $\lambda\in \Gamma$, {\it (2)} implies {\it (3)}. Suppose {\it (3)} is satisfied. Then, for every  $\mu \in \Gamma$, and $v\in V$,
\begin{eqnarray*}
0&=&(f(\xi))(v)=\left (\prod_{\lambda\in {\Gamma}} (\xi-\lambda)\right )(v)\\
&=&(\xi-\mu)\left (\prod_{\lambda\in {\Gamma\setminus\{\mu\}}} (\xi-\lambda)\right )(v)\\
&=&(\xi-\mu)((f_\mu(\xi))(v)),
\end{eqnarray*}
whence $(f_\mu(\xi))(v)\in V_\mu^\xi$. Set
$$v_\mu:=\pi_\mu^\xi(v).
$$  
By Corollary~\ref{poly1}, 
$$id_V=
\sum_{\mu \in \Gamma} \pi_\mu^\xi
$$
and so 
$$
v=\sum_{\mu \in \Gamma} \pi_\mu^\xi(v)=\sum_{\mu\in \Gamma} v_\mu .
$$
Now, assume $v=\sum_{\mu \in \Gamma} w_\mu$ for some $w_\mu \in V_\mu^\xi$. Since   
$$
\pi_\mu^\xi(w_\lambda)=\delta_{\lambda\mu}w_\mu
$$ 
(where $\delta_{\lambda\mu}$ is the Kronecker delta), we get 
$$
v_\mu=\pi_\mu^\xi(v)=\sum_{\mu \in \Gamma}\pi_\mu^\xi(w_\lambda)=w_\mu,
$$
giving {\it (1)}.
\end{proof}

%%%%%%%%%%%%%%%%%%%%%%%%%%%%%%%%%%%%%%%%%%%%%%%%%%%%%%%%%%%%%%%%
\section{Primitive axial algebras} \label{2}

\subsection{Fusion laws}
Recall~\cite{DSV} that a {\it fusion law} is a pair $(\mathcal F, \ast)$ such that $\mathcal F$ is a set and $\ast$ is a map from the cartesian product ${\mathcal F}\times {\mathcal F}$ to the power set $2^{\mathcal F}$. 
A {\it morphism} between  two fusion laws 
$(\mathcal F_{1}, \ast_1)$ and $(\mathcal F_{2}, \ast_2)$ is a map 
$$
\phi\colon {\mathcal F_{1}} \to {\mathcal F_{2}}
$$ 
such that, for $\alpha,\beta\in {\mathcal F_{1}}$, 
$$
\phi(\alpha\ast_1 \beta)\subseteq \phi(\alpha) \ast_2\phi(\beta).
$$
An {\it isomorphism} of fusion laws is a bijective morphism such that its inverse is also a morphism. 
A  fusion law $(\mathcal F, \ast)$ is said to be {\it finite} if ${\mathcal F}$ is a finite set.  In this paper we deal with fusion laws $(\mathcal F, \ast)$ where $\mathcal F$ is a  finite set containing the spectrum of the adjoint action of an idempotent element in an $R$-algebra. Therefore, we assume $1_R\in \mathcal F\subseteq R$ and $1_R\in 1_R\ast 1_R$.  Without loss of generality, we may also assume that  $0_R\in \mathcal F$.  Further, for every morphism $\phi$ of fusion laws, we'll assume that $1^\phi=1$ and $0^\phi=0$.  More generally, when considering morphisms between fusion laws, one may want to preserve some possible algebraic relations between the elements of the set $\mathcal F$. To this aim, for two fusion laws $\mathcal F$ and $\mathcal F^\prime$, let $F$ (resp. $F^\prime$) be the subring generated by $\mathcal F$ (resp. $\mathcal F^\prime$). We call a morphism of fusion laws $\phi\colon \mathcal F\to \mathcal F^\prime$ an {\it algebraic morphism} if it is the restriction to $\mathcal F$ of a ring homomorphism $F \to F^\prime$.  
In particular, let $I$ be an ideal of $R$ and define
\begin{equation}\label{quotientfusion}
\mathcal F_{R/I}:=\{ \mu+I\:|\: \mu\in \mathcal F\}.
\end{equation}
%Under suitable hypothesis on $\mathcal F$, the canonical projection of $R$ onto $R/I$ induces an algebraic isomorphism of fusion laws.

\begin{remark}\label{mor}
With the above notation, assume the elements of $\mathcal F$ are pairwise distinguishable and $I\neq R$. Then, by Remark~\ref{homd}, $|\mathcal F_{R/I}|=|\mathcal F|$, so, for every $\tilde{\mu}\in \mathcal F_{R/I}$, there exists a unique $\mu\in \mathcal F\cap \tilde{\mu}$. Setting, 
for every $\mu$ and $\delta$ in $\mathcal F$, 
$$
(\mu+I) \ast_{R/I} ( \delta+I)= \{ \lambda+I\:|\:\lambda \in \mu\ast \delta \}
$$
we obtain a map
\begin{equation}\label{asti}
\ast_{R/I} :\mathcal F_{R/I} \times \mathcal F_{R/I} \to 2^{\mathcal F_{R/I}}
\end{equation}
such that  the canonical projection $R\to R/I$ induces an algebraic isomorphism between $ (\mathcal F, \ast)$ and $({\mathcal F}_{R/I}, \ast_{R/I} )$.
\end{remark}

\subsection{Axial algebras}
Let $V$ be an algebra over a ring $R$ and $ (\mathcal F, \star)$ a fusion law with $\mathcal F\subseteq R$. 
An element $a\in V$ is said to be an $\mathcal F$-{\it axis} (or simply {\it axis} when $\mathcal F$ is clear from the context) if
\begin{enumerate}
\item [Ax0] $a$ is an idempotent, that is $a^2=a$;
\item[Ax1] $V=V_{\mathcal F}^a$;
\item[Ax2] $V_\lambda^a V_\mu^a \subseteq V_{\lambda\star \mu}^a$ for every $\lambda, \mu \in \mathcal F$. 
\end{enumerate}
Further, $a$ is called {\it primitive} if, 
\begin{enumerate}
\item[Ax3]   $V_1^a=Ra$.
\end{enumerate}
An {\it axial algebra over $R$  with generating set $A$ and fusion law $(\mathcal F, \star)$} is a quadruple 
$$(R,V, A, (\mathcal F, \star))$$
such that 
\begin{enumerate}
\item $R$ is an associative commutative ring with identity $1$;
\item $\mathcal F$ is a subset of $R$ containing $1$ and $0$; 
\item $V$ is a commutative non associative $R$-algebra;
\item $A$ is a set of $\mathcal F$-{\it axes}.
\end{enumerate} 
An axial algebra $(R,V, A, (\mathcal F, \star))$ is called {\it primitive} if every element of $A$ is primitive.

A {\it Frobenius} axial algebra is an axial algebra $(R,V, A, (\mathcal F, \star))$ endowed with an associative bilinear form $\kappa:V\times V\to R$ such that the map $a\mapsto \kappa(a,a)$ is constant on the set of axes.

\begin{remark}\label{associative}
By Pierce decomposition, a commutative associative unitary algebra generated by idempotents is an axial algebra with respect to the fusion law $(\mathcal F_a, \circ)$, where $\mathcal F_a=\{0,1\}$ and 
$$
1\circ 1=\{1\}, \:\:1\circ 0=0\circ 1=\emptyset,  \:\:0\circ 0=\{0\}.
$$
 \end{remark}
 
In the sequel we will consider only fusion laws $(\mathcal F, \star)$ satisfying the conditions
\begin{enumerate}
\item [F1] $\:\:\:\:\:1\in 1\star 1$ 
\item [F2] $\:\:\:\:\:0\in 0\star 0$.
\end{enumerate}

Thus, we have

\begin{lemma}\label{any}
Any commutative associative unitary algebra $V$ generated by a set $A$ of idempotents is an axial algebra with respect to any fusion law $(\mathcal F, \star)$.
\end{lemma}
\begin{proof}
Let $a\in A$. Then $V_\mu^a=\{0\}$ for every $\mu\in \mathcal F\setminus \{1,0\}$. So, by Remark~\ref{associative}, $a$ trivially satisfies axioms Ax0, Ax1, Ax2. 
\end{proof}

Let $(R,V, A, (\mathcal F, \star))$ be an axial algebra and 
 assume the elements of $\mathcal F$ are pairwise distinguishable. As in the proof of  Proposition~\ref{equivalence}, for every $v\in V$, denote by $v_1$ the projection of $v$ into $V_1^a$ with respect to the decomposition of $V$ into $\ad_a$-eigenspaces. If $V$ is primitive, we have 
 $$v_1=\lambda_a(v)a$$ 
 for some $\lambda_a(v) \in R$ which is generally not unique. On the other hand, if  the annihilator ideal $$Ann_R(a):=\{r\in R| ra=0\}$$ 
of $a$ in $R$ is trivial, then $\lambda_a(v)$ is unique, and we say that $a$ is a {\it free} axis. Clearly this condition is satisfied when $R$ is a field. 
As an immediate consequence we have  

\begin{lemma}\label{lambdafunction}
Let $(R,V, A, (\mathcal F, \star))$ be a primitive axial algebra and assume that the elements of $\mathcal F$ are pairwise distinguishable and the axes in $A$ are free. Then, for every $a\in A$, there is a well defined $R$-linear map 
\begin{equation}
\label{lambda}
\begin{array}{rcccc}
\lambda_a & \colon & V& \to & R\\
&& v &\mapsto & \lambda_a(v)
\end{array}
\end{equation}
such that every $v\in V$ decomposes uniquely as
\begin{equation}\label{justapply}
v=\lambda_a(v)a+\sum_{\mu\in \mathcal F\setminus \{1\}}v_\mu, 
\end{equation}
with $v_\mu\in V_\mu^a$.
\end{lemma}
We call the map $\lambda_a$ the {\em projection map}  associated to the axis $a$. Note that is $V$ is a Frobenius axial algebra with bilinear map $\kappa$, then for every $v\in V$, $\lambda_a(v)=\kappa(a,v)$ (see~\cite{HRS}). As we shall see, the projection maps will play a crucial r\^ole in the classification of primitive axial algebras.

We say that an axis $a$ of $V$ is {\it weak primitive} if every element $v$ in $V$ can be decomposed as follows:
\begin{equation}\label{wp1}
v=l_a(v)a+\sum_{\mu\in \mathcal F\setminus \{1\}}v_\mu
\end{equation}
where $l_a(v)$ is an element of $R$ which depends on $v$ and $a$, and $v_\mu\in V^a_\mu$. 
Note that in general the decomposition in (\ref{wp1}) and $l_a(v)$ are not uniquely determined by $v$. If $a$ is weak primitive for every $a\in A$, we say that $V$ is {\it weak primitive}.
\begin{lemma}\label{pp}
If the elements of $\mathcal F$ are pairwise distinguishable, in particular, if $R$ is a field, then weak primitivity is equivalent to primitivity. 
\end{lemma}
\begin{proof}
Trivially, primitivity implies weak primitivity. Since the elements $\mathcal{S}$ are pairwise distiguishable the other direction is immediate from Lemma 2.3, namely suppose that $V$ is weak primitive, fix $a\in A$ and let $v_1\in V^a_1$. Then, by weak primitivity, there exist $l\in R$, $v_\mu\in V^a_\mu$, such that 
$$
v_1=la+\sum_{\mu\in \mathcal F\setminus \{1\}}v_\mu.
$$
Hence 
$$\sum_{\mu\in \mathcal F\setminus \{1\}}v_\mu=v_1-la\in (\sum_{\mu\in \mathcal F\setminus \{1\}}V^a_\mu )\cap V^a_1,
$$
and the last intersection is  trivial by Proposition~\ref{equivalence}. Thus $v_1=la\in Ra$.
\end{proof}

\begin{lemma}\label{condition}
Let $V$ be an algebra over a ring $R$ and $ (\mathcal F, \star)$ a fusion law with $\mathcal F\subseteq R$ and let $a$ be an idempotent in $V$.  Assume that the elements of $\mathcal F$ are pairwise distinguishable, $Ann_R(a)=\{0_R\}$, and $\ad_a$ is semisimple with spectrum $\mathcal F$. Then, 
\begin{enumerate}
\item $a$ is an $\mathcal F$-axis if and only if, for every  $\gamma, \delta\in \mathcal F$, and $v, w\in V$ 
\begin{equation}\label{cond2}
 (\prod_{\eta \in \gamma\star \delta}(\ad_a-\eta )) (f_\gamma(\ad_a)( v) \cdot f_\delta(\ad_a) (w))=0.
\end{equation}
%for every  $\gamma, \delta\in \mathcal F$, and $v, w\in V$;\\
\item if $a$ is an axis, then $a$ is primitive if and only if, for every  $w\in V$, there exists $l_w\in R$ such that
 \begin{equation}\label{cond1}
 f_1(\ad_a)(w-l_w a) =0.
 \end{equation}
In particular, if  $a$ is primitive, $l_w=\lambda_w(a)$.
\end{enumerate}
\end{lemma}
\begin{proof}
To prove $(1)$, note that, since $\ad_a$ is semisimple, for every $v,w\in V$, we can write
\begin{equation}\label{sost}
f_\gamma(\ad_a)( v) \cdot f_\delta(\ad_a) (w)=\sum_{\epsilon \in \gamma\star\delta}u_\epsilon+\sum_{\theta \in \mathcal F\setminus \gamma\star\delta}u_\theta,
\end{equation}
 with $u_\epsilon \in V_\epsilon^a$, $u_\theta \in V_\theta^a$. 
Then, 
\begin{eqnarray}\label{autovettori}
\lefteqn{\left (\prod_{\eta \in \gamma\star \delta}(\ad_a-\eta )\right )\left (f_\gamma(\ad_a)( v) \cdot f_\delta(\ad_a) (w)\right )=} \\
&& =\left (\prod_{\eta \in \gamma\star \delta}(\ad_a-\eta )\right )\left (\sum_{\epsilon \in \gamma\star\delta}u_\epsilon+\sum_{\theta \in \mathcal F\setminus \gamma\star\delta}u_\theta\right ) \nonumber\\
&&=\sum_{\epsilon \in \gamma\star\delta}\left (\prod_{\eta \in \gamma\star \delta}(\ad_a-\eta )u_\epsilon\right )+\sum_{\theta  \in \mathcal F\setminus \gamma\star\delta}\left (\prod_{\eta \in \gamma\star \delta}(\ad_a-\eta )u_\theta\right )\nonumber\\
&&=\sum_{\epsilon  \in \gamma\star\delta}\left (\prod_{\eta \in \gamma\star \delta}(\epsilon-\eta )\right )u_\epsilon +\sum_{\theta  \in \mathcal F\setminus \gamma\star\delta}\left (\prod_{\eta \in \gamma\star \delta}(\theta-\eta )\right )u_\theta  \nonumber \\
&&=\sum_{\theta  \in \mathcal F\setminus \gamma\star\delta}\left (\prod_{\eta \in \gamma\star \delta}(\theta-\eta )\right )u_\theta  \nonumber
\end{eqnarray}
where the last equality holds since as $\epsilon \in \gamma\star \delta$, every product $\prod_{\eta \in \gamma\star \delta}(\epsilon-\eta )$ is zero. Since the elements of $\mathcal F$ are pairwise distinguishable and $u_\theta$ are eigenvectors relative to distinct eigenvalues, it follows that Equation~(\ref{cond2}) holds if and only if 
$$f_\gamma(\ad_a)( v) \cdot f_\delta(\ad_a) (w)\in V_{\gamma \star \delta}.
$$ 
Thus, if $a$ is an $\mathcal F$-axis, then by Proposition~\ref{equivalence} and the fusion law, 
$$
f_\gamma(\ad_a)( v) \cdot f_\delta(\ad_a) (w)=f_\gamma(\gamma)v_\gamma\cdot  f_\delta(\delta)v_\delta\in V_{\gamma \star \delta},$$
and Equation~(\ref{sost}) holds. Conversely, assume Equation~(\ref{cond2}) holds and   let $v\in V_\gamma^a$ and $w\in V_\delta^a$. By Proposition~\ref{equivalence}, 
$$vw=\frac{1}{f_\gamma(\gamma)f_\delta(\delta)}f_\gamma(\ad_a)(v)f_\delta(\ad_a)(w)\in V_{\gamma \star \delta}.$$

For the second assertion, note that $l_wa\in V_1^a$, so, by Proposition~\ref{equivalence},
\begin{eqnarray*}
w_1&=&\pi_1^{\ad_a}(w)\\
&=&\pi_1^{\ad_a}(w-l_wa+l_wa)\\
&=&\pi_1^{\ad_a}(w-l_wa)+\pi_1^{\ad_a}(l_wa)\\
&=&\pi_1^{\ad_a}(w-l_wa)+l_wa\\
&=&\frac{1}{f_1(1)}f_1(\ad_a)(w-l_wa)+l_wa
\end{eqnarray*}
whence $a$ is primitive if and only if Equation~(\ref{cond1}) holds. The last part follows by the definition of $\lambda_a(w)$.
\end{proof}
Equations~(\ref{cond2}) and~(\ref{cond1}) will be used in the construction of the initial object in Section~\ref{3}.

%%%%%%%%%%%%%%%%%%%%%%%%%%%%%%%%%%
\section{Categories of primitive axial algebras}\label{3}

Let $k$ be a positive integer. Given a fusion law $(\mathcal F_0,  \ast_0)$,   we define the category $({\mathcal O},{\mathcal Mor})$ of primitive $k$ generated axial algebras with fusion law isomorphic to $(\mathcal F_0,  \ast_0)$ and some relevant subcategories of it.

Denote by 
${\mathcal O}$  the class whose objects are the primitive axial algebras 
$$
(R, V, A, (\mathcal F, \ast) )
$$
such that 
\begin{itemize}
\item [{\rm O1}]$A$ is a $k$-tuple of distinct free axes, 
\item [{\rm O2}]$(\mathcal F, \ast)$ is isomorphic to $(\mathcal F_0,  \ast_0)$, 
\item  [{\rm O3}] the elements of $\mathcal F$ are pairwise distinguishable in $R$.
\end{itemize}
To avoid confusion between the elements of $R$ and the elements of $V$, we shall always assume that $R\cap V=\emptyset$. This clearly implies no relevant loss of generality.
\noindent For two elements 
\begin{equation} \label{v1v2}
{\mathcal V_1}:=(R_1, V_1,A_1, (\mathcal F_1, \ast_1) ) \mbox{ and }{\mathcal V}_2:=(R_2, V_2, A_2, (\mathcal F_2, \ast_2 ))
\end{equation}
in ${\mathcal O}$,
let ${\mathcal Mor}({\mathcal V}_1,{\mathcal  V}_2)$ be the set of maps
$$\phi\colon R_1 \cup V_1 \to R_2\cup V_2$$ 
satisfying the following conditions:
\begin {enumerate}
\item  [{\rm H1}] $\phi_{|_{R_1}}$ is a homomorphism of rings with identity between $R_1$ and $R_2$  that induces by restriction an isomorphism of fusion laws between $({\mathcal F}_1,\ast_1)$ and $({\mathcal F}_2,\ast_2)$; 
\item  [{\rm H2}] $\phi_{|_{V_1}}$ is a ring homomorphism $V_1\to V_2$ such that, if $A_1=(c_1, \ldots , c_k)$ and $A_2=(c_1^\prime , \ldots , c_k^\prime)$, then for every $i\in \{1, \ldots , k\}$,  $c_i^\phi = c_i^\prime$ and \item  [{\rm H3}] $(\gamma v)^\phi=\gamma^\phi v^\phi$, for every $\gamma\in R_1$ and $v \in V_1$.
\end{enumerate}

Denote by $\mathcal Mor$ the class of all $\phi\in {\mathcal Mor}({\mathcal V_1},{\mathcal V_2})$ where ${\mathcal V_1}$ and ${\mathcal V_2}$ range in $\mathcal O$. Then clearly  $({\mathcal O},{\mathcal Mor})$ is a category.

If $A$ is the $k$-tuple $(c_1, \ldots , c_k)$, in the sequel we shall write $c\in A$ instead of $c\in \{c_1, \ldots , c_k\}$.

 \begin{remark}
 \label{labello}
Let $\phi \colon {\mathcal V_1}\to {\mathcal V_2}$ be a morphism in  $\mathcal Mor$.
\begin{enumerate}
\item Since $\phi_{|_{R_1}}$ is a ring homomorphism, the induced isomorphism of fusion laws  in (H1) is an algebraic isomorphism.
\item Note that $\phi_{|_{R_1}}$ induces on ${V_2}$  an $R_1$-algebra structure by setting, for every $\delta$ in $R_1$ and $w \in {V_2}$, $\delta w:=\delta^\phi w$. Giving ${V_2}$ such an $R_1$-algebra structure, conditions {\rm H2} and {\rm H3} are  equivalent to saying that $\phi_{|_{ V_1}}$ is an $R_1$-algebra homomorphism (such that  $A_{1}^\phi = A_2$).
\item Finally, if   ${\mathcal  V_1} = {\mathcal V_2}$, and $\phi$ is an automorphism, $\phi_{|_{V_1}}$ is not necessarily an  automorphism of $R_1$-algebras but a {\it semi-automorphism} of $R_1$-algebras (see~\cite{HRS}). Clearly,  $\phi_{|_{V_1}}$ is an automorphism of $R_1$-algebras if and only if $\phi_{|_{R_1}}$ is the identity map. 
\end{enumerate} 
\end{remark}
\medskip

We define some subcategories of $({\mathcal O},{\mathcal Mor})$ in the following way, in order to be able to impose some algebraic relations between the elements of the fusion law. In the following definition, the reader might find it useful to keep in mind that the $x_i$'s correspond to the elements of $\mathcal F_0\setminus \{0,1\}$, the $y_i$'s, $w_i$'s, and $z_{i,j}$'s are needed in order to make the elements corresponding to $0,1, x_1, \ldots , x_n$ pairwise distinguishable, and the $t_h$'s will be needed to guarantee the invertibility of certain elements in some particular cases. Let 
\begin{itemize}
\item[-] $n:=|\mathcal F_0\setminus \{0,1\}|$;
\item[-] $x_i, y_i, w_i, z_{i,j}, t_1, \ldots , t_h$ be $3n+\frac{n(n-1)}{2}+h$ algebraically independent indeterminates over $\Q$, for $i,j\in \{1, \ldots , n\}$, with $i<j$, and $h\in \N$; 
\item[-] $D$ be the polynomial ring 
$${\Z}[x_i, y_i, w_i, z_{i,j}, t_1, \ldots , t_h\:|\:i,j\in \{1, \ldots n\}, i<j, h\in \N];$$
\item[-] $L$ be a proper ideal of $D$ containing the set 
$$ \Sigma:=\{x_iy_{i}-1, \:\: (1-x_i)w_{i}-1, \mbox{ and } \:(x_i-x_j)z_{i,j}-1
 \mbox{ for all } 1\leq i<j\leq n \};$$  
\item [-] $\hat D:=D/L$.  
\end{itemize}
For $d\in D$, we denote the coset $L+d$ by $\hat d$. 

\begin{remark}\label{dimenticato}
Note that, with the above notation, if $F$ is a field and $0,1, d_1, \ldots , d_n$ are pairwise distinguishable elements of $F$, then there is a unique homomorphism 
$$\pi\colon D\to F[t_1, \ldots , t_h]$$ 
such that, for $i,j\in \{1, \ldots , n\}$, with $i<j$, $l\in \{1, \ldots , h\}$; 
$$
x_i^\pi=d_i, y_i^\pi=d_i^{-1}, w_i^\pi=(1-d_i)^{-1}, z_{i,j}^\pi=(d_i-d_j)^{-1}, t_l^\pi=t_l.
$$
As $\Sigma\subseteq \ker \pi\neq D$,  an ideal $L$ as in the above definition always exists.
Since $L$ is proper and, for $1\leq i<j\leq n$,  the elements $\hat x_i$, $1-\hat x_i$, and $\hat x_i-\hat x_j$ are invertible in $\hat D$,  the elements $1, 0, \hat x_1,\ldots, \hat x_n$ are pairwise distinguishable in ${\hat D}$. 
\end{remark}

Define $({\mathcal O}_{L},{\mathcal Mor}_{L})$ as the full subcategory of $({{\mathcal O},{\mathcal Mor}})$  whose objects are the primitive axial algebras $(R, V, A, (\mathcal F, \ast) )\in {\mathcal O}$ that 
satisfy the further condition
\begin{itemize}
\item [{\rm O4}] there exists a ring homomorphism $\zeta\colon \hat D\to R$
inducing a bijection between $\{1, 0,\hat x_1, \ldots , \hat x_n\}$ and $\mathcal F$.
\end{itemize}

\begin{remark}
If $(\Sigma)$ is the ideal of $D$ generated by the set $\Sigma$, then ${\mathcal O}_{(\Sigma)}={\mathcal O}$.
\end{remark}
\begin{proof}
By definition, ${\mathcal O}_{(\Sigma)}\subseteq {\mathcal O}$. Conversely, assume $(R, V, A, (\mathcal F, \ast) )\in {\mathcal O}$ and let $\mathcal F=\{0,1,s_1, \ldots , s_n\}$. Let $\zeta \colon D\to R$ be a ring homomorphism mapping, for every $i, j\in \{1, \ldots n\}$ with $i<j$,  $x_i$ to $s_i$, $y_i$ to $s_i^{-1}$,  $w_i$ to $(1-s_i)^{-1}$, and $z_{i,j}$ to $(s_i-s_j)^{-1}$.  Since the elements of $\mathcal F$ are pairwise distinguishable, $\zeta$ induces a bijection between $\{\hat x_1, \ldots , \hat x_n\}$ and $\mathcal F\setminus \{0,1\}$. Moreover, $(\Sigma)\subseteq \ker \zeta$ and $D/\ker \zeta\cong \zeta(D)\leq R$. So $\zeta(D)$ is a subring of $R$ isomorphic to a quotient of $D/(\Sigma)$, whence $(R, V, A, (\mathcal F, \ast) )\in {\mathcal O}_{(\Sigma)}$.
\end{proof}

\begin{lemma} \label{omolambda}
Let ${\mathcal  V_1},{\mathcal V_2} \in \mathcal O$ be as in Equation~(\ref{v1v2}), let $\phi \in {\mathcal Mor}({\mathcal V_1},{\mathcal V_2})$, and let $a\in A_1$. Then
\begin{enumerate}
\item for every $\mu\in \mathcal{F}_1$ and $v\in (V_1)^a_{\mu}$, we have $v^{\phi}\in (V_2)^{a^\phi}_{\mu^{\phi}}$.
\item for every $v\in V_1$, we have
$
\lambda_{a^\phi}(v^{\phi})=(\lambda_{a}(v))^{\phi}.
$
\end{enumerate}
\end{lemma}
\begin{proof}
Since ${\mathcal V_1},{\mathcal V_2} $ are primitive axial algebras, the result  follows applying $\phi$  to the decomposition of $v$ in Equation~(\ref{justapply}). 
\end{proof}

%%%%%%%%%%%%%%%%%%%%%%%%%%%%%%%%%%%%%%%%%%%%%%%
\section{Universal primitive axial algebras}\label{universal}

In this section we keep the notation of Section~\ref{3}. In particular we assume $L$ to be a proper ideal of $D$ containing the set $\Sigma$ and $\hat D=D/L$. We construct an initial object $\mathcal V_L=(\Ro,\Vo, \Amo, (\overline{\mathcal F}, \overline \star))$ for the category $(\mathcal O_L, {\mathcal Mor}_L) $ in several steps as follows. 

\subsection{Construction of $\aV$.} 

We start by constructing a universal non-associative algebra $\aV$ generated by a $k$-tuple of distinct idempotents over an extension ring $\aR$ of $\hat D$. Namely, let  $\hW$ be the free commutative magma with distinct generators $\ha_1,\ldots , \ha_k$, subject to the condition that $\ha_i$ is idempotent for every $i\in \{1, \ldots , k\}$. Set 
$\A:=(\ha_1,\ldots , \ha_k) $.  
Define $\aR:={\hat D}[\Lambda]$ to be the ring of polynomials with coefficients in $\hat D$ and  indeterminates set  
$$\Lambda:=\{\lambda_{ \ha,  \hw}\:|\: \ha \in \A,  \hw\in \hW, \ha\neq \hw\}
$$ where $\lambda_{ \ha,  \hw}=\lambda_{ \ha',  \hw'}
$ if and only if $ \ha= \ha'$ and $ \hw= \hw'$.
Finally, set $\aV:=\aR[\hW]$ to be the set of all formal linear combinations $\sum_{\hw\in \hW}\gamma_\hw \hw$ of the elements of $\hW$ with coefficients $\gamma_w$ in $\aR$ (with only finitely many of them different from zero) and endow $\aV$ with  the usual structure of a commutative non-associative $\aR$-algebra. Clearly the construction of $\hat V$ implies

\begin{remark}\label{uniV}{ \sc (The universal property of $\hat V$)} 
Let $V$ be a commutative $\hat R$-algebra and let $a_1^\prime, \ldots , a_k^\prime$ be $k$ distinct idempotents generating $V$ as an $\hat R$-algebra. Then, the map $\hat a_i\mapsto a_i^\prime$ extends to a unique $\hat R$-algebra epimorphism  from $\hat V$ onto $V$.
\end{remark}

\subsection{An axial quotient $\aV/\hJ$ of $\aV$}

Let $\So$ be the set $\{1,0, \hat x_1,\ldots, \hat x_n\}$ (defined in Section~3) and let $\star\colon  \So \times \So \to 2^{\So}$ be a map such that $(\So,\star)$ is a fusion law isomorphic to $(\mathcal F_0, \ast_0)$. Since, obviously, a fusion law is isomorphic to $(S_0,\ast_0)$ if and only if it is isomorphic to $(\So,\star)$, we may assume $(\mathcal F_0, \ast_0)= (\So,\star)$. By Lemma~\ref{condition}, in order to obtain from $\hat V$ an axial algebra with fusion law $(\So,\star)$, we need to factor out the ideal $J$ of $\hat V$ generated by the elements 
\begin{equation}\label{auto}
f_1(\ad_\ha)(\hw-\lambda_{\ha,\hw}\ha), 
\end{equation}
for all  $\ha\in \A $ and $\hw\in W$, where  $\lambda_{\ha,\ha}:=1$, and 
\begin{equation}\label{fusion}
\left (\prod_{\eta \in \gamma\star \delta}(\ad_\ha-\eta \:{\rm id}_\aV)\right )\left (f_\gamma(\ad_\ha)( \hv) \cdot f_\delta(\ad_\ha) (\hw)\right )
\end{equation}
for all $ \hv, \hw\in  \aV$, $\gamma, \delta\in \So$, $\ha\in \A$. Here, for $\mu\in \{\gamma, \delta\}$, $f_\mu$ is the polynomial  defined in Equation~(\ref{fmu}) of Section~\ref{2}, with $\Gamma=\hat{\mathcal F}$. 

For a subset $\hat U$ of $\hat V$ denote by $\hat U/\hJ$ the set $\{\hu+\hJ\:|\:\hu\in \hat U\}$. 

\begin{lemma}\label{wp}
 For every $\ha\in \A\setminus \hJ$, $\ha+\hJ$ is an $\hat{\mathcal F}$-axis for the algebra $\hat{V}/\hJ$ and for every $\hw\in \hW$, $\hw+\hJ$ decomposes as follows:
 \begin{equation}\label{decJ}
\hw+\hJ=(\lambda_{\ha,\hw}\ha+\sum_{\mu\in \hat{\mathcal F}\setminus \{1\}}\hw_\mu)+\hJ
\end{equation}
where $ \hw_\mu+\hJ$ is a $\mu$-eigenvector for $\ad_{\ha+\hJ}$.
\end{lemma}
\begin{proof}
 Let $\ha\in \A\setminus \hJ$. Since $\ha$ is an idempotent, $\ha+\hJ$ is an idempotent. Let $f:=\prod_{\lambda \in \mathcal F}(x-\lambda)$. By Equation~(\ref{auto}), for every $\hw\in \hat W$, we have (with the notation of Section~\ref{3})
$$
f_1(\ad_\ha)(\hw)\equiv \lambda_{\ha, \hw}\prod_{\eta\in \mathcal F\setminus \{1\}}(1-\eta)\ha \:\:(\bmod \:\hJ),
$$
whence
$$
f(\ad_\ha)(\hw)= (\ad_\ha-1)f_1(\ad_\ha)(\hw)\equiv \lambda_{\ha, \hw}\prod_{\eta\in \mathcal F\setminus \{1\}}(1-\eta)(\ad_\ha-1)(a)\equiv 0 \:\: (\bmod \:\hJ).
$$
Since $\hat V$ is linearly spanned by $\hat W$, $f(\ad_\ha)\hat V\subseteq \hJ$, whence by Proposition~\ref{equivalence}, $\ha+\hJ$ satisfies condition Ax1. Conditions Ax2, Ax3, and Equation~(\ref{decJ}) follow by Lemma~\ref{condition}.
\end{proof}

\subsection{Adjusting  $\aR$ and $\aV/\hJ$}
Assume 
$$
\sum_{\hw\in \hW} \gamma_\hw \hw\equiv 0 \:\:(\bmod \:\hJ).
$$   
Then, for every $\ha \in \A \setminus \hJ$, by Equation~(\ref{decJ}), we get 
\begin{equation}\label{wow}
0\equiv \left ((\sum_{\hw\in \hW} \gamma_\hw \lambda_{\ha,\hw})\ha+\sum_{\mu\in \hat{\mathcal F}\setminus \{1\}}\sum_{\hw\in \hW} \gamma_\hw \hw_\mu\right )\:\:(\bmod \:\hJ)
\end{equation}
where $\hw_\mu+\hJ$ is a $\mu$-eigenvector for $\ad_{\ha+\hJ}$. Since by Lemma~\ref{wp}, $\ha+\hJ$ is an $\hat{\mathcal F}$-axis in $\hat V/\hJ$, and the elements  of $\hat{\mathcal F}$ are pairwise distinguishable, by Lemma~\ref{equivalence} $\hat V/\hJ$ decomposes as the direct sum of $\ad_{\ha+\hJ}$-eigenspaces. Thus, by Equation~(\ref{wow}), it follows 
$$
(\sum_{\hw\in \hW} \gamma_\hw \lambda_{\ha,\hw})\ha\equiv 0 \:\:(\bmod\:\hJ).
$$
Since we want the generating axes of the target algebra $\Vo$ to be free, we need to replace $\aR$ by $\Ro:=\aR/I_0$, where 
$I_0$ is the ideal of $\aR$ generated by the set
\begin{equation}\label{I0}
\left \{\sum_{\hw\in \hW}\gamma_\hw \lambda_{\ha,\hw}\:|\:\ha\in \A,\: \gamma_\hw\in \hat R \mbox{ such that }\sum_{\hw\in \hW}\gamma_\hw \hw\in  \hJ \right \}.
\end{equation}

Set $\hJ_0:=\hJ+I_0\aV$, for every $\hv\in \hat V$, let $\bv:=\hv+\hJ_0$, and, for a subset $\hat U$ of $\hat V$, let ${\bf U}:=\{\bu\:|\:\hu\in \hat U\}$. Thus $\Vo=\aV/\hJ_0$. We now prove that $\hJ_0\neq \hat V$, whence $I_0\neq \hat R$. 

\begin{lemma}\label{contained}
Let ${\mathcal V}:=(R, V, A, (\mathcal F, \ast) )$ be an element of  $\mathcal O_L$, let $\phi\colon \aR\cup\aV  \to R\cup V$ be a map  that satisfies conditions (H1) (with respect to $(\So,\star)$ and $(\mathcal F, \ast)$), (H2), and (H3). Then $I_0\subseteq \ker \phi_{|_{\aR}}$,  $\hJ_0\subseteq  \ker \phi_{|_{\aV}}$. 
\end{lemma}
\begin{proof}
 By Remark~\ref{labello}(2), $V$ is an $\aR$-algebra and, since by Lemma~\ref{condition},  $\hJ\subseteq \ker \phi_{|_{\aV}}$,  
$\phi$ induces an $\aR$-algebra homomorphism 
$$\begin{array}{clcl}
\phi_{\aV/\hJ}:&\aV/\hJ&\to &V,
\end{array}$$
$$\:\:\:\:\:\:\:\:\:\:\:\:\:\begin{array}{clcl}
& \hv+\hJ&\mapsto &\hv^{\phi}.
\end{array}
$$
By condition (H2), $\A\cap \hJ=\emptyset$ and so by Lemma~\ref{wp},  for every $\ha\in \A$ and $\hw\in \hW$, we can write   
$$
\hw+\hJ=(\lambda_{\ha,\hw}\ha+\sum_{\mu \in \So\setminus \{1\}} \hw_\mu)+  \hJ,
$$
where, for every $\mu\in \So\setminus \{1\}$, $\hw_\mu+\hJ$ is a $\mu$-eigenvector for $\ad_{\ha+\hJ}$. 
 Condition (H3) implies that, for every $\mu\in \mathcal F$, $\ha\in \A$, $\phi_{\aV/\hJ}$ maps $\mu$-eigenvectors for $\ad_{\ha+\hJ}$ to  $\mu^{\phi}$-eigenvectors for $\ad_{\ha^{\phi}}$. Thus, if $\hv=\sum_{\hw\in \hW}\gamma_\hw \hw\in \hJ$, then $\hv^{\phi}=0$, whence
\begin{eqnarray*}
\lefteqn{0=\lambda_{\ha^{\phi}}(\hv^{\phi})}\\
&=&\lambda_{\ha^{\phi}}\left (\sum_{\hw\in \hW}\gamma_\hw^{\phi}\hw^{\phi}\right )\\
&=& \sum_{\hw\in \hW}\gamma_\hw^{\phi} \lambda_{\ha^{\phi}}(\hw^{\phi})\\
&=&  \sum_{\hw\in \hW}\gamma_\hw^{\phi} \lambda_{\ha^{\phi}}((\lambda_{\ha,\hw}\ha+\sum_{\mu \in \So\setminus \{1\}} \hw_\mu)^{\phi})\\
&=&   \sum_{\hw\in \hW}\gamma_\hw^{\phi} \lambda_{\ha^{\phi}}((\lambda_{\ha,\hw})^{\phi}\ha^{\phi})+\sum_{\hw\in \hW}\gamma_\hw^{\phi} \lambda_{\ha^{\phi}}(\sum_{\mu \in \So\setminus \{1\}} \hw_\mu^{\phi}))\\
&=&   \sum_{\hw\in \hW}\gamma_\hw^{\phi}(\lambda_{\ha,\hw})^{\phi}=(\sum_{\hw\in \hW}\gamma_\hw\lambda_{\ha,\hw})^{\phi}.
\end{eqnarray*}
Thus $I_0\subseteq \ker \phi_{|_{\aR}}$. Finally, by condition (H3), $(I_0\aV)^{\phi}=I_0^{\phi}V^{\phi}=0V_1=0$, whence $\hJ_0\subseteq \ker \phi_{|_{\aV}}$.
\end{proof}

\begin{lemma} \label{example}
We have $\hJ_0\neq \aV$, $I_0\neq \hat R$, and the map $\ha_i\mapsto \ba_i$ is a bijection from $\A$ onto $\Amo $. 
\end{lemma}
\begin{proof}
Let $\hat M$ be a maximal ideal of $\aR$ and  let $\F:=\aR/\hat M$. 
Since, by Remark~\ref{dimenticato}, the elements of $\hat{\mathcal F}$ are pairwise distinguishable, by Remark~\ref{mor}, $(\hat{\mathcal F}_\F, \star_{\F} )$ is a fusion law isomorphic to $(\hat{\mathcal F}, \star )$. 
%by Remark~\ref{homd}, the elements of $\hat{\mathcal F}_\F$ are pairwise distinguishable, thus the canonical projection $\hat R\to \hat R/\hat M$ induces a bijection $\varphi\colon \hat{\mathcal F}\to \hat{\mathcal F}_\F$. 
%Let  
%$$
%\star_\F :\hat{\mathcal F}_\F \times \hat{\mathcal F}_\F \to 2^{\hat{\mathcal F}_\F}
%$$  
%be the map defined, for every $\mu$ and $ \delta$ in $\hat{\mathcal F}$, by 
%\begin{equation}\label{asti}
% ( \mu+{\hat M}) \star_{\F} (\delta+{\hat M})=\{ \lambda+{\hat M}\:|\:\lambda \in \mu\star \delta \}
%\end{equation}
%Since $\varphi$ is a bijection, $\star_{\F}$ is well defined and $\varphi$ induces an algebraic isomorphism between $ (\mathcal F, \star)$ and $(\hat{\mathcal F}_\F, \star_{\F} )$.

Now let $\F^k$ be the direct sum of $k$ copies of $\F$ as an $\F$-algebra and $\mathcal B:=(e_1,\ldots ,e_k)$ be the canonical basis of $\F^k$. Then, for every $i\in \{1, \ldots , k\}$, $e_i$ is an idempotent and 
$$
\F^k=\F e_i\oplus \ker \ad_{e_i}.
$$
Therefore, by Lemma~\ref{any}, the $4$-tuple $$( \F, \F^k, \mathcal B, (\hat{\mathcal F}_\F, \star_{\F} ))$$ is a primitive (associative) axial algebra. Moreover, the axes $e_1, \ldots , e_k$ are free, since $\F$ is a field. Whence 
$$
( \F, \F^k, \mathcal B, (\hat{\mathcal F}_\F, \star_{\F} ))\in \mathcal O_L.
$$
By Remark~\ref{uniV},  the map sending each $\ha_i\in \A$ to $e_i\in \mathcal B$  extends to a unique $\hat R$-algebra epimorphism $\phi_{\aV}\colon \aV \to \F^k$.  Let 
$$
\phi\colon \aR \cup \aV \to \F \cup \F^k
$$ 
be the map whose restrictions to $\aR$ and $\aV$ are the canonical projection on $\F$ and $\phi_{\aV}$, respectively. Then $\phi$   satisfies the conditions H1, H2, and H3. Therefore, by Lemma~\ref{contained}, $\hJ_0\subseteq \ker \phi_{|_{\aV}}\neq \aV$ and $I_0\subseteq \hat M\neq \hat R$. Finally, since $k=|\{\ha_1, \ldots , \ha_k\}|\geq |\{\ba_1, \ldots , \ba_k\}|\geq |\{e_1, \ldots , e_k\}|=k$, we get the result.
\end{proof}
\medskip

For every $r\in \aR$, set $\bar r:=r+I_0$ and, for a subset $S$ of $\hat R$, let ${\overline S}:=\{\bar s\:|\:s\in S\}$. By Lemma~\ref{example}, $I_0\neq \hat R$ and thus, by Remark~\ref{mor}, $(\hat{\mathcal F}_{\Ro}, \star_{\Ro})$ is a fusion law isomorphic to $(\hat{\mathcal F}, \star)$.   We 
 denote $(\hat{\mathcal F}_{\Ro}, \star_{\Ro})$ by $(\overline{\mathcal{F}}, \overline{\star})$. Since, by hypothesis, $(\hat{\mathcal F}, \star)$ is isomorphic to $(\mathcal F_0, \ast_0)$, we have

\begin{cor}
\label{rem2}
 $(\overline{\mathcal F}, \overline \star)$ is a fusion law isomorphic to $(\mathcal F_0, \ast_0)$ and the elements of $\overline{\mathcal F}$ are pairwise distinguishable. 
\end{cor}

\begin{lemma}\label{wpbis}
 $(\aR, \Vo, \Amo, (\So, \star))$ is a  primitive axial algebra generated by $\Amo$ and, for $\ha \in \A$, $\hw \in \hW$, $\bw$ decomposes as 
 \begin{equation}\label{eqwpbis}
\bw=\lambda_{\ha,\hw}\ba+\sum_{\mu\in \hat{\mathcal F}\setminus \{1\}}\bw_\mu
\end{equation}
where $ \bw_\mu$ is a $\mu$-eigenvector for $\ad_{\ba}$.
\end{lemma}
\begin{proof}
By Lemma~\ref{example},  $\Vo\neq \{0\}$. Since $\hat V$ is generated as algebra by $\hat a_1, \ldots , \hat a_k$, $\Vo$ is generated as algebra by $\ba_1, \ldots , \ba_k$ and the result follows by Lemma~\ref{wp}, since $\hJ_0$ is an ($\aR$-invariant) ideal containing $\hJ$.
\end{proof}

\subsection{The initial object}
Set 
$$
\mathcal V_L:=(\Ro,\Vo, \Amo, (\overline{\mathcal F}, \overline \star)),
$$
where $\Amo$ is endowed with the order induced by the order of $\A$ by the map $\ha_i\mapsto \ba_i$.

\begin{lemma}\label{free1}
$\mathcal V_L\in \mathcal O_L$. Moreover, for every $\ba \in \Amo$, if $\lambda_\ba:\Vo\to \Ro$ is the $\Ro$-linear map defined in Lemma~\ref{lambdafunction}, then for every $\hat w\in \hat W$, we have
\begin{equation}\label{lambdone}
\lambda_\ba(\bw)=\lambda_{\hat a, \hat w}+I_0.
\end{equation}
\end{lemma}
\begin{proof}
Since, by Lemma~\ref{wpbis}, $(\aR, \Vo, \Amo, (\hat {\mathcal F},\star))$ is a  primitive axial algebra generated by the ordered set of idempotents $\Amo$, and by definition $I_0\subseteq Ann_\aR(\Vo)$, we get that $(\Ro,\Vo, \Amo, (\overline {\mathcal F}, \overline \star))$ is a primitive axial algebra generated by $\Amo$. Moreover, by Corollary~\ref{rem2}, the elements of $\overline{\mathcal F}$ are pairwise distinguishable in $\Ro$. Thus the result follows once we show that the axes in $\Amo$ are free. Suppose that there exist $r\in \aR$ and $\ha \in \A$ such that $\bar r\in Ann_\Ro(\ba)$. Then $r\ha\in \hJ_0$ and hence there exist $i_0\in I_0$ and $\sum_{\hw\in \hW}\gamma_\hw \hw\in  \aV$ such that
$$
r\ha-i_0\left(\sum_{\hw\in \hW}\gamma_\hw \hw\right )\in \hJ.
$$
Then, by the definition of $I_0$,
$$
r\lambda_{\ha,\ha}-i_0\left(\sum_{\hw\in \hW}\gamma_\hw \lambda_{\ha,\hw}\right )\in I_0.
$$
Since $\lambda_{\ha,\ha}=1$ and $I_0$ is an ideal, it follows that $r \in I_0$, thus $\bar r=0_\Ro$ and $\ba$ is free.

The last assertion follows by Lemma~\ref{lambdafunction} and Lemma~\ref{wpbis}, taking the image of Equation~(\ref{eqwpbis}) under the canonical projection $\hat R\to \Ro$.
\end{proof}

\begin{theorem}\label{free}
$\mathcal V_L$ is an initial object in the category $({\mathcal O}_L,{\mathcal Mor}_L)$.
\end{theorem}

\begin{proof}
Assume  ${\mathcal  V}:=(R, V, A, (\mathcal F, \ast))$ is an object in $\mathcal O_L$. By condition O4 of the definition of $\mathcal O_L$, there is a unitary ring homomorphism  $\zeta:\hat D\to R$  inducing a bijection between $\{\hat x_1,  \ldots , \hat x_n\}$ and $\mathcal F\setminus \{0,1\}$. Let $A:=(b_1, \ldots , b_k)$ and let $t$ be the map $\ha_i\mapsto b_i$ from $\A$ to $A$ and 
$
\bar t$ be the map $\ba_i\mapsto b_i$ from $\Amo$ to $\A$. Since $\hW$ is the free commutative magma generated by $\A$, there is a unique magma homomorphism 
$$\chi:\hW\to V,$$ extending the map $t$. Since the elements of $\Lambda$ are algebraically independent over $\hat D$, there is a unique homomorphism of $\hat D$-algebras 
 $$\hat \psi :\aR\to R$$ extending $\zeta$ and 
  such that,  for $i\in \{1,\ldots , k\}$ and  $\hw \in \hW\setminus \{\ha_i\}$,  
  \begin{equation}\label{lambdax}
 \lambda_{\ha_i, \hw}^{\hat \psi}=\lambda_{b_i}(\hw^\chi),
 \end{equation}
  where $\lambda_{b_i}$ is the function defined in Lemma~\ref{lambdafunction}.  Define 
$$
\begin{array}{rcccc}
\hat \chi &:& \aV& \to &V\\
& & \sum_{\hw\in \hW}\gamma_\hw \hw & \mapsto & \sum_{\hw\in \hW}\gamma_\hw^{\hat \psi}\: \hw^\chi
\end{array}.
$$
Then $\hat \chi$ is a ring homomorphism extending $t$. Moreover,  for every $\gamma\in \aR$ and $\hv=\sum_{\hw\in \hW}\gamma_\hw \hw\in \aV$, we have 
\begin{eqnarray*}
(\gamma  \hv)^{\hat \chi}&=&(\gamma \sum_{\hw\in \hW} \gamma_\hw \hw)^{\hat\chi}=( \sum_{\hw\in \hW} \gamma \gamma_\hw \hw)^{\hat\chi}\\
&=&\sum_{\hw\in \hW}\gamma^{\hat \psi} \gamma_\hw^{\hat \psi} \hw^\chi=\gamma^{\hat \psi}(\sum_{\hw\in \hW} \gamma_\hw^{\hat \psi} \hw^\chi)=\gamma^{\hat \psi} v^{\hat \chi}.\\
\end{eqnarray*}
Since, for every $\ha_i \in \A$, $\hv\in \aV$, and $\gamma \in \hat {\mathcal F}$,
$$
 [ (\ad_{\ha_i}-\gamma\: {\rm id}_\aV) (\hv)]^{\hat \chi} =(\ad_{b_i}-\gamma^{\hat \psi}\: {\rm id}_{V}) (\hv^{\hat \chi}),
$$
by Lemma~\ref{condition}, it follows that $\hJ_0$ is contained in $\ker \hat \chi$. Hence $\hat \chi$ induces a ring homomorphism 
$$\begin{array}{rccc}
\phi_V : &\Vo& \to &V\\
& \hv+\hJ_0 & \mapsto & \hv^{\hat \chi}
\end{array}
$$
 extending $\overline t$. Moreover,  for every $\gamma\in \aR$ and ${\bv}\in \Vo$,  we have
$$
(\gamma \bv)^{\phi_{V}}=(\gamma \hv)^{\hat\chi}=\gamma^{\hat \psi} \hv^{\hat \chi}=\gamma^{\hat \psi} {\bv}^{\phi_{V}}. 
$$
From this equality, as in the proof of Lemma~\ref{contained}, we get that $I_0\subseteq \ker \hat \psi$. Then  $\hat \psi$ induces a well defined homomorphism of $\hat D$-algebras $\phi_{R}:\Ro\to R$, $r+I_0\mapsto r^{\hat\psi}$. It follows that  the map $\phi\colon \Ro\cup \Vo \to R\cup V$ such that $\phi_{|\Ro}=\phi_R$ and $\phi_{|\Vo}=\phi_V$ belongs to ${\mathcal Mor}({\mathcal V}_L, \mathcal V)$ . 

Since $\aR=\hat D[\Lambda]$, $\phi_R$ is completely determined by its values on the elements $\lambda_{\ha,\hw}+I_0$, with $\ha\in \A$ and $\hw\in \hW\setminus \{\ha\}$. Further, by Equation~(\ref{lambdax}), for every $\ha_i\in \A$, $\hw\in \hW\setminus \{\ha_i\}$, 
\begin{equation}\label{bendef}
(\lambda_{\ha_i,\hw}+I_0)^{\phi_R}=\lambda_{\ha_i,\hw}^{\hat\psi}=\lambda_{b_i}(\hw^\chi)=\lambda_{b_i}((\hw+\hJ_0)^{\phi_{V}}),
\end{equation}
whence  $\phi_R$ is completely determined by the images $(\hw+\hJ_0)^{\phi_{V}}$, with $\hw \in \hW$.  Since $\phi_{V}$ is a ring homomorphism extending $t$, such images are uniquely determined by $t$, whence the uniqueness of  $\phi$. 
\end{proof}

\begin{cor}\label{baruf}
Every  permutation $\sigma $ of the set $\{\ba_1,\dots , \ba_k\}$ extends to a unique $\Ro$-algebra semi-automorphism $\sigma^\ast$ of $\Vo$. 
\end{cor}
\begin{proof}
Let $\A^\sigma:=(\ha_{1^\sigma}, \ldots , \ha_{k^\sigma})$ and let ${\mathcal V}_L^\sigma:=(\Ro^\sigma,\Vo^\sigma, \Amo^\sigma, (\overline{\mathcal F}, \overline \star))$ be the axial algebra obtained by repeating the construction of $\mathcal V_L$ with $\A^\sigma$ in place of $\A$, with the convention that, for every object $X$ appearing in that construction, we denote by $X^\sigma$ the corresponding object in the construction of  ${\mathcal V}_L^\sigma$. In particular, 
 $\hW^\sigma=\hW$, hence $\hat R^\sigma=\hat R$ and $\hat V^\sigma=\hat V$. A direct check shows that $\hJ^\sigma=\hJ$ and $I_0^\sigma=I_0$, so 
$\Ro^\sigma=\Ro$ and $\Vo^\sigma=\Vo$.  By Theorem~\ref{free}, ${\mathcal V}_L^\sigma$ is also an initial object in $(\mathcal O_L, \mathcal Mor_L)$.  Hence $\sigma$ extends to an isomorphism $\sigma^\ast\in \mathcal Mor({\mathcal V}_L, {\mathcal V}_L^\sigma)$, so by Remark~\ref{labello}, ${\sigma^\ast}_{|_\Vo}$ is a  semi-automorphism of $\Vo$ with $(\lambda_{\ha,\hw})^{\sigma^\ast}=\lambda_{\ha^{\sigma^\ast}, \hw^{\sigma^\ast}}$, for every $\ha\in \A$ and $\hw\in \hW\setminus \{\ha\}$.
\end{proof}

%Note that, for a generic object $\bf V$, the above assertion is false (see e.g. the algebra $Q_2(\eta)$ constructed in  \cite[Section 5.3]{3A}). We say that ${\bf V}:=(R, V, \mathcal A, (\mathcal F, \ast)) \in \mathcal O_L$ is {\it symmetric} if every  permutation $\sigma $ of the set $\mathcal A$ extends to a unique $R$-algebra automorphism $g_\sigma$ of $V$.  

By the above construction, $\Ro$ and $\Vo$ are $\hat D$-modules. Let ${\mathcal V}:=(R, V, A, (\mathcal F, \ast)) \in {\mathcal O}_{L}$. 
Since ${\mathcal V}\in {\mathcal O}_{L}$, by O4 there is a ring homomorphism $\zeta: \hat D\to R$ inducing a bijection between $\{\hat x_1, \ldots , \hat x_n\}$ and $\mathcal F\setminus \{0,1\}$. Thus $\zeta$ induces on $R$ a $\hat D$-module  structure and so the tensor products $\Ro\otimes_{\hat D} R$ and $\Vo\otimes_{\Ro}(\Ro\otimes_{\hat D} R) $ are well defined.

\begin{cor}\label{tensor}
Let ${\mathcal V}:=(R, V, A, (\mathcal F, \ast)) \in {\mathcal O}_{L}$ with $A=(b_1, \ldots , b_n)$ and set $\tilde{R}:=\Ro\otimes_{\hat D}R$.  Then
\begin{enumerate}
\item  there exists a unique surjective ring homomorphism $\varphi_R: \tilde R \to R$ such that, for every $i,j\in \{1, \ldots , k\}$ and $r\in R$, we have
\begin{equation}\label{ciserve1}
((\lm_{\hat a_i, \hat a_j}+I_0)\otimes r)^{\varphi_R}=r\lm_{b_i}(b_j).
\end{equation}
\item   there exists a unique surjective $R$-algebra homomorphism $\varphi_V: \Vo\otimes_{\Ro}\tilde R \to V$ such that, for every $i \in \{1, \ldots , k\}$ and $\tilde r\in \tilde R$, we have
\begin{equation}\label{ciserve2}
(a_i\otimes \tilde r)^{\varphi_V}=\tilde{r}^{\varphi_R}b_i.
\end{equation} 
In particular 
\begin{equation}\label{ciserve3}
(\Vo\otimes_{\Ro}\tilde R )\ker(\varphi_R)= \Vo\otimes_{\Ro} \ker(\varphi_R)\leq \ker (\varphi_V).
\end{equation}
\end{enumerate}
\end{cor}

\begin{proof}
By Theorem~\ref{free} there exists a unique $\phi\in \mathcal Mor(\mathcal V_L, \mathcal V)$. Then $\phi_{|\Ro}:\Ro\to R$ is a ring homomorphism and $\phi_{|\Vo}:\Vo\to V$ is a ring homomorphism mapping $a_i$ to $b_i$ for every $i\in \{1, \ldots , k\}$. Let $\varphi_R: {\tilde R} \to R$ be the map  defined by $\bar r\otimes r\mapsto r\bar r^{\phi_{|\Ro}}$, for each $\bar r \in \Ro$, $r\in R$. Then $\varphi_R$  is the unique surjective ring homomorphism satisfying Equation~(\ref{ciserve1}), proving $(1)$. 

Note that $\tilde R$ has a natural structure of an $\Ro$-algebra, whence $\Vo\otimes_{\Ro}\tilde R$ is an $\tilde R$-algebra and the homomorphism $r\mapsto 1\otimes r$ of $R$ into $\tilde R$ induces a structure of $R$-algebra on $\Vo\otimes_{\Ro}\tilde R$. As above, the map $\varphi_V: \Vo\otimes_{\Ro}\tilde R \to V$ defined by $\bv\otimes \tilde r\mapsto \tilde r^{\varphi_R} \bv^{\phi_{|\Vo}}$, for each $\bv \in \Vo$, $\tilde r\in \tilde R$, is the unique surjective algebra homomorphism  satisfying Equation~(\ref{ciserve2}). Now $(2)$ follows, since $\ker (\varphi_R)$ is an ideal of $\tilde R$. 
\end{proof}

%\begin{remark}\label{furab}
%Note that, with the notation of Corollaries~\ref{baruf} and~\ref{tensor}, $f_\sigma\otimes id_{R} $ is an automorphism in ${\mathcal Mor}(\overline{{\bf V}}_L\otimes R,\overline{{\bf V}}_L\otimes R)$.\end{remark}

\bigskip
\noindent {\bf Question 1.}
{\it Can we define a variety of axial algebras corresponding to a given fusion law $(\mathcal F,\ast)$?}

\noindent {\bf Question 2.}
{\it  Is it true that any ideal $I$ of $\aR$ containing $I_0$ defines an axial algebra with free axes?}

%%%%%%%%%%%%%%%%%%%%%%%%%%%%%%%%%%%%%%
\section{$2$-generated primitive axial algebras of Monster type $(\alpha, \beta)$}\label{table}

%In this section we recall some elementary properties of primitive axial algebras of Monster type $(\alpha, \beta)$ and then focus on the $2$-generated case. 

In this section we keep the notation of Section~\ref{universal}. Assume $k=2$ and let $(\mathcal F_0, \ast_0)$ be the Monster fusion law $\Mab$ in Table~\ref{Ising}.
\begin{table}
$$ 
\begin{array}{|c||c|c|c|c|}
\hline
\star & 1 & 0 & \al & \bt\\
\hline
\hline
1 & 1 & & \al & \bt\\
\hline
0 & & 0 & \al & \bt\\
\hline
\al & \al & \al & 1,0 & \bt\\
\hline
\bt & \bt & \bt & \bt & 1,0,\al\\
\hline
\end{array}
$$
\caption{Fusion law $\Mab$} \label{Ising}
\end{table}
Recall that this means that the adjoint action of every axis has spectrum 
$\{1,0,\al,\bt\}$ and, for any two eigenvectors $u$ and $v$ with eigenvalues 
$\gm,\dl\in\{1,0,\al,\bt\}$, the product $uv$ is a sum of eigenvectors for 
the eigenvalues contained in $\gm\star\dl$. 

Let ${\mathcal V}=(V, R, A, (\mathcal F, \ast))\in {\mathcal O}_{L}$, where $A=(a_0, a_1)$.  
 Set ${\mathcal F}_+:=\{1,0, \alpha\}$ and ${\mathcal F}_-:=\{ \beta\}$.
Recall (see e.g.~\cite[Section~3]{KMS}) that, for every $a\in A$, the partition $\{{\mathcal F}_+, {\mathcal F}_-\}$ of ${\overline{\mathcal F}}$ induces a $\Z_2$-grading on ${\mathcal F}$ which, on turn, induces, for every $a\in A$, a $\Z_2$-grading $\{V_+^a, V_-^a\}$ on $V$ where $V_+^a:=V_1^a+V_0^a+V_{\alpha}^a$ and $V_-^a=V_{ \beta}^a$. It follows that the unique linear map $\tau_a$ from $ V$ to $V$ such that 
   $$
   v\mapsto 
\begin{cases}
v & \mbox{ if } v\in  V_+^a\\
-v &  \mbox{ if }  v\in  V_-^a
\end{cases} 
$$  
 is  an involutory automorphism of $V$ (see~\cite[Proposition 3.4]{HRS}). The map  $\tau_a$ is called the {\it Miyamoto involution} associated to the axis $a$. 
  
\begin{lemma}\cite[Lemma~2.1]{FMS2}\label{invariant}
Let $a$ and $b$ be axes of $V$. Then $\tau_a$ and $\tau_{b}$ fix $ab- \beta(a+b)$.
\end{lemma}
 
Let $\rho:=\tau_{a_0}\tau_{a_1}$. For $i\in \Z$, set $$a_{2i}:=a_0^{\rho^i} \:\:\mbox{  and } \:a_{2i+1}:=a_1^{\rho^i}.
$$
and 
$$\mathcal A:=\{a_i\:|\: i\in \mathbb Z\}.
$$
 Since $\rho$ is an automorphism of $V$, for every $j\in \Z$,  $a_j$ is an axis. We denote simply by $\tau_j$ the Miyamoto involution $\tau_{a_j}$ associated to $a_j$. Note that, for every $i,j\in \Z$,
\begin{equation}\label{obviousrelation}
a_i^{\tau_j}=a_{2j-i}
\end{equation}

\begin{lemma}\cite[Lemma~2.2]{FMS2}\label{invariances}
For every $m\in \N$,  and $ i,j\in \Z$ such that $i\equiv j \: \bmod m$, we have  
$$
a_ia_{i+m}-\beta (a_i+a_{i+m})=a_ja_{j+m}-\beta (a_j+a_{j+m}).
$$
\end{lemma}

Finally, for $m\in \N$ and $ i\in \{0, \ldots , n-1\}$, set
\begin{equation}\label{s}
s_{i,m}:=a_ia_{i+m}-\beta (a_i+a_{i+m}). 
\end{equation}
and 
$$
\mathcal S:=\{s_{i,n}\:|\: n\in \N, i\in  \{0, \ldots , n-1\}\}.$$
Note that the products of two elements in $\mathcal A$ lie in the linear span of $\mathcal A\cup \mathcal S$.

By Lemma~\ref{lambdafunction}, for every $a\in \mathcal A$ there is a linear map $\lambda_a\colon V\to R$ such that every $v\in V$ can be written as in Equation~(\ref{justapply}). 
For $i\in \N$,  define
\begin{equation}\label{lambdas}
    \lambda_i:=\lambda_{a_0}(a_i),
    \end{equation} 
    and let 
 \begin{equation}\label{ai}
 a_i=\lambda_i  a_0+u_i+v_i +w_i
 \end{equation}
 be the decomposition of $a_i$ into $\ad_{a_0}$-eigenvectors, where $u_i$ is a $0$-eigenvector, $v_i$ is an $\alpha$-eigenvector and $w_i$ is a $\beta$-eigenvector.    

We focus now on the  initial object ${\mathcal V_L}=( \Ro, \Vo, (\ba_0, \ba_1), (\overline{\mathcal F}, \overline{\ast}))$ for the category $ {\mathcal O}_{L}$, where $L$ is an ideal of $D$ containing $2t_1-1$, so that $\hat 2$ is invertible in $\hat D$. 
%$(\overline{\mathcal F} \overline{\ast})$ is a fusion law isomorphic to $(\mathcal F_0, \ast_0)$. 
With an abuse of notation, for every isomorphism of fusion laws $\phi\colon (\mathcal F_0, \ast_0)\to (\mathcal F, \ast)$ we still denote by $\alpha$ and $\beta$ the images $\alpha^\phi$ and $\beta^\phi$, respectively. In paticular, $\overline x_1=\al$ and $\overline x_2=\bt$. 
According to the definitions given above, we have the set  
$$\{\ba_i\:|\: i\in \mathbb Z\}$$
of axes and the set 
$$\{\bs_{i,n}\:|\: n\in \N, i\in \{0, \ldots , n-1\}\}.$$
We will derive some linear relations between these elements and their products. It will turn out that, except for the case $\al\in \{2\bt, 4\bt\}$, $\Vo$ is linearly spanned by $$B:=\{ \ba_{-2}, \ba_{-1}, \ba_0, \ba_1, \ba_2, \bs_{0,1}, \bs_{0,2}, \bs_{1,2}\},$$ the above mentioned  relations giving the structure constants of the algebra. All the computations of this section have been accomplished with the use of Singular~\cite{Singular} in~\cite[genericsakuma.s]{code}. In Section~\ref{al=4bt} we will find similar relations in the case $\al=4\bt$.
 
Let $f$ be the isomorphism in $\mathcal Mor(\mathcal V_L, \mathcal V_L^\sigma)$  induced by the permutation $\sigma$ that swaps $\ba_0$ with $\ba_1$, as  defined in Corollary~\ref{baruf}. By Remark~\ref{labello}, $f_{|_\Vo}$ is a semi-automorphism of the $\Ro$-algebra $\Vo$ with associated ring automorphism $f_{|_\Ro}$.  For every $j\in \Z$, we have
\begin{equation}\label{obviousrelationf}
\ba_j^f=\ba_{-j+1}
\end{equation}

Further, by Lemma~\ref{omolambda}, for every $i\in \N$, we have 
\begin{equation}\label{nonsisa}
\lambda_{\ba_0}(\ba_{-i})=\lambda_i, \:\:
\lambda_{\ba_1}(\ba_0) =\lambda_1^f, \mbox{ and }\:\:
 \lambda_{\ba_1}(\ba_{-1})= \lambda_2^f.
 \end{equation}

 \begin{lemma}\label{primo}
Let $i\in \N$, then  
\begin{eqnarray*}
\ba_0\bs_{0,i}&=&( \alpha- \beta) \bs_{0,i}+[(1- \alpha)\lambda_i +\beta( \alpha- \beta- 1)]\ba_0+\frac{1}{2}\beta( \alpha - \beta)(\ba_i+\ba_{-i}).
\end{eqnarray*}
\begin{eqnarray*}
\ba_1\bs_{0,1}&=&( \alpha- \beta) \bs_{0,1}+[(1- \alpha)\lambda_1^f +\beta( \alpha- \beta- 1)]\ba_1+\frac{1}{2}\beta( \alpha - \beta)(\ba_0+\ba_{2}).
\end{eqnarray*}
and 
\begin{eqnarray*}
\ba_{-1}\bs_{0,1}&=&( \alpha- \beta) \bs_{0,1}+[(1- \alpha)\lambda_1^f +\beta( \alpha- \beta- 1)]\ba_{-1}+\frac{1}{2}\beta( \alpha - \beta)(\ba_0+\ba_{-2}).
\end{eqnarray*}
Moreover, for $i\in \N\setminus \{1\}$, 
\begin{eqnarray*}
\ba_1\bs_{1,i}&=&( \alpha- \beta) \bs_{1,i}+[(1- \alpha)\lambda_i^f +\beta( \alpha- \beta- 1)]\ba_1+\frac{1}{2}\beta( \alpha - \beta)(\ba_{-i+1}+\ba_{i+1}).
\end{eqnarray*}
\end{lemma}
\begin{proof}
The first formula is in~\cite[Lemma~3.1]{R}, the second is obtained by applying $f$  to the first one, the third is obtained by applying $\tau_0$ to the second one and the fourth is obtained by applying $f$ to the first one.
\end{proof}

 \begin{lemma}
\label{lambdaa}
With the above notation,   
\begin{enumerate}
\item $\bu_i=\frac{1}{\alpha} ((\lambda_i  - \beta - \alpha \lambda_i ) \ba_0 + \frac{1}{2}(\alpha - \beta) (\ba_i + \ba_{-i})-\bs_{0,i} )$;
\item $\bv_i=\frac{1}{\alpha} ((\bt-\lambda_i )\ba_0+\frac{\bt}{2}(\ba_i+\ba_{-i})+\bs_{0,i})$;
\item $\bw_i=\frac{1}{2}(\ba_i-\ba_{-i})$.
\end{enumerate}
\end{lemma}

\begin{proof}
(3) follows from the definitions of $\tau_0$ and $\ba_i$, 
(2) is just a rearranging of $\ba_0\ba_i=\lambda_i  \ba_0+\alpha \bv_i +\beta \bw_i$ using Equation~(\ref{s}), and (1) follows rearranging Equation~(\ref{ai}). \end{proof}

Set $T_0:=\langle \tau_0, \tau_1\rangle$ and $T:=\langle \tau_0, f\rangle$.

\begin{lemma}\label{action}
The groups $T_0$ and $T$ are dihedral groups, $T_0$ is a normal subgroup of $T$ such that $|T:T_0|\leq 2$. 
For every $n\in \N$, the set $\{\bs_{0, n}, \ldots , \bs_{n-1,n}\}$ is invariant under the action of $T$. In particular, if $K_n$ is the kernel of this action, we have 
\begin{enumerate}
\item $K_1=T$;
\item $K_2=T_0$, in particular $\bs_{0,2}^f=\bs_{1,2}$; 
\item $T/K_3$ induces the full permutation group on the set $\{\bs_{0, 3}, \bs_{1, 3}, \bs_{2,3}\}$ with point stabilisers generated by $\tau_0K_3$, $\tau_1K_3$ and $fK_3$, respectively. In particular $\bs_{0,3}^f=\bs_{1,3}$ and $\bs_{0,3}^{\tau_1}=\bs_{2,3}$. 
\end{enumerate}  
\end{lemma}
\begin{proof}
$T_0$ and $T$ are dihedral groups since they are both generated by pairs of involutions ($\tau_0$ and $\tau_1$, resp. $\tau_0$ and $f$). Since $f$ swaps by conjugation $\tau_0$ and $\tau_1$, $T_0$ is a normal subgroup of $T$ of index at most $2$. By the definition of $\bs_{i,n}$ and Equations~(\ref{obviousrelation}) and~(\ref{obviousrelationf}), it follows that the set $\{\bs_{0, n}, \ldots , \bs_{n-1,n}\}$ is invariant under the action of $T$. The remaining assertions follow immediately.
\end{proof}

For $i,j\in \N$, let $\bu_i, \bv_i, \bw_i$ be the $\ad_{\ba_0}$-eigenvectors defined as in Equation~(\ref{ai}) with respect to the axis $\ba_i$ and set
 \begin{equation}\label{pq}
 \bp_{ij}:=\bu_i\bu_j + \bu_i\bv_j\:\: \mbox{ and } \:\:\bq_{ij}:=\bu_i \bv_j+\frac{1}{\alpha^2}\bs_{0,i}\bs_{0,j}.
 \end{equation}

\begin{remark}\label{remarkproducts}
Note that, by Lemma~\ref{lambdaa}, it follows that the vectors $\bp_{ij}$  and $\bq_{ij}$ are linear combinations of products between two elements of $\mathcal A$ or between an element of $\mathcal A$ by an element of $\mathcal S$. The next lemma gives a formula which can be effectively used to express the products between two elements of $\mathcal S$ as linear combinations of products $\bf x\bf y$, with $(\bf x, \bf y)\in \mathcal A\times (\mathcal A\cup \mathcal S)$.%(\mathcal A\times \mathcal A)\cup (\mathcal A\times \mathcal S)$. 
\end{remark}
 
  \begin{lemma}
  \label{psi}
 For $i,j\in \N$ we have
 \begin{equation}\label{ss}
 \bs_{0,i} \bs_{0,j}=-\alpha (\ba_0\bp_{ij}- \alpha \bq_{ij}).
 \end{equation}
\end{lemma}
 \begin{proof}
Since $\bu_i$ is a $0$-eigenvector and $\bv_j$ is  an $\alpha$-eigenvector for $\ad_{\ba_0}$, 
 by the fusion law, we have $\ba_0\bp_{ij}=\alpha(\bu_i\cdot \bv_j)$ and the result follows.  
\end{proof}

\begin{lemma}
\label{a0s2f}
In the algebra $\Vo$, the following equalities hold:
\begin{eqnarray*}
\lefteqn{( \alpha-2   \beta)\ba_0\bs_{1,2}=\beta^2( \alpha- \beta)(\ba_{-2}+\ba_2)} \\
&+&\left [-2   \alpha   \beta   \lmu+2\beta(1- \alpha) \lmf+\frac{\beta}{2}(4   \alpha^2  -2   \alpha   \beta- \alpha   +4   \beta^2-2   \beta)\right ](\ba_1+\ba_{-1})\\
&+& \frac{1}{(\alpha-\beta)}\left [(6   \alpha^2-8   \alpha   \beta-2   \alpha+4   \beta)   \lmu^2+(2   \alpha^2-2   \alpha)   \lmu   \lmf \right .\\
&&+2(-2   \alpha^2-2   \alpha   \beta+ \alpha)(\alpha-\beta) \lmu-4\beta(   \alpha -1)(\alpha-\beta)   \lmf - \alpha   \beta(\alpha-\beta)   \lmd \\
&& \left .+(4   \alpha^2   \beta-2   \alpha   \beta+2   \beta^3)(\alpha-\beta)\right ]\ba_0\\
&+& \left [-4\alpha \lmu-4(\alpha-1)\lmf+(4\alpha^2-2   \alpha   \beta- \alpha+4   \beta^2-2   \beta)\right ]\bs_{0,1}\\
&+& 2\beta ( \alpha -  \beta)\bs_{0,2}
\end{eqnarray*}
and 
\begin{eqnarray*}
\lefteqn{4(\alpha-2\beta)\bs_{0,1} \cdot \bs_{0,1}=\beta(\alpha-\beta)^2(\alpha-4\beta)(\ba_{-2}+\ba_2)}\\
&+ &\left[4\alpha\beta(\alpha-\beta)\lmu+2(-\alpha^3+5\alpha^2\beta+\alpha^2-4\alpha\beta^2-5\alpha\beta+4\beta^2)\lmf \right .\\
& &\left .\beta(-10\alpha^2\beta-\alpha^2+14\alpha\beta^2+7\alpha\beta-4\beta^3-6\beta^2)\right ] (\ba_{-1}+\ba_1)\\
& +&2\left [2(-3\alpha^2+4 \alpha \beta+\alpha-2 \beta) \lmu^2+2\alpha(1-\alpha) \lmu \lmf \right .\\
& &\left . +2(\alpha^3+4 \alpha^2 \beta-6 \alpha \beta^2-3 \alpha \beta+4 \beta^2)\lmu+2\alpha\beta(\alpha -1) \lmf+\alpha\beta(\alpha- \beta)\lmd \right .\\
& &\left .+\beta(-\alpha^3 -8 \alpha^2 \beta+13 \alpha \beta^2+4 \alpha \beta-4 \beta^3-4 \beta^2)\right ] \ba_0\\
&+ &4\left [2   \alpha(\alpha-  \beta)   \lmu+\alpha( \alpha- 1)  \lmf+(-6   \alpha^2   \beta+10   \alpha   \beta^2+ \alpha   \beta-4   \beta^3)\right ]\bs_{0,1}\\
& -& 2\alpha\beta(\alpha -    \beta)\bs_{0,2}+2\beta( \alpha - \beta)(\alpha-2\beta) \bs_{1,2}.
\end{eqnarray*}

\end{lemma} 
\begin{proof}
 By the fusion law, 
 \begin{equation}
 \label{left2}
 \ba_0(\bu_1\cdot \bu_1 - \bv_1\cdot \bv_1 + \lambda_{\ba_0}(\bv_1\cdot \bv_1)\ba_0)=0.
 \end{equation} 
Substituting  in the left side of~(\ref{left2}) the values for $\bu_1$ and $\bv_1$ given in Lemma~\ref{lambdaa} we get the first equality. The expression for $(\al-2\bt)\ba_0\bs_{1,2}$ allows us to write explicitly  the vector  
$$
(\al-2\bt)(\ba_0\bp_{11}-\al \bq_{11})
$$ as a linear combination of $\ba_{-2}, \ba_{-1}, \ba_{0}, \ba_1, \ba_2, \bs_{0,1}, \bs_{0,2}, \bs_{1,2}$. The second equality now follows from  Equation~(\ref{ss}) in Lemma~\ref{psi}.
 \end{proof}

\begin{lemma}\label{a3} In the algebra $\Vo$, the following equalities hold:
\begin{enumerate} 
\item  $\beta(\alpha-\beta)^2(\alpha-4\beta)(\ba_3-\ba_{-2})=\bc$, 
\item  $\beta(\alpha-\beta)^2(\alpha-4\beta)(\ba_4-\ba_{-1})=-\bc^{\tau_1},$
\end{enumerate} where
\begin{eqnarray*}
%\lefteqn{\bc=}\\
\bc&=&
(\alpha-\beta)\left[ 4\alpha\beta\lmu-2(\alpha-1)(\alpha-4\beta)\lmf+\beta(-\alpha^2-5\alpha\beta-\alpha+6\beta)\right ]\ba_{-1}\\
&+& \left [4(-3\alpha^2+4\alpha\beta+\alpha-2\beta)\lmu^2-4\alpha(\alpha-1)\lmu\lmf \right .\\
&&
+(6\alpha^3+6\alpha^2\beta-2\alpha^2-16\alpha\beta^2-2\alpha\beta+8\beta^2)\lmu
+4\alpha\beta(\alpha-1)\lmf \\
&& \left . +2\alpha\beta(\alpha -\beta)
\lmd+\beta(\alpha-\beta)(-2\alpha^2-8\alpha\beta+\alpha+4\beta^2+2\beta)
\right ]  \ba_0\\
&+&\left [ 4\alpha(\alpha-1)\lmu\lmf+4(13\alpha^2-4\alpha\beta-\alpha+2\beta){\lmu^f}^2-4\alpha\beta(\beta-1)\lmu\right .\\
&&+(-6\alpha^3-6\alpha^2\beta+2\alpha^2+16\alpha\beta^2+2\alpha\beta-8\beta^2)\lmf-2\alpha\beta(\alpha-\beta)\lmdf\\
& & \left .+\beta(\alpha-\beta)(2\alpha^2+8\alpha\beta-\alpha-4\beta^2-2\beta)\right ]\ba_1\\
&+&(\alpha-\beta)\left [2(\alpha-1)(\alpha-4\beta)\lmu-4\alpha\beta\lmf+\beta(\alpha^2+5\alpha\beta+\alpha-6\beta)\right ]\ba_2\\
&+& 4\alpha(\alpha-2\beta+1)(\lmu-\lmf)\bs_{0,1}\\
&-& 4\beta (\alpha-\beta)^2(\bs_{0,2}-\bs_{1,2}). 
\end{eqnarray*}
\end{lemma}
\begin{proof}
Since $\bs_{0,1}\bs_{0,1}$ is invariant under $f$, we have $4(\al-2\bt)[\bs_{0,1}\bs_{0,1}- (\bs_{0,1}\bs_{0,1})^f]=0$. Then,  (1) follows from the expression for $4(\alpha-2\beta)\bs_{0,1} \bs_{0,1}$ given in Lemma~\ref{a0s2f}. By applying $\tau_1$ to (1), we get (2).
\end{proof}

From the formulae in Lemmas~\ref{a0s2f} and~\ref{a3}, it is clear that we have different pictures according whether $\al-2\bt$ and $\al-4\bt$ are invertible in $\Ro$ or not. 
Since we are most concerned with algebras over a field, later we will also assume that $\al-2\bt$ and $\al-4\bt$ are either invertible or zero. To this aim we choose appropriately the ideal $L$ of $D$ introduced in Section~\ref{3}, after Remark~\ref{labello}.
It follows that we have  to deal separately with the three cases of Rehren's Trichotomy: 
\begin{enumerate}
\item The {\it regular case}: here we assume that $L$ contains the elements 
$$\: \:2t_1-1, \:\:(x_1-2x_2)t_2-1,\:\:(x_1-4x_2)t_3-1,\: \mbox{ and } \:\:t_h \mbox{ for } 4\leq h\in \N,
$$ so that $\al-2\bt$ and $\al-4\bt$ are invertible in $\Ro$. In this case we denote $\mathcal O_L$ by $\mathcal O_r$. 
\item The {\it $\alpha=2\bt$ case}: here  $L$ contains the elements 
$$2t_1-1, \:\:x_1-2x_2 \: \mbox{ and } \:\:t_h \mbox{ for }  2\leq h\in \N,
$$
so that $\al=2\bt$.  In this case we denote  $\mathcal O_L$ by $\mathcal O_{2\bt}$. 
\item The {\it $\alpha=4\bt$ case}: here $L$ contains the elements 
$$2t_1-1, \:\:x_1-4x_2 \: \mbox{ and } \:\:t_h \mbox{ for }  2\leq h\in \N,
$$
so that $\al=4\bt$.  In this case we denote  $\mathcal O_L$ by $\mathcal O_{4\bt}$. 
\end{enumerate}

In~\cite{FMS2} the case $\al=2\bt$ is considered in details: in particular it is shown that every $2$-generated primitive axial algebra  in  $\mathcal O_{2\bt}$ is at most $8$ dimensional and this bound is the best possible.
The following result, which can be compared to Theorem~3.7 in~\cite{R}, is a  consequence of the {\it resurrection principle}~\cite[Lemma~1.7]{IPSS10}. 
\begin{proposition}\label{span}
Let $\mathcal V_r=(\Ro,\Vo, (\ba_0, \ba_1), (\overline{\mathcal F}, \overline \star))$  be the initial object in the category $\mathcal O_r$. Then
$\Vo$ is linearly spanned by the set $\{ \ba_{-2}, \ba_{-1}, \ba_0, \ba_1, \ba_2, \bs_{0,1}, \bs_{0,2}, \bs_{1,2}\}
$.
\end{proposition}
\begin{proof}
Let $\bf U$ be the  $\Ro$-linear span in $\Vo$ of the set  
$$B:=\{ \ba_{-2}, \ba_{-1}, \ba_0, \ba_1, \ba_2, \bs_{0,1}, \bs_{0,2}, \bs_{1,2}\}.$$
By Lemmas~\ref{a0s2f} and~\ref{a3}, since $\al-2\bt$ and $\al-4\bt$ are invertible in $\Ro$, we get that
$$\ba_0\cdot \bs_{1,2}\in \bf U\:\: \mbox{ and }\:\:  \ba_3\in \bf U.$$
By Equation~(\ref{obviousrelation}) and Lemma~\ref{action}, the set $B$ (hence $\bf U$) is invariant under the action of $\tau_0$. Moreover,  by Equation~(\ref{obviousrelationf}) and Lemma~\ref{action},  since $\ba_{-2}^f=\ba_3$, $\bf U$ is also invariant under $f$. Since $T=\langle \tau_0, f\rangle$, 
\begin{equation}\label{Uinvariant}
{\bf U} \mbox{ is } T\mbox{-invariant}.
\end{equation}
By Equation~(\ref{s}), $\bf U$ contains $\ba_0\ba_1$ and  $\ba_0\ba_2$. Hence, by taking the $T$-orbits of those elements, we get, by Equation~(\ref{Uinvariant}), that 
\begin{equation}
 \ba_i \ba_j\in {\bf U}, \mbox{ for every } i,j\in \Z.
\end{equation}
Similarly, since by Lemma~\ref{primo}, for $i\in \{1,2\}$, $\ba_0\bs_{0,i}\in \bf U$, and  by Lemma~\ref{a0s2f}, $\ba_0 \bs_{1,2}\in \bf U$, we get that 
\begin{equation}
\ba_j \bs_{0,i} \in {\bf U}\mbox{ and } \ba_j \bs_{1,2} \in {\bf U},  \mbox{ for every } j\in \Z \mbox{ and } i\in \{1,2\}.
\end{equation}
%$\bf U$ contains all the products 
%$\ba_j \bs_{0,i}$ and $\ba_j \bs_{1,2}$ for $j\in \Z, i\in \{1,2\}$.  
%
By Lemma~\ref{psi} and Remark~\ref{remarkproducts}, it follows that 
\begin{equation}\label{}
\bs_{0,i}\cdot \bs_{0,j}\in {\bf U}, \mbox{ for every }i,j\in \{1,2\}. 
\end{equation}
 As  $ \bs_{0,1}\cdot \bs_{1,2}=(\bs_{0,1}\cdot \bs_{0,2})^f$ and  $\bs_{1,2}\cdot \bs_{1,2}=(\bs_{0,2}\cdot \bs_{0,2})^f$, and  $\bf U$ is invariant under $f$,  we have also 
 \begin{equation}
 \{\bs_{0,1}\cdot \bs_{1,2},\:\:  \bs_{1,2}\cdot \bs_{1,2}\}\subseteq  \bf U.
 \end{equation}  
 Finally, from the identity $\ba_3\cdot \ba_3-\ba_3=0$, we get 
 $$ \bs_{0,2}\cdot \bs_{1,2} \in \bf U.
 $$ 
 Hence $\bf U$ is closed under the algebra multiplication, that is $\bf U$ is a subalgebra of $\Vo$. Since $\Vo$ is generated by $\ba_0$ and $\ba_1$, we get that $\bf U=\Vo$. \end{proof}

\begin{remark}\label{struct}
Note that the above proof gives a constructive way to compute the structure constants of the algebra $\Vo$ relative to the spanning set $B$. This has been accomplished with the use of Singular~\cite{Singular} in~\cite[genericsakuma.s]{code}.
\end{remark}

\begin{cor}\label{Rin}
Let $\mathcal V_r=(\Ro,\Vo, (\ba_0, \ba_1), (\overline{\mathcal F}, \overline \star))$  be the initial object in the category $\mathcal O_r$.
Then $\Ro$ is generated as a $\hat D$-algebra by $\lambda_1$, $\lambda_2$,  $\lambda_1^f$, and $\lambda_2^f$.\\
\end{cor}

\begin{proof}
By definition, $\Ro$ is generated a $\hat D$-algebra by the set
$$
\overline{\Lambda}:=\{ {\bar \lambda}_{\hat a_i, \hat w}\: | \hat w\in \hat W, i\in \{0,1\}\}.
$$
Thus we need to show that $\overline{\Lambda}$ lies in the $\hat D$-subalgebra $R^\ast$ of $\Ro$ generated by $\lambda_1$, $\lambda_2$,  $\lambda_1^f$, and $\lambda_2^f$.
Note first that, by Equation~(\ref{eqwpbis}) in Lemma~\ref{wpbis}, 
$$
{\bar \lambda}_{\hat a_i, \hat w}=\lambda_{\ba_i}(\bw).
$$
Therefore, since $\lambda_{\ba_i}$ is a linear function,   by Proposition~\ref{span}, we just need to show that, for every $\bv\in \{ \ba_{-2}, \ba_{-1}, \ba_0, \ba_1, \ba_2, \bs_{0,1}, \bs_{0,2}, \bs_{1,2}\}$,  and for every $i\in \{0,1\}$,
\begin{equation}\label{rast}
\lambda_{\ba_i}(\bv)\in R^\ast.
\end{equation} 
Moreover, since,  for every $\bv\in \Vo$, $\lambda_{\ba_1}(\bv)=(\lambda_{\ba_0}(\bv^f))^f$, and $\Ro=\Ro^f$, we are reduced to proving Equation~(\ref{rast}) for $i=0$.

By definition we have 
$$\lambda_{\ba_0}(\ba_0)=1, \:\:\lambda_{\ba_0}(\ba_1)=\lambda_1,\mbox{ and }\:\:\lambda_{\ba_0}(\ba_2)=\lambda_2.$$ Since $\tau_0$  is an $\Ro$-automorphism of $\Vo$ fixing $\ba_0$, we get
$$\lambda_{\ba_0}(\ba_{-1})=\lambda_{\ba_0}((\ba_{1})^{\tau_0})=\lambda_1,$$
$$\lambda_{\ba_0}(\ba_{-2})=\lambda_{\ba_0}((\ba_{2})^{\tau_0})=\lambda_2,$$
 and 
$$
\lambda_{\ba_0}(\bs_{0,1})=\lambda_{\ba_0}(\ba_0\ba_1-\beta \ba_0-\beta \ba_1)=\lambda_1-\beta-\beta \lambda_1\in R^\ast,
$$
$$
\lambda_{\ba_0}(\bs_{0,2})=\lambda_{\ba_0}(\ba_0\ba_2-\beta \ba_0-\beta \ba_2)=\lambda_2-\beta-\beta \lambda_2\in R^\ast.
$$
Finally, let $\bu_1$ and $\bv_1$ be as defined in Equation~(\ref{ai}). Substituting the expressions for $\bu_1$ and $\bv_1$ given in Lemma~\ref{primo} in the formula defining  $\bp_{1,1}$ (in Equation~(\ref{pq})), we obtain a vector $\bw$ in the $\Ro$-linear span of $\{ \ba_{-2}, \ba_{-1}, \ba_0, \ba_1, \ba_2, \bs_{0,1}, \}$ such that
\begin{equation}\label{asin2}
\bp_{1,1}=\bw+\frac{(\al-\bt)}{2\al}\bs_{1,2}.
\end{equation}
In particular, since $\lambda_{\ba_0}$ is linear, we get 
\begin{equation}\label{lambdaw}
\lambda_{\ba_0}(\bw)\in R^\ast.
\end{equation}
Since $\bu_1$, resp. $\bv_1$, is a $0$-eigenvector, resp. $\alpha$-eigenvector, for $\ad_{\ba_0}$, by the fusion law, $\bu_1\bu_1+\bu_1\bv_1$ is a sum of a $0$ and an $\al$-eigenvector for $\ad_{\ba_0}$, whence 
$$0=\lambda_{\ba_0}(\bp_{1,1})=
\lambda_{\ba_0}(\bw)+\frac{(\al-\bt)}{2\al}\lambda_{\ba_0}(\bs_{1,2}).
$$
Since $\al-\bt$ is invertible in $\Ro$, by Equation~(\ref{lambdaw}), we get  
$$
\lambda_{\ba_0}(\bs_{1,2})=-\frac{2\al}{\al-\bt}\lambda_{\ba_0}(\bw)\in R^\ast.
$$
\end{proof}

\begin{remark}\label{lambdaw2}
The explicit expression of $\lambda_{\ba_0}(\bs_{1,2})$ as a polynomial in $\lambda_1$, $\lambda_1^f$, $\lambda_2$, and $\lm_2^f$ can be computed as in Remark~\ref{struct}. We give it here for the convenience of the reader
\begin{equation}\label{lms12}
\lambda_{\ba_0}(\bs_{1,2})=\frac{2(\alpha-1)}{\alpha-\beta}\lambda_1^2-\frac{2(\alpha-1)}{\alpha-\beta}\lambda_1\lambda_1^f+(1-2\beta)\lambda_1+\beta\lambda_2-\beta.
\end{equation}
\end{remark}

\begin{lemma}\label{campo}
Let $\F$ be a field of characteristic other than $2$, $\al, \bt\in \F$ with $\al\not \in \{2\bt, 4\bt\}$. If $\mathcal V:=(V, \F, (a_0, a_1), \mathcal M(\al, \bt))$ is a primitive $2$-generated  axial algebra of Monster type $(\al, \bt)$, then  $\mathcal V\in \mathcal O_r$.
\end{lemma}
\begin{proof}
Condition O2 is obvious.
Since $\F$ is a field, $a_0$ and $a_1$ are free axes and $0,1,\al, \bt$ are pairwise distinguishable. In particular conditions O1 and O3 are satisfied with $k=2$. 
Since $\al\not \in \{2\bt, 4\bt\}$ and $\F$ has characteristic other than $2$, $\al-2\bt$, $\al-4\bt$, and $2$ are invertible in $\F$, therefore composing the map $\pi$ defined in Remark~\ref{dimenticato} (here $h=3$) with the evaluation $t_1\mapsto 2^{-1}$, $t_2\mapsto (\al-2\bt)^{-1}$, and $t_3\mapsto (\al-4\bt)^{-1}$, we get a ring homomorphism 
$$\eta\colon D\to \F$$
whose kernel contains the elements $2t_1-1$, $(\al-2\bt)t_2-1$, and $(\al-4\bt)t_3-1$. Thus condition O4 is also satisfied, hence $\mathcal V\in \mathcal O_r$.
\end{proof}
\bigskip

\begin{proposition}\label{reg}
Let $\F$ be a field of characteristic other than $2$, $\al, \bt \in \F$ with $\al\neq 4\bt$ and let $V$ be a 
 $2$-generated primitive axial algebra of Monster type $(\al, \bt)$ over $\F$. Then $V$ has dimension at most $8$.
\end{proposition}
\begin{proof}  If $\al\neq 2\bt$, by Lemma~\ref{campo}, $(V, \F, (a_0, a_1), \mathcal M(\al, \bt))\in \mathcal O_r$. Hence, by Corollary~\ref{tensor}, $V$ is isomorphic to a quotient of $\Vo\otimes_{\Ro } (\Ro\otimes_{\hat D}\F)$. Now the result follows since, by Proposition~\ref{span}, $\Vo$ is linearly spanned by the eight vectors $\ba_{-2}$, $\ba_{-1}$, $\ba_0$, $\ba_1$, $\ba_2$, $\bs_{0,1}$, $\bs_{0,2}$, and $\bs_{1,2}$. If $\al=2\bt$, then $(V, \F, (a_0, a_1), \mathcal M(\al, \bt))\in \mathcal O_{2\bt}$ and the result follows with a similar argument, using Proposition~3.7 in~\cite{FMS2} in place of Proposition~\ref{span}.
\end{proof}

%%%%%%%%%%%%%%%%%%%%%%%%%%%%%%%%%%%%%%%%%%%%%%%%%%%%%%%%%%%%
\section{ The case $\al=4\bt$}\label{al=4bt}
 
Let $\mathcal V_{4\bt}=(\Ro,\Vo, (\ba_0, \ba_1), (\overline{\mathcal F}, \overline \star))$  be the initial object in the category $\mathcal O_{4\bt}$. Note that, since we are assuming $\al=4\bt$ and $\al\neq \bt$,  the characteristic of $\Ro$ is other than $3$.
We start by finding relations in the universal algebra $\Vo$. All the computations of this section have been accomplished with the use of Singular~\cite{Singular} in~\cite[genericsakuma-4bt.s]{code}.

\begin{lemma}\label{lemma6.9}
In the algebra $\Vo$ the following equality holds
\begin{eqnarray}\label{eqs} % equazione rel; controllata
0&=&3\bt^2\left[8\lmu-(18\bt-1)\right ]\ba_{-1} \nonumber \\
&-&
\left[4(16\bt-1)\lmu^2+8(4\bt-1)\lmu\lmf-16\bt(13\bt-1)\lmu \right . \nonumber \\
&& \left .-8\bt(\bt-1)\lmf-12\bt^2\lmd+9\bt^2(10\bt-1)\right ]\ba_0\nonumber \\
&+& \left [8(4\bt-1)\lmu\lmf+4(16\bt-1)(\lmf)^2-8\bt(\bt-1)\lmu \right . \nonumber \\
&& \left .-16\bt(13\bt-1)\lmf-12\bt^2\lmdf+9\bt^2(10\bt-1)\right ]\ba_1 \nonumber \\
&-&3\bt^2\left [8\lmf-(18\bt-1)\right ]\ba_2 \nonumber \\
&+&8(2\bt+1)(\lmu-\lmf)\bs_{0,1} \nonumber \\
&-&18\bt^2(\bs_{0,2}-\bs_{1,2}). \nonumber 
\end{eqnarray}
\end{lemma}
\begin{proof}
This follows by Equation (1) in Lemma~\ref{a3} replacing $\al$ with $4\bt$ and dividing by $2\bt$.
\end{proof}

It will be convenient to have at hand also the next result, which is obtained by expliciting $\bs_{1,2}$ in the equation of Lemma~\ref{lemma6.9}.

\begin{cor}\label{eqs2f} % formula per S2f; controllata
In the algebra $\Vo$ the following equality holds
\begin{eqnarray*}
\bs_{1,2}&=&\bs_{0,2}-\frac{1}{6}\left[8\lmu-(18\bt-1)\right ]\ba_{-1}\\
&+&\frac{1}{18\bt^2}\left[4(16\bt-1)\lmu^2+8(4\bt-1)\lmu\lmf-16\bt(13\bt-1)\lmu \right . \nonumber \\
&& \left .-8\bt(\bt-1)\lmf-12\bt^2\lmd+9\bt^2(10\bt-1)\right ]\ba_0\nonumber \\
&-& \frac{1}{18\bt^2}\left [8(4\bt-1)\lmu\lmf+4(16\bt-1)(\lmf)^2-8\bt(\bt-1)\lmu \right . \nonumber \\
&& \left .-16\bt(13\bt-1)\lmf-12\bt^2\lmdf+9\bt^2(10\bt-1)\right ]\ba_1 \nonumber \\
&+&\frac{1}{6}\left [8\lmf-(18\bt-1)\right ]\ba_2 \nonumber \\
&-&\frac{4}{9\bt^2}(2\bt+1)(\lmu-\lmf)\bs_{0,1} \nonumber \\
\end{eqnarray*}
\end{cor}

\begin{lemma}\label{Rel} % controllata
In the algebra $\Vo$ the following equalities hold
\begin{eqnarray}\label{eqs2}
0&=&3\bt^2\left [8\lmf-(18\bt-1)\right ](\ba_{-2} -\ba_2)\nonumber \\
&+&4\left[2(4\bt-1)\lmu\lmf+(16\bt-1)(\lmf)^2-2\bt(4\bt-1)\lmu-4\bt(13\bt-1)\lmf  \right . \nonumber \\
&& \left .-3\bt^2\lmdf+3\bt^2(12\bt-1)\right ](\ba_1-\ba_{-1}).\nonumber 
\end{eqnarray}
and 
\begin{eqnarray}\label{eqs3}
0&=&3\bt^2\left [8\lmu-(18\bt-1)\right ](\ba_{3} -\ba_{-1})\nonumber \\
&+&4\left[2(4\bt-1)\lmu\lmf+(16\bt-1){\lmu}^2-2\bt(4\bt-1)\lmf-4\bt(13\bt-1)\lmu  \right . \nonumber \\
&& \left .-3\bt^2\lmd+3\bt^2(12\bt-1)\right ](\ba_0-\ba_{2}).\nonumber 
\end{eqnarray}
\end{lemma}
\begin{proof}
The first equation is obtained by subtracting from the equation in Lemma~\ref{lemma6.9} its image under $\tau_0$. The second is obtained by applying $f$ to the first.
\end{proof}

\begin{cor}\label{cor:a0s2f}
In the algebra $\Vo$ the following inclusion holds
$$\{ \ba_0 \bs_{1,2}, \ba_1\bs_{0,2}\} \subseteq  \langle \ba_{-2},  \ba_{-1},  \ba_{0},  \ba_{1},  \ba_{2},  \ba_{3},  \bs_{0,1}, \bs_{0,2}\rangle .
$$  
\end{cor}
\begin{proof}
By the first equation in Lemma~\ref{a0s2f}, 
$$\ba_0 \bs_{1,2}\in \langle \ba_{-2},  \ba_{-1},  \ba_{0},  \ba_{1},  \ba_{2},  \bs_{0,1}, \bs_{0,2}\rangle .
$$ Applying $f$ we get
$$\ba_1 \bs_{0,2}\in \langle \ba_{3},  \ba_{2},  \ba_{1},  \ba_{0},  \ba_{-1},  \bs_{0,1}, \bs_{1,2}\rangle 
$$ and the result follows  by Lemma~\ref{eqs2f}.
\end{proof}

\begin{lemma}\label{rel1} % controllata
In the algebra $\Vo$ the following equality holds
\begin{eqnarray*}
0&=&\frac{1}{3}\left [(4\bt-2)\lmu+(16\bt-1)\lmf-4\bt(9\bt-1)\right ]\ba_{-1}\\
&-&\frac{1}{18\bt^2}\left[4(32\bt^2-23\bt+3)\lmu^2+48\bt(4\bt-1)\lmu\lmf+2\bt(32\bt-2)(\lmf)^2 \right . \\
&& \left . -8\bt(82\bt^2-37\bt+3)\lmu-4\bt^2(106\bt-13)\lmf+6\bt^2(6\bt-1)\lmd \right . \\
&& \left . -12\bt^3\lmdf+9\bt^2(70\bt^2-17\bt+1)\right ]\ba_0\\
&-& \frac{1}{9\bt}\left [ 4(4\bt-1)\lmu \lmf+2(16\bt-1)(\lmf)^2-\bt(28\bt-10)\lmu-\bt(152\bt-11)\lmf \right . \\
&& \left .-6\bt^2\lmdf+18\bt^2(10\bt-1)\right ]\ba_1\\
&+&\frac{\bt}{6}\left [8\lmf-(18\bt-1)\right ]\ba_2\\
&-&\frac{1}{9\bt^2}\left  [ 4(4\bt-1)\lmu \lmf+2(16\bt-1)(\lmf)^2-8\bt(5\bt-2)\lmu-2\bt(100\bt-7)\lmf \right . \\
&& \left . -6\bt^2\lmdf+6\bt^2(48\bt-5) \right ]\bs_{0,1}\\
&+&\frac{1}{6}\left [ 8\lmf-(18\bt-1) \right ]\bs_{0,2}
\end{eqnarray*}
\end{lemma}
\begin{proof}
This is obtained by multiplying the expression for $\bs_{1,2}$ given in Lemma~\ref{eqs2f} by $\ba_0$ and subtracting the expression for $\ba_0\bs_{1,2}$ of Lemma~\ref{a0s2f}.
\end{proof}

\begin{lemma} \label{s1s2}
In the algebra $\Vo$ the following hold
\begin{enumerate}
\item $\lambda_{\ba_0}(\bs_{0,3})=\lm_3-\bt-\bt\lm_3$;
\item $\lambda_{\ba_0}(\bs_{1,3})=\lambda_{\ba_0}(\bs_{2,3})=-(2\bt-1)\lmu+\bt\lm_3-\bt$;
%\item $\lambda_{\ba_0}(\bs_{1,3})= -(2\bt-1)\lmu+\bt\lm_3-\bt$
\item $\{ \ba_0\bs_{1,3}, \ba_0\bs_{2,3}\}\subseteq \langle  \ba_{-3}, \ba_{-2},  \ba_{-1}, \ba_0, \ba_1, \ba_2, \ba_3, \bs_{0,1}, \bs_{0,2}, \bs_{0,3} \rangle$;
\item $\ba_1\bs_{0,3}\in \langle   \ba_{-2},  \ba_{-1}, \ba_0, \ba_1, \ba_2, \ba_3, \ba_4, \bs_{0,1}, \bs_{1,2}, \bs_{1,3} \rangle$;
\item 
\begin{eqnarray*} % controllata
\bs_{0,1}\bs_{0,2}&=&
\tfrac{3\bt^2}{8}\left [8(\lmu-\lmf)+2\bt-1\right ](\ba_{-2}+\ba_2)\\
&+&
\tfrac{1}{2}\left [2(4\bt-1)\lmu \lmf+(16\bt-1)(\lmf)^2+6\bt(10\bt-1)\lmu \right . \\
&&\left . +3\bt(4\bt-1)\lmf+6\bt^2\lmd-3\bt^2\lmdf -228\bt^3+21\bt^2\right ](\ba_{-1}+\ba_1)\\
&+&
\tfrac{1}{4\bt}\left [8(4\bt-1)\lmu^2\lmf+4(16\bt-1)\lmu(\lmf)^2-16\bt(21\bt-3)\lmu^2 \right . \\
&&\left . -4\bt(140\bt-17)\lmu\lmf-4\bt(16\bt-1)(\lmf)^2-12\bt(8\bt-1)\lmu\lmd \right . \\
&& \left . -24\bt^2\lmf \lmd-12\bt^2\lmu\lmdf+4\bt^2(423\bt-45)\lmu+4\bt^2(126\bt-15)\lmf \right . \\
&& \left . +\bt^2(198\bt-15)\lmd+12\bt^3\lmdf+12\bt^3\lmt-\bt^3(1338\bt-129)\right ] \ba_0\\
&+&
\tfrac{1}{2\bt}\left [-4(4\bt-1)\lmu\lmf-2(16\bt-1)(\lmf)^2+12\bt(10\bt-1)\lmu \right .\\
&& \left . +6\bt(36\bt-5)\lmf+12\bt^2\lmd+6\bt^2\lmdf-3\bt^2(163\bt-16)\right ]\bs_{0,1}\\
&+&
\tfrac{1}{4}\left [8(7\bt-1)\lmu+24\bt\lmf-174\bt^2+11\bt\right ] \bs_{0,2}\\
&+& 9\bt^2 \bs_{1,2}
-\frac{3\bt^2}{4}(4\bs_{0,3}-\bs_{1,3}-\bs_{2,3}).
\end{eqnarray*}
\end{enumerate}
\end{lemma}
\begin{proof}
As in the proof of Corollary~\ref{Rin}, we have
$$
\lambda_{\ba_0}(\bs_{0,3})=\lambda_{\ba_0}(\ba_0\ba_3-\bt \ba_0-\bt\ba_3)=\lm_3-\bt-\bt\lm_3,
$$
giving $(1)$.
To prove $(2)$, note first that $\bs_{1,3}-\bs_{2,3}$ is a $\bt$-eigenvector for $\ad_{\ba_0}$ since it is negated by $\tau_0$. Hence $\lambda_{\ba_0}(\bs_{1,3}-\bs_{2,3})=0$ by Lemma~\ref{lambdafunction}, whence  
\begin{equation}\label{lambdas13}
\lambda_{\ba_0}(\bs_{1,3})=\lambda_{\ba_0}(\bs_{2,3}).
\end{equation}

For $i\in \{1,2\}$, let $\bu_i$, $\bv_i$ be the $\ad_{\ba_0}$-eigenvectors defined by Equation~(\ref{ai}). Substituting the expressions for $\bu_i$ and $\bv_i$ given in Lemma~\ref{lambdaa} in the formula defining  $\bp_{1,2}$ (in Equation~(\ref{pq})), we obtain a vector $\bw$ in the $\Ro$-linear span of $\{\ba_{-3}, \ba_{-2}, \ba_{-1}, \ba_0, \ba_1, \ba_2, \ba_3, \bs_{0,1}, \bs_{0,2}, \bs_{1,2}, \bs_{0,3}\}$ such that
\begin{equation}\label{asin}
\bp_{1,2}=\bw+\frac{5}{32}(\bs_{1,3}+\bs_{2,3}).
\end{equation}
By the fusion law, $\bp_{1,2}=\bu_1\bu_2+\bu_1\bv_2\in \Vo_{0,\al}^{\ba_0}$, whence $\lambda_{\ba_0}(\bp_{1,2})=0$ again by Lemma~\ref{lambdafunction}. Thus
$$
\lambda_{\ba_0}(\bs_{1,3}+\bs_{2,3})=-\frac{32}{5}\lambda_{\ba_0}(\bw).
$$ 
Thus, by Equation~(\ref{lambdas13}), $\lambda_{\ba_0}(\bs_{1,3})=\frac{1}{2}\lambda_{\ba_0}(\bs_{1,3}+\bs_{2,3})=-\frac{16}{5}\lambda_{\ba_0}(\bw)$ and, computing the explicit expression of $\bw$, we get claim $(2)$ (see~\cite[genericsakuma-4bt.s]{code}).

Now, by the fusion law,
\begin{equation}\label{ulupo} 
\bu_1\bu_2\in \Vo_{0}^{\ba_0}.
\end{equation}
 Hence $\lambda_{\ba_0}( \bu_1\bu_2)=0$. Substituting  the values for $\bu_1$ and $\bu_2$ given in Lemma~\ref{lambdaa}, we get a vector $\bv\in \langle \ba_{-3}, \ba_{-2}, \ba_{-1}, \ba_0, \ba_1, \ba_2, \ba_3, \bs_{0,1}, \bs_{1,2}, \bs_{1,3}, \bs_{2,3} \rangle$ such that 
$$
\bu_1\bu_2=\bv+\frac{1}{16\bt^2}\bs_{0,1}\bs_{0,2}.
$$
 In particular 
\begin{equation}\label{lambdav}
\lambda_{\ba_0}(\bv)\in \hat D[\lmu, \lmf, \lmd, \lmdf, \lm_3]
\end{equation}
and, 
as above, we get
$$
\lambda_{\ba_0}(\bs_{0,1}\bs_{0,2})=-16\bt^2\lambda_{\ba_0}(\bv)\in \hat D[\lmu, \lmf, \lmd, \lmdf, \lm_3].
$$
Similarly, substituting  the values for $\bv_1$ and $\bv_2$ given in Lemma~\ref{lambdaa}, we get a vector 
$\overline{\bv}\in \langle \ba_{-3}, \ba_{-2}, \ba_{-1}, \ba_0, \ba_1, \ba_2, \ba_3, \bs_{0,1}, \bs_{1,2}, \bs_{1,3}, \bs_{2,3} \rangle$ such that
$$
\bv_1\bv_2=\overline{\bv}+\frac{1}{16\bt^2}\bs_{0,1}\bs_{0,2}.
$$
 Therefore, 
 $$\lambda_{\ba_0}(\bv_1\bv_2)=\lambda_{\ba_0}(\overline{\bv})+\frac{1}{16\bt^2}\lambda_{\ba_0}(\bs_{0,1}\bs_{0,2})\in \hat D[\lmu, \lmf, \lmd, \lmdf, \lm_3].
 $$

By the fusion law, $ \bv_1 \bv_2 - \lambda_{\ba_0}(\bv_1 \bv_2)\ba_0 \in \Vo_{0}^{\ba_0}$, whence by Equation~(\ref{ulupo}), 
 \begin{equation}
 \label{left12}
 \ba_0(\bu_1 \bu_2 - \bv_1 \bv_2 + \lambda_{\ba_0}(\bv_1 \bv_2)\ba_0)=0.
 \end{equation} 
 
Substituting  in the left side of~(\ref{left12}) the values for $\bu_1$, $\bu_2$, $\bv_1$, and $\bv_2$ given in Lemma~\ref{lambdaa} we get an expression of  $\ba_0(\bs_{1,3}+\bs_{2,3})$ as a linear combination of the vectors $\ba_{-3}, \ba_{-2}, \ba_{-1}, \ba_0, \ba_1, \ba_2, \ba_3, \bs_{0,1}, \bs_{1,2}, \bs_{1,3}, \bs_{2,3} $. Now $(3)$ follows, since $\ba_0(\bs_{1,3}-\bs_{2,3})=\bt(\bs_{1,3}-\bs_{2,3})$. Furthermore,
\begin{eqnarray*}
\ba_1\bs_{0,3}&=&(\ba_0\bs_{1,3})^f\in \langle  \ba_{-3}, \ba_{-2},  \ba_{-1}, \ba_0, \ba_1, \ba_2, \ba_3, \bs_{0,1}, \bs_{0,2}, \bs_{0,3} \rangle^f\\
&=&\langle  \ba_{4}, \ba_{3},  \ba_{2}, \ba_1, \ba_0, \ba_{-1}, \ba_{-2}, \bs_{0,1}, \bs_{1,2}, \bs_{1,3} \rangle .
\end{eqnarray*}

Finally, by (3) we can compute explicitly the vector $
(\al-2\bt)(\ba_0\bp_{1,2}-\al\bp_{1,2})
$
and $(4)$ follows by the formula in Lemma~\ref{psi} (see~\cite[genericsakuma-4bt.s]{code} for the explicit computation).
\end{proof}

\begin{lemma}\label{s2s2}
In the algebra $\Vo$ we have
$$
\bs_{0,2}\bs_{0,2}\in \langle \ba_{-4}, \ba_{-2} , \ba_{0}, \ba_{2} , \ba_4, \bs_{0,2}, \bs_{0,4}, \bs_{2,4}, \rangle .
$$
\end{lemma}
\begin{proof}
Let $\bu_2$ and $\bv_2$ be the $\ad_{a_0}$-eigenvectors defined in Equation~(\ref{ai}).

 By the fusion law, 
 \begin{equation}
 \label{left}
 \ba_0(\bu_2 \bu_2 - \bv_2 \bv_2 + \lambda_{\ba_0}(\bv_2 \bv_2)\ba_0)=0.
 \end{equation} 
 With an argument similar to that used in the proof of Lemma~\ref{s1s2} we first compute $\lambda_{\ba_0}(\bv_2 \bv_2)$. Then, 
substituting the values for $\bu_2$ and $\bv_2$ given in Lemma~\ref{lambdaa} in the left hand side of Equation~(\ref{left}),  we obtain a linear combination of the product  $\ba_0\bs_{2,4}$ and the vectors $\ba_{-4}, \ba_{-2} , \ba_{0}, \ba_{2} , \ba_4, \bs_{0,2}, \bs_{0,4}, \bs_{2,4}$. Since the coefficient of $\ba_0\bs_{2,4}$ is not zero, we can express $\ba_0\bs_{2,4}$ as a linear combination of $\ba_{-4}, \ba_{-2} , \ba_{0}, \ba_{2} , \ba_4, \bs_{0,2}, \bs_{0,4}, \bs_{2,4}$. This expression for $\ba_0\bs_{2,4}$ allows us to write explicitly  the vector  
$$
(\al-2\bt)(\ba_0\bp_{22}-\al \bq_{22})
$$ as a linear combination of $\ba_{-4}, \ba_{-2} , \ba_{0}, \ba_{2} , \ba_4, \bs_{0,2}, \bs_{0,4}, \bs_{2,4}$. Thus the claim follows from Lemma~\ref{psi}.
\end{proof}

\begin{lemma}\label{rel2} % controllata
In the algebra $\Vo$ the following relation holds
\begin{eqnarray*}
0&=&\tfrac{3\bt^2}{8}\left [8(\lmu-\lmf)+2\bt-1\right  ] (\ba_{-2}-\ba_4)\\
&+& \tfrac{1}{2}\left  [ 2(4\bt-1)\lmu\lmf+(16\bt-1)(\lmf)^2+6\bt(10\bt-1)\lmu+3\bt(4\bt-1)\lmf \right .\\
&& \left . +6\bt^2\lmd-3\bt^2\lmdf-228\bt^3+21\bt^2 \right  ] (\ba_{-1}-\ba_3)\\
&+&\tfrac{1}{8\bt}\left  [
16(4\bt-1)\lmu^2\lmf+8(16\bt-1)\lmu(\lmf)^2-32\bt(21\bt-3)\lmu^2 \right .\\
&& \left . -8\bt(140\bt-17)\lmu\lmf-8\bt(16\bt-1)(\lmf)^2-24\bt(8\bt-1)\lmu\lmd-48\bt^2\lmf\lmd \right . \\
&& \left . -24\bt^2\lmu\lmdf+120\bt^2(28\bt-3)\lmu+24\bt^2(43\bt-5)\lmf+6\bt^2(66\bt-5)\lmd \right . \\
&& \left . +24\bt^3\lmdf+24\bt^3\lmt-9\bt^3(298\bt-29)\right ](\ba_0-\ba_2)\\
&-&\tfrac{15\bt^2}{4}(\bs_{0,3}-\bs_{2,3})
\end{eqnarray*}
\end{lemma}
\begin{proof}
Using Lemma~\ref{s1s2} we compute $\bs_{0,1}\bs_{0,2}-(\bs_{0,1}\bs_{0,2})^{\tau_1}$, obtaining the right handside of the equation in the statement. On the other hand, by Lemma~\ref{action}, we have $\bs_{0,1}\bs_{0,2}-(\bs_{0,1}\bs_{0,2})^{\tau_1}=0$, thus the result follows.
\end{proof}

\begin{lemma}\label{rel3} % controllata
In the algebra $\Vo$ the following relation holds
\begin{eqnarray*}
0&=&
\tfrac{\bt^2}{8}\left [
8\lmu-24\lmf+42\bt-5\right ]\ba_{-2}\\
&+& \tfrac{1}{24}\left [
-72\lmu^2+8(52\bt-1)\lmu\lmf+12(16\bt-1)(\lmf)^2+8(71\bt^2+15\bt-1)\lmu \right .\\
&&\left .
-4(130\bt^2+15\bt-1)\lmf+48\bt^2\lmd
-36\bt^2\lmdf-\bt(2358\bt^2-159\bt-4)\right ] \ba_{-1}\\
&+&\tfrac{1}{36\bt^2}\left [
24(16\bt-1)\lmu^3-8(28\bt^2+13\bt+4)\lmu^2\lmf \right . \\
&& \left . +4(80\bt^2-25\bt+8)\lmu(\lmf)^2-2\bt(1464\bt^2+475\bt-52)\lmu^2 \right . \\
&& \left . -16\bt^2(226\bt-70)\lmu\lmf -16\bt(32\bt^2-7\bt+2)(\lmf)^2-36\bt^2(24\bt-1)\lmu\lmd \right .\\
&& \left . -24\bt^2(5\bt-2)\lmf\lmd-108\bt^3\lmu\lmdf+2\bt^2(7074\bt^2-173\bt-49)\lmu \right . \\
&& \left . +4\bt^2(720\bt^2-224\bt+8)\lmf+3\bt^3(588\bt-37)\lmd+36\bt^4\lmdf+108\bt^4\lmt \right . \\
&& \left .-9\bt^3(834\bt^2-63\bt-2)\right ]\ba_0\\
&+&\tfrac{1}{36\bt^2}\left [
36\bt(-16\bt+1)\lmu^2\lmf-4(8\bt^2+14\bt-4)\lmu(\lmf)^2\right .\\
&& \left . +4(128\bt^2-40\bt+2)(\lmf)^3+36\bt^2(16\bt-4)\lmu^2\right .\\
&& \left . +4\bt(1304\bt^2-96\bt-8)\lmu\lmf +2\bt(632\bt^2+123\bt-20)(\lmf)^2+108\bt^3\lmf\lmd \right . \\
&& \left . +216\bt^3\lmu\lmdf+12\bt^2(64\bt-7)\lmf\lmdf-4\bt^2(900\bt^2-152\bt-1)\lmu\right . \\
&& \left .-2\bt^2(6714\bt^2-317\bt-31)\lmf-36\bt^4\lmd-3\bt^2(588\bt^2-37\bt)\lmdf\right . \\
&& \left .-108\bt^4\lm_3^f+9\bt^3(834\bt^2-63\bt-2)\right ]\ba_1\\
&+&\tfrac{1}{24}\left [
-12(16\bt-1)\lmu^2-24(12\bt-3)\lmu\lmf-8(16\bt-1)(\lmf)^2\right . \\
&&\left . +4(58\bt^2-17\bt+1)\lmu-8\bt(35\bt-1)\lmf+36\bt^2\lmd-48\bt^2\lmdf\right . \\
&& \left . +\bt(2358\bt^2-159\bt-4)\right ]\ba_2\\
&+&\tfrac{\bt^2}{8}\left [24\lmu-8\lmf-42\bt+5 \right ]\ba_3\\
&+&\tfrac{1}{9\bt^2}\left [
9\bt(16\bt-7)\lmu^2+8(4\bt^2+4\bt+1)\lmu\lmf-(176\bt^2-31\bt+8)(\lmf)^2\right . \\
&& \left .-\bt(528\bt^2-197\bt+8)(\lmu-\lmf)+9\bt^3(\lmd-\lmdf )\right ]\bs_{0,1}\\
&+&\tfrac{1}{12}\left [4(50\bt-2)\lmu+8(5\bt-2)\lmf-3\bt(210\bt-11)\right ] \bs_{0,2}\\
&-&\tfrac{1}{12}\left [8(11\bt+1)\lmu+8(19\bt-4)\lmf-3\bt(210\bt-11)\right ]\bs_{1,2}\\
&-&\tfrac{15\bt^2}{4}(\bs_{0,3}-\bs_{1,3})
\end{eqnarray*}
\end{lemma}
\begin{proof}
The relation is obtained by computing the product $\bs_{0,1} \bs_{1,2}$ in two different ways and subtracting the two expressions. The first way is to multiply by $\bs_{0,1}$ the expression for $\bs_{1,2}$ as a linear combination of $\ba_{-2}$, $\ba_{-1}$, $\ba_0$, $\ba_1$, $\ba_2$, $\bs_{0,1}$, and $\bs_{0,2}$ given in Lemma~\ref{eqs2f}: here the products $\bs_{0,1}\ba_i$ are given in Lemma~\ref{primo}, the product $\bs_{0,1}\bs_{0,1}$ is given in Lemma~\ref{a0s2f}, and the product $\bs_{0,1} \bs_{0,2}$ is given in Lemma~\ref{s1s2}. The second way, since $\bs_{0,1}\bs_{1,2}=(\bs_{0,1}\bs_{0,2})^f$, is to apply $f$ to the 
expression for $\bs_{0,1}\bs_{0,2}$ in Lemma~\ref{s1s2}. 
\end{proof}

\begin{cor}\label{ch5}
Assume $5$ is invertible in $\Ro$. Then,
$$
\{\bs_{1,3}, \bs_{2,3}\} \subseteq \langle \ba_{-3}, \ba_{-2}, \ba_{-1}, \ba_{0}, \ba_{1}, \ba_{2}, \ba_{3}, \bs_{0,1}, \bs_{0,2}, \bs_{0,3}\rangle .
$$
\end{cor}
\begin{proof}
If $5$ is invertible in $\Ro$, by Lemma~\ref{rel3} we may express $\bs_{1,3}$ as a linear combination of $\ba_{-2}$, $\ba_{-1}$, $\ba_{0}$, $\ba_{1}$, $\ba_{2}$, $\ba_{3}$, $\bs_{0,1}$, $\bs_{0,2}$, $\bs_{1,2}$, $\bs_{0,3}$. As $\bs_{2,3}=(\bs_{1,3})^{\tau_0} $ and $\langle \ba_{-3}, \ba_{-2}, \ba_{-1}, \ba_{0}, \ba_{1}, \ba_{2}, \ba_{3}, \bs_{0,1}, \bs_{0,2}, \bs_{0,3}\rangle$ is invariant under $\tau_0$, we get the claim. 
\end{proof}

\begin{lemma}\label{rel23}
In the algebra $\Vo$ the following relation holds
\begin{eqnarray*}
0&=&\tfrac{\bt^2}{8}\left [-24\lmu+8\lmf+42\bt-5\right ] \ba_{-3} \\
&+&
\tfrac{1}{12}\left [6(16\bt-1)\lmu^2+36(4\bt-1)\lmu\lmf+4(16\bt-1)(\lmf)^2 -2(40\bt^2-17\bt+1)\lmu\right .\\
&& \left . +4\bt(26\bt-1)\lmf-18\bt^2\lmd+24\bt^2\lmdf -\bt(1170\bt^2-75\bt-2) \right ] \ba_{-2} \\
&+& \tfrac{1}{36\bt^2}\left [
36\bt(16\bt-1)\lmu^2\lmf+8(4\bt^2+7\bt-2)\lmu(\lmf)^2 \right . \\
&&-8(64\bt^2-20\bt+1)(\lmf)^3-144\bt^2(4\bt-1)\lmu^2-4\bt(1268\bt^2-87\bt-8)\lmu\lmf \\
&& -8\bt(122\bt^2+33\bt-5)(\lmf)^2-108\bt^3\lmf\lmd-216\bt^3\lmu\lmdf -12\bt^2(64\bt-7)\lmf\lmdf\\
&& +4\bt^2(1170\bt^2-179\bt-1)\lmu +2\bt^2(6822\bt^2-344\bt-31)\lmf +144\bt^4\lmd\\
&&  \left . +3\bt^3(570\bt-37)\lmdf +108\bt^4\lmt^f-9\bt^3(1290\bt^2-105\bt-2)\right ] \ba_{-1} \\
&+& \tfrac{1}{72\bt^2}\left [
-48(16\bt-1)\lmu^3+64(16\bt^2+\bt+1)\lmu^2\lmf+64(8\bt^2+2\bt-1)\lmu(\lmf)^2\right .\\
&& -4\bt(48\bt^2-691\bt+52)\lmu^2-8\bt^2(356\bt+127)\lmu\lmf \\
&& -8\bt(16\bt^2+19\bt-8)(\lmf)^2 
+144\bt^2\lmu\lmd-96\bt^2(2\bt+1)\lmf\lmd \\
&& +4\bt^2(486\bt^2-637\bt+49)\lmu 
 +8\bt^2(441\bt^2+89\bt-8)\lmf+12\bt^3(3\bt-4)\lmd\\
 &&\left .+144\bt^4\lmdf-9\bt^2(1014\bt^3-135\bt^2+4\bt) \right ] \ba_{0} \\
&+& \tfrac{1}{24}\left [
72\lmu^2-8(52\bt-1)\lmu\lmf-12(16\bt-1)(\lmf)^2-8(71\bt^2+15\bt-1)\lmu \right . \\
&&\left .+4(130\bt^2+15\bt-1)\lmf-48\bt^2\lmd+36\bt^2\lmdf+2358\bt^3-159\bt^2-4\bt \right ] \ba_{1} \\
&+& \tfrac{1}{4\bt}\left [
-8(4\bt-1)\lmu^2\lmf-4(16\bt-1)\lmu(\lmf)^2+48\bt(7\bt-1)\lmu^2 \right . \\
&& +4\bt(140\bt-17)\lmu\lmf+4\bt(16\bt-1)(\lmf)^2+12\bt(8\bt-1)\lmu\lmd+24\bt^2\lmf\lmd \\
&& +12\bt^2\lmu\lmdf -4\bt^2(421\bt-45)\lmu-4\bt^2(126\bt-15)\lmf-\bt^2(198\bt-15)\lmd \\
&& \left .-12\bt^3\lmdf-12\bt^3\lmt+8\bt^3(165\bt-16) \right ] \ba_{2} \\
&+& \tfrac{1}{2}\left [
-2(4\bt-1)\lmu\lmf-(16\bt-1)(\lmf)^2-6\bt(10\bt-1)\lmu-3\bt(4\bt-1)\lmf \right . \\
&& \left . -6\bt^2\lmd+3\bt^2\lmdf+3\bt^2(76\bt-7) \right ] \ba_{3} \\
&+& \tfrac{1}{8}\left [
-24\bt^2(\lmu-\lmf)-3\bt^2(2\bt-1) \right ] \ba_{4} \\
&+&\tfrac{1}{9\bt^2} \left [
-9\bt(16\bt-7)\lmu^2-8(2\bt+1)^2\lmu\lmf+(176\bt^2-31\bt+8)(\lmf)^2 \right . \\
&&
\left . +\bt(528\bt^2-197\bt+8)(\lmu-\lmf)-9\bt^3(\lmd-\lmdf) \right ] \bs_{0,1} \\
&+&
\tfrac{1}{12}\left [-8(25\bt-1)\lmu-8(5\bt-2)\lmf+3\bt(210\bt-11) \right ] \bs_{0,2} \\
&+&
\tfrac{1}{12}\left [8(11\bt+1)\lmu+8(19\bt-4)\lmf-3\bt(210\bt-11) \right ] \bs_{1,2} 
\end{eqnarray*}
\end{lemma}
\begin{proof}
The result follows by subtracting the image under $\tau_0$ of the equation in Lemma~\ref{rel3} from the equation in Lemma~\ref{rel2}.
\end{proof}

\begin{theorem}\label{newth}
Let $\F$ be a field of characteristic other than $2$ and $3$ and let $\bt\in \F\setminus \{ \frac{1}{2}\}$.
Let $V$ be a $2$-generated primitive axial algebra of Monster type $(4\bt, \bt)$ over $\F$ with generators $a_0, a_1$. Then 
\begin{enumerate}
\item if $\lambda_{a_0}(a_1)\neq \frac{18\bt-1}{8}\neq\lambda_{a_1}(a_0)$, then $V=\langle a_{-1}, a_0, a_1, a_2, s_{0,1}, s_{0,2} \rangle$;
\item if $\lambda_{a_0}(a_1)\neq \frac{18\bt-1}{8}=\lambda_{a_1}(a_0)$, then $V=\langle a_{-2}, a_{-1}, a_0, a_1, a_2, s_{0,1}, s_{0,2} \rangle$;
\item if $\lambda_{a_0}(a_1)= \frac{18\bt-1}{8}\neq \lambda_{a_1}(a_0)$, then $V=\langle a_{-3}, a_{-2}, a_{-1}, a_0, a_1, s_{0,1}, s_{1,2} \rangle$;
\item if $\lambda_{a_0}(a_1)= \lambda_{a_1}(a_0)=\frac{18\bt-1}{8}$, then $V=\langle a_{-2}, a_{-1}, a_0, a_1, a_2, a_3, s_{0,1}, s_{0,2}\rangle$.
\end{enumerate}
\end{theorem}
\begin{proof}
By Corollary~\ref{tensor}, there are a surjective ring homomorphism 
$$
\varphi_R\colon \tilde \F=\Ro\otimes _{\hat D}\F \to \F
$$ 
and a surjective algebra homomorphism  
$$
\varphi_V\colon \Vo\otimes_{\Ro} \tilde \F\to V
$$
mapping $\ba_i\otimes 1$ to $a_i$, for $i\in \Z$. We identify the image under $\varphi_R$ of an element of $\Ro\otimes_{\hat D} \F$ with the element itself while we denote by $v$ the element $(\bv\otimes 1)^{\varphi_V}$, for $\bv\in \Vo$. In particular, we write $\lambda_i$ in the place of $\lambda_{a_0}(a_i)$, for $i\in \Z$, $\lmf$ in place of $\lambda_{a_1}(a_0)$, $\lmdf$ in place of $\lambda_{a_1}(a_{-1})$, and  $\lmt^f$ in place of $\lambda_{a_1}(a_{-2})$\footnote{Recall that in general $f$ is not a homorphism of $\F$, so that here $\lambda_i^f$ is just a notation and does not necessarily denote the image of $\lambda_i$ under $f$.}. When we apply $\varphi_V$ to the relations obtained so far in this section  we get similar relations in $V$, where  $\ba_{i}$ are replaced by $a_{i}$ for $i\in \Z$ and $\bs_{i,n}$ are replaced by $s_{i,n}$, respectively.
\medskip

{\bf Claim 1.} {\it If $\lmu\neq \frac{18\bt-1}{8}$ and  $\lmf\neq \frac{18\bt-1}{8}$, then $V=\langle a_{-1}, a_0, a_1, a_2, s_{0,1}, s_{0,2} \rangle$.}
\medskip

Let $U:=\langle a_{-1}, a_0, a_1, a_2, s_{0,1}, s_{0,2} \rangle$. 
By hypothesis, the coefficient of $a_{-2}-a_2$  in the first equation (resp. $a_{3}-a_{-1}$  in the second equation) of Lemma~\ref{Rel} is not zero, whence, by Lemma~\ref{Rel} and Lemma~\ref{eqs2f}, we get 
\begin{equation}
\{a_{-2}, a_3, s_{1,2}\} \subseteq  U.
\end{equation} 
It follows that 
\begin{equation}\label{Utau0}
U^{\tau_0}=\langle a_{1}, a_{0}, a_{-1}, a_{-2}, s_{0,1}, s_{0,2} \rangle=U
\end{equation}
 and 
similarly 
\begin{equation}\label{Utau1}
U^{\tau_1}=\langle a_{3}, a_{2}, a_{1},a_{0}, s_{0,1}, s_{0,2} \rangle=U.
\end{equation}
Hence, $a_i\in U$ for every $i\in \Z$. Moreover, by Equation~(27) and Lemma~\ref{primo},
\begin{equation}\label{a0U}
a_0U\subseteq U.
\end{equation}
Since $a_3\in U$, by Equation~(\ref{a0U}) we get 
\begin{equation}\label{s3inU}
s_{0,3}\in U.
\end{equation}
By Lemma~\ref{action}, $s_{2,3}=s_{0,3}^{\tau_1}$ and $s_{1,3}=s_{2,3}^{\tau_0}$, whence, by Equations~(\ref{Utau0}),~(\ref{Utau1}), and~(\ref{s3inU}), it follows that 
$\{s_{0,3} , s_{1,3}, s_{2,3}\}\subseteq U,
$
which implies that  
$$a_ia_j\in U \:\mbox{ for every }\: \{i,j\}\subseteq \{-1, 0, 1, 2\}.
$$ 
By Lemma~\ref{primo} and Equations~(\ref{Utau0}) and~(\ref{Utau1}), it follows that 
$$
a_is_{0,1}\in U\: \mbox{ for every } \:i\in \{-1, 0, 1, 2\}.
$$
 By Corollary~\ref{cor:a0s2f}, $a_1s_{0,2}\in U$. Since $a_{-1}s_{0,2}=(a_1s_{0,2})^{\tau_0}$ and $a_{2}s_{0,2}=(a_0s_{0,2})^{\tau_1}$, again by Equations~(\ref{Utau0}) and~(\ref{Utau1}),  we have 
 $$
 \{a_1s_{0,2}, a_{-1}s_{0,2}, a_{2}s_{0,2} \}\subseteq U.
 $$
  Finally, for $\{i,j\}\subseteq \{1, 2\}$, $s_{0,i}s_{0,j}\in U $ by Lemmas~\ref{a0s2f},~\ref{s1s2}, and ~\ref{s2s2}. Therefore $U$ is a subalgebra of $V$. Since $a_0, a_1\in U$, we have $V=U$.
\medskip

{\bf Claim 2.} {\it
Set 
$$
P(x)=\frac{40\bt^2-14\bt+1}{12\bt^2}x-\frac{(10\bt-1)^2}{192\bt^2}\in \F[x].
$$
\begin{enumerate}
\item If $\lmu= \frac{18\bt-1}{8}$, then either $a_{0}=a_2$ or 
$
\lmd=P(\lmf).
$
\item If $\lmf= \frac{18\bt-1}{8}$, then either $a_{-1}=a_1$ or 
$
\lmdf=P(\lmu).
$
\end{enumerate}}
\medskip

If $\lmu= \frac{18\bt-1}{8}$, then the second equation in Lemma~\ref{Rel} becomes
$$
0=12\bt^2(P(\lmu^f)-\lmd)(a_0-a_2)
$$
and (1) follows.  Similarly for the second claim, using the first equation in Lemma~\ref{Rel}.

\medskip

{\bf Claim 3.} {\it Assume $\lmu\neq  \frac{18\bt-1}{8}=\lmf$. Then $V=\langle a_{-2}, a_{-1}, a_0, a_1, a_2, $ $s_{0,1}, s_{0,2} \rangle$}.
\medskip

In this case let $U:=\langle a_{-2}, a_{-1}, a_0, a_1, a_2, s_{0,1}, s_{0,2} \rangle$. Notice that
\begin{equation}\label{a0U1}
a_0U\leq U\:\mbox{ and }\: U^{\tau_0}=U.
\end{equation}
Since $\lmu\neq \frac{18\bt-1}{8}$, by the second equation in Lemma~\ref{Rel} we get 
\begin{equation}\label{eqa3}
a_{3}\in \langle a_{-1}, a_0, a_2\rangle,
\end{equation}
whence $a_0a_3\in \langle a_{-1}, a_0, a_2, s_{0,1}, s_{0,2} \rangle,$
so, by the definition of $s_{0,3}$, it follows that 
\begin{equation}\label{eqs03}
s_{0,3}\in \langle a_{-1}, a_0, a_2, s_{0,1}, s_{0,2} \rangle .
\end{equation}
Then, by Equations~(\ref{eqa3}) and (\ref{eqs03}), 
\begin{eqnarray*}
s_{2,3}&=&(s_{0,3})^{\tau_1}\in \langle a_{-1}, a_0, a_2, s_{0,1}, s_{0,2} \rangle^{\tau_1}\\
&=&\langle a_{3}, a_2, a_0, s_{0,1}, s_{0,2} \rangle \leq \langle a_{-2}, a_0, a_1, s_{0,1}, s_{0,2} \rangle
\end{eqnarray*}
and similarly
$$
s_{1,3}=(s_{2,3})^{\tau_0}\in \langle a_{1}, a_{-2}, a_0, s_{0,1}, s_{0,2} \rangle .
$$
Therefore, by the definitions of $s_{1,3}$ and $s_{2,3}$, we get
\begin{equation}\label{s13}
\{a_3, a_{-2}a_1, a_{-1}a_2, s_{0,3}, s_{1,3}, s_{2,3}\}\subseteq U.
\end{equation}

Since by assumption
\begin{equation}\label{lmf=}
\lmf=\frac{18\bt-1}{8},
\end{equation}
 by Claim 2 either $a_{-1}=a_1$ or $\lmdf=P(\lmu)$. Assume $a_{-1}=a_1$, whence, by the definition of $s_{1,2}$,  $s_{1,2}\in \langle a_1\rangle$. Then, by~\cite[Theorem~1.1]{axet}, $a_{-2}=a_0$ and so, by Lemma~\ref{primo}, $\langle\langle a_{0}, a_2\rangle \rangle=\langle a_{0}, a_2, s_{0,2}\rangle$. Since $s_{1,2}\in \langle a_1\rangle$,  by Lemma~\ref{eqs2f}, we get $s_{0,2}\in \langle a_{0}, a_1, a_{2}, s_{0,1}\rangle=U$. By Lemmas~\ref{primo} and~\ref{a0s2f}, $\langle a_{0}, a_1, a_{2}, s_{0,1}\rangle$ is closed under the algebra product, so must be equal to $V$.
Now assume 
\begin{equation}\label{lmdf=P}
\lmdf=P(\lmu).
\end{equation}
 Then (see~\cite[genericsakuma-4bt.s]{code}) the equation in Lemma~\ref{rel1} becomes
\begin{eqnarray}\label{val}
0&=&
\frac{2\bt-1}{24}\left (16\lmu-1\right )(a_{-1}+a_{1}) \nonumber \\
&-&\frac{1}{36\bt^2} \left [ 
8(32\bt^2-23\bt+3)\lmu^2-2\bt(264\bt^2-178\bt+19)\lmu \right .\\
&&\left .+12\bt^2(6\bt-1)\lmd-3\bt^2(22\bt-3) \right ] a_0 \nonumber \\
&+& \frac{2\bt-1}{12\bt}\left (16\lmu-1\right )s_{0,1}. \nonumber 
\end{eqnarray}
If $\lmu\neq \frac{1}{16}$, then, since we are assuming $\bt\neq \tfrac{1}{2}$, by Equation~(\ref{val}) we get $s_{0,1}\in \langle a_{-1}, a_0, a_1\rangle$, whence $ \langle a_{-1}, a_0, a_1\rangle$ is a subalgebra of $V$, hence $V=\langle a_{-1}, a_0, a_1\rangle$ and the claim holds. So assume  
\begin{equation}\label{lmu=1/16}
\lmu= \frac{1}{16}.
\end{equation}
 Then Equation~(\ref{val}) becomes
$$
\left (-\frac{6\bt-1}{3}\lmd+\frac{1056\bt^3-344\bt^2+33\bt-1}{384\bt^2}\right )a_0=0, 
$$
whence
\begin{equation}\label{lmd=}
\lmd=\frac{1056\bt^3-344\bt^2+33\bt-1}{128\bt^2(6\bt-1)}.
\end{equation}
Substituting in the equation of Lemma~\ref{rel2} the values of $\lmu$, $\lmf$, $\lmd$, and $\lmdf$ given in Equations~(\ref{lmf=}),~(\ref{lmdf=P}),~(\ref{lmu=1/16}),and~(\ref{lmd=}), we get that the coefficient of $a_4$ in that equation is equal to 
$$
\frac{3\bt^2(32\bt-1)}{16}.
$$
If $\bt\neq \frac{1}{32}$, we can express $a_4$ as a linear combination of $a_{-2}, a_{-1}, a_0, a_2, a_3, s_{0,3}, s_{2,3}$ and so, by Equation~(\ref{s13}), we get $a_4\in U$. Thus $U$ is invariant under $\tau_0$ and $\tau_1$, whence $a_i\in U$ for every $i\in \Z$. As in the proof of Claim 1 we get $U=V$.

Finally assume $\bt=\frac{1}{32}$. In this case, since by hypothesis $\bt\neq \frac{1}{2}$, $\F$ has characteristic other than $5$. As previously, substituting, this time, in the equation of Lemma~\ref{rel3} the above values of $\lmu$, $\lmf$, $\lmd$, and $\lmdf$ we get that the coefficient of $a_{-2}$ is equal to  $-\frac{15}{65536}\neq 0$. Thus, by that equation, we get
$$ 
a_{-2}\in \langle  a_{-1}, a_0, a_1, a_2, a_3, s_{0,1}, s_{0,2}, s_{1,2}, s_{0,3}, s_{1,3} \rangle ,
$$ 
whence, by Lemma~\ref{eqs2f} and Equation~(\ref{s13}),
$$
a_{-2}\in \langle a_{-1}, a_0, a_1, a_2, s_{0,1}, s_{0,2} \rangle .
$$ 
As in the previous case, we conclude that $V=\langle a_{-1}, a_0, a_1, a_2, s_{0,1}, s_{0,2} \rangle$. 
\medskip

{\bf Claim 4}. {If $\lmu= \frac{18\bt-1}{8}$ and $\lmf\neq \frac{18\bt-1}{8}$, $V=\langle a_{-3}, a_{-2}, a_{-1}, a_0, a_1, s_{0,1}, s_{1,2} \rangle$.}
\medskip

By the symmetry of the hypothesis, this follows from Claim 3 swapping the even subalgebra and the odd subalgebra.
\medskip

{\bf Claim 5}. {If $\lmu=\lmf= \frac{18\bt-1}{8}$, then $V=\langle a_{-2}, a_{-1}, a_0, a_1, a_2, a_3, s_{0,1}, s_{0,2}\rangle$.}
\medskip

Set 
\begin{equation}\label{omega}
\omega:=P\left (\frac{18\bt-1}{8}\right )=\frac{480\bt^3-228\bt^2+28\bt-1}{64\bt^2}.
\end{equation}
By Claim 2, we have 
$$
\mbox{ either } a_0=a_2, \mbox{ or } a_1=a_{-1},\mbox{  or } \lmd=\lmdf=\omega .
$$
Assume first $a_0=a_2$. Then, $s_{0,2}\in \langle a_0\rangle$ and so
%by~\cite[Theorem~1.1]{axet}, $a_3=a_{-1}$ and so, by Lemma~\ref{primo}, $\langle\langle a_{-1}, a_1\rangle \rangle=\langle a_{-1}, a_1, s_{1,2}\rangle$. 
by Lemma~\ref{eqs2f}, it follows that  $s_{1,2}\in 
\langle a_{-1}, a_0, a_{1}\rangle$. Thus, by Lemmas~\ref{primo} and~\ref{a0s2f}, $\langle a_{-1}, a_0, a_{1}, s_{0,1}\rangle$ is closed under the algebra product and hence must be equal to $V$. Similarly, if $a_{-1}=a_1$, we get $V=\langle a_0, a_{1}, a_2, s_{0,1}\rangle$.  

Assume now $a_0\neq a_2$, $a_{-1}\neq a_1$ and $\lmd=\lmdf=\omega$. Substituting in the equation of Lemma~\ref{rel1} the value $\frac{18\bt-1}{8}$ in place of $\lmu$ and $\lmf$, and the value $\omega$ in place of $\lmd$, and $\lmdf$, we get 
\begin{eqnarray*}
0&=&(2\bt-1)(12\bt-1)\left [\frac{1}{8}(a_{-1}+a_1)-\frac{48\bt^2-18\bt+1}{64\bt^2}a_0+\frac{1}{4\bt}s_{0,1}\right ].
\end{eqnarray*}
Furthermore, applying $f$ to the equation of Lemma~\ref{rel1} and again substituting the above values of $\lmu$, $\lmf$, $\lmd$, and $\lmdf$, we get 
\begin{eqnarray*}
0&=&(2\bt-1)(12\bt-1)\left [\frac{1}{8}(a_{2}+a_0)-\frac{48\bt^2-18\bt+1}{64\bt^2}a_1+\frac{1}{4\bt}s_{0,1}\right ].
\end{eqnarray*}
Hence, if $\bt\neq \tfrac{1}{12}$, we get $a_{-1}\in \langle a_0, a_1, s_{0,1}\rangle$ and  
$a_{2}\in \langle a_0, a_1, s_{0,1}\rangle$. By Lemma~\ref{primo}, it follows that $\langle a_0, a_1, s_{0,1}\rangle$ is a subalgebra of $V$, whence $V=\langle a_0, a_1, s_{0,1}\rangle$.

From now on we assume $\bt=\tfrac{1}{12}$, thus, by Equation~(\ref{omega}) and the assumptions in Claim 5, we have
$$
\lmu=\lmf=\lmd=\lmdf=\frac{1}{16}.
$$
Moreover, since by hypothesis $\bt\neq \tfrac{1}{2}$, $\F$ has characteristic other than $5$. Set
$$
U:=\langle a_{-3}, a_{-2}, a_{-1}, a_0, a_1, a_2, a_3, s_{01,}, s_{0,2}\rangle . 
$$
Replacing each of $\lmu$, $\lmf$, $\lmd$, and $\lmdf$ by $\frac{1}{16}$ in the equation of Corollary~\ref{eqs2f} we get
$$
s_{1,2}=s_{0,2}.
$$
Computing the value of the function $\lambda_{a_0}$ on both sides of the equation in Lemma~\ref{rel3}, we get $0=\frac{1}{768}(\lmt-\lmt^f)$,  whence 
$$
\lmt^f=\lmt.
$$ 
%Similarly, computing the value of the function $\lambda_{a_0}$ (resp. $\lambda_{a_1}$) on both sides of the equation in Lemma~\ref{rel2}, we get $0=\frac{5}{36864}(16\lm_4-1)$ (resp. $0=\frac{5}{36864}(16\lm_4^f-1)$),  whence 
%$$ \lm_4=\lm_4^f=\tfrac{1}{16}. $$ 
It follows that
equation in Lemma~\ref{rel23} (multiplied by 2304/5) gives
\begin{eqnarray} \label{a4}% e' la rel23
a_4&=&-a_3+\tfrac{8}{5}(6\lmt-1)(a_2-a_{-1})+a_{-2}+a_{-3},
\end{eqnarray}
whence $a_4\in \langle a_{-3}, a_{-2}, a_{-1}, a_2, a_3\rangle\leq U$.
Applying $\tau_1$ to the above equation we get
\begin{eqnarray*} \label{a5}% e' la rel23^tau1
a_5&=&a_{-2}+a_{-1}-\tfrac{8}{5}(6\lmt-1)(a_0-a_{3})-a_{4},
\end{eqnarray*}
whence $a_{5}\in \langle a_{-3}, a_{-2}, a_{-1}, a_0, a_2, a_3\rangle \leq U$. It follows that  $U$ is invariant under the action of $\tau_0$ and $\tau_1$, whence $a_i\in U$ for every $i\in \Z$. 
By Lemma~\ref{rel3} we get (see~\cite[genericsakuma-4bt.s]{code})
\begin{equation}\label{S}
s_{1,3}=s_{0,3}-\tfrac{1}{20}(16\lmt-1)(a_0-a_1)+\tfrac{1}{12}(a_{-2}-a_3).
\end{equation}
and so
\begin{equation}\label{S2}
s_{2,3}=s_{1,3}^{\tau_0}=s_{0,3}-\tfrac{1}{20}(16\lmt-1)(a_0-a_{-1})+\tfrac{1}{12}(a_2-a_{-3}).
\end{equation}
Multiplying by $a_1$ both sides of Equation~(\ref{S}) and subtracting the fourth equation in Lemma~\ref{primo} (with $i=3$), by Equation~(\ref{a4}), we get
\begin{eqnarray}\label{Relnew}
0&=&\tfrac{1}{96}a_{-2}+\tfrac{5}{288}a_{-1}-(\tfrac{1}{10}\lmt -\tfrac{17}{720})a_0-(\tfrac{1}{15}\lmt -\tfrac{31}{1440})a_1+\tfrac{5}{288}a_2+\tfrac{1}{144}a_3 \nonumber\\
&&-(\tfrac{4}{5}\lmt -\tfrac{7}{15})s_{0,1}+\tfrac{1}{3}s_{0,2}
+\tfrac{1}{12}s_{0,3}, 
\end{eqnarray}
whence $s_{0,3}\in U$ and so, by Lemma~\ref{primo},
\begin{equation}\label{a00U}
a_0U\leq U.
\end{equation} 
Furthermore, subtracting from Equation~(\ref{Relnew}) its image under $\tau_0$ (and multiplying by $12$) we get  
\begin{eqnarray*}\label{Relnew1}
0&=&(\tfrac{4}{5}\lmt-\tfrac{1}{20})(a_1-a_{-1})+\tfrac{1}{12}(a_{-3}+a_{-2}-a_{2}-a_{3}), 
\end{eqnarray*}
whence $U=\langle a_{-2}, a_{-1}, a_0, a_1, a_2, a_3, s_{01,}, s_{0,2}\rangle$. 
Finally, by Equations~(\ref{S}) and~(\ref{S2}), we have
\begin{equation}\label{s33}
\{s_{1,3}, s_{2,3}\}\subseteq U.
\end{equation}
By Equations~(\ref{a4}),~(\ref{a5}), and~(\ref{a00U}) and the invariance of $U$ under the actions of $\tau_0$ and $\tau_1$, it follows that 
$$
\{ s_{0,4}, s_{2,4}, s_{0,5}, s_{3,5} \} \subseteq U. 
$$
Therefore, 
as in the previous cases we see that $U$ is a subalgebra of $V$, whence $V=U$ and the claim is proved.
\end{proof}

\noindent {\em Proof of Theorem~\ref{Felix}.}
The result follows from Proposition~\ref{reg}, Theorem~\ref{newth}, and~\cite[Theorem~1.1]{FMS2}.
\hfill$\square$

%%%%%%%%%%%%%%%%%%%
%%%%%%%%%%%%%%%%%%%%%%%%%%%%%%%%%%%%%%%%%
\section{The regular case} \label{generic}

In this section we prove Theorems~\ref{nec} and~\ref{symmetric}. As in Section~\ref{table}, we keep the notation of Section~\ref{universal}.
Let $\mathcal V=(\Ro,\Vo, \Amo, (\overline{\mathcal F}, \overline \star))$  be the initial object in the category $\mathcal O_r$. Note that in this case $\hat D$ is a non trivial quotient of 
$$\Z[1/2, x_1, x_2, x_1^{-1}, x_2^{-1}, (x_1-x_2)^{-1}, (x_1-2x_2)^{-1},(x_1-4x_2)^{-1}].
$$
By Corollary~\ref{Rin}, $\Ro=\hat D[\lambda_1, \lambda_1^f, \lambda_2, \lambda_2^f]$. The elements $\lambda_1, \lambda_1^f, \lambda_2, \lambda_2^f$ are not necessarily indeterminates on $\hat D$, as they have to satisfy the relations imposed by the definition of $\Ro$. In particular, 
\begin{lemma}\label{rels3}
In the ring $\Ro$ the following equalities hold
\begin{enumerate}
\item $\lambda_{\ba_0}(\bs_{2,3}-(\bs_{2,3})^f)=0$,
\item $\lambda_{\ba_0}(\ba_{4}\ba_4-\ba_4)=0$,%$\lambda_{\ba_0}((\bs_{2,3}-(\bs_{2,3})^f)^{\tau_1})=0$,
\item $\lambda_{\ba_1}(\bs_{2,3}-(\bs_{2,3})^f)=0$,
\item $(\lambda_{\ba_0}(\ba_{4}\ba_4-\ba_4))^f=0$.
\end{enumerate}
\end{lemma}
\begin{proof}
Assertions {\it (1)} and {\it (3)} follow since by Lemma~\ref{action},  $\bs_{2,3}-\bs_{2,3}^f=0$ in $\Vo$. Assertion {\it (2)} and  {\it (4)} follow since, by definition, $\ba_4$ is an idempotent.
\end{proof}

The four identities in Lemma~\ref{rels3} produce four polynomials $p_i(x,y,z,t)$ for $i\in \{1, \ldots , 4\}$ in $\hat D [x,y,z,t]$ (with $x,y,z,t$ algebraically independent indeterminates on $\hat D$), that simultaneously annihilate on the quadruple $(\lambda_1, \lambda_1^f, \lambda_2, \lambda_2^f)$.  Precisely,  let 
$$B:=\{\ba_{-2}, \ba_{-1}, \ba_{0}, \ba_{1}, \ba_{2}, \bs_{0,1}, \bs_{0,2}, \bs_{1,2}\}
$$
as in the proof of Proposition~\ref{span}. By Lemma~\ref{a3} we can write explicitly $\ba_3$ as a linear combination of the elements of $B$ with coeffcients in $\Ro$. By Lemmas~\ref{primo} and~\ref{a0s2f}, for every $\bf b\in B$, we can write $\ba_0\bf b$ (hence $\bs_{0,3}$, $\bs_{2,3}=(\bs_{0,3})^{\tau_1}$, and $\bs_{2,3}^f$) again as a $\Ro$-linear combinations of the elements of $B$.  Using linearity of $\lambda_{\ba_0}$ and $\lambda_{\ba_1}$ and Corollary~\ref{Rin}, we get the desired polynomials. As mentioned in the Introduction, these polynomials are too large to be displayed here but can be computed using~\cite[genericsakuma.s]{code}.

For $i\in \{1,2\}$, define
$$q_i(x,z):=p_i(x,x,z,z).
$$ 
\bigskip

{\it Proof of Theorem~\ref{nec}}.
Let $\F$ be a field of characteristic other than $2$ and let $(V, \F, (a_0, a_1), \mathcal M(\al, \bt)) \in \mathcal M_r(2,\F)$. By Lemma~\ref{campo},  $(V, \F, (a_0, a_1), \mathcal M(\al, \bt)) \in \mathcal O_r$.
Let $\mathbb A^4_\F$ be the $4$-dimensional affine space over $\F$ and let 
$$
 P_V:=(\lambda_{ a_0}( a_1), \lambda_{ a_1}( a_0), \lambda_{ a_0}( a_2), \lambda_{ a_1}( a_{-1}))\in \mathbb A^4_\F.
 $$
Let 
\begin{equation*}
\begin{array}{ccccc}
 \xi &\colon &\mathcal M_r(2,\F) &\to &\mathbb A^4_\F \\
  & & V & \mapsto & P_V.
  \end{array}
  \end{equation*}
 
Set $\tilde \F:=\Ro\otimes_{\hat D} \F$ and  let $\varphi_\F\colon \tilde \F\to \F$ be the map defined in Corollary~\ref{tensor}.(1) (with $R=\F$). Let $\F_0$ be the prime subfield of $\F$. Since $(V, \F, (a_0, a_1), \mathcal M(\al, \bt)) \in \mathcal O_r$, there is a ring homomorphism 
 $$
 \eta: \hat D \to \F_0(\al, \bt)
 $$
which extends  to a homomorphism $\bar \eta$ of the corresponding polynomial rings on the indeterminates $x$, $y$, $z$, $t$:
$$
\bar \eta: \hat D[x,y,z,t]\to  \F_0(\al, \bt)[x,y,z,t].
$$   
%We denote the images of an element $\delta$ of $\Ro\otimes_{\hat D} \F$ in $\F$ via $v_{P}$ by $\bar \delta$ and by $\overline{p}_i$ and $\overline{q}_i$ the polynomials in $\F[x,y,z,t]$ and $\F[x,z]$ corresponding to $p_i$ and $q_i$, respectively. 
By Lemma~\ref{rels3}, for every $i\in \{1,2,3,4\}$, 
$$
p_i(\lm_1, \lmf, \lmd, \lmdf)=0 \mbox{ in } \Ro.
$$
Hence, by Corollary~\ref{tensor}, 
$$
p_i^{\bar \eta}(P_V)=(p_i(\lm_1, \lmf, \lmd, \lmdf)\otimes 1)^{\varphi_\F}=(0 \otimes 1)^{\varphi_\F}=0  \mbox{ in } \F.
$$
So, $\im (\xi) \subseteq \Var(p_1^{\bar \eta}, p_2^{\bar \eta}, p_3^{\bar \eta}, p_4^{\bar \eta})$, 
%with 
%$T:=\{p_1^{\bar \eta}, p_2^{\bar \eta}, p_3^{\bar \eta}, p_4^{\bar \eta}\},$
proving (1).

Now set
$$
U_{P_V}:=\frac{\Vo\otimes_{\Ro} \tilde \F}{(\Vo\otimes_{\Ro} \tilde \F)\ker (\varphi_\F)}.
$$
Note that, by the definition of $U_{P_V}$, $\ker (\varphi_\F)\leq ann_{\tilde \F}(\Ro\otimes_{\hat D} \F)$, hence $U_{P_V}$ is an $\F$-algebra via $\varphi_\F$. Moreover, by Equation~(\ref{ciserve3}), $(\Vo\otimes_{\Ro} \tilde \F)\ker (\varphi_\F)\leq \ker (\varphi_V)$, whence $V$ is isomorphic to a quotient of $U_{P_V}$. 
A straightforward check shows that $\Vo\otimes_{\Ro} \tilde \F$ (resp. $U_{P_V}$) is a primitive axial algebra of Monster type $(\alpha, \beta)$ with generating axes $a_0\otimes 1$ and $a_1\otimes 1$ (resp. their images in $U_{P_V}$), which proves (2).
\hfill $\square$
\bigskip

\begin{remark}\label{casosimmetrico}
Assume, under the hypotheses of Theorem~\ref{nec}, that the algebra $V$ is also symmetric. Then 
$$\lambda_{ a_0}(a_1)= \lambda_{a_1}(a_0) \mbox{ and } 
\lambda_{a_0}(a_2),=\lambda_{a_1}(a_{-1})
$$
so the pair $ (\lambda_{a_0}(a_1), \lambda_{a_0}(a_2))$ is a common zero of the polynomials $
q_1^{\bar \eta}$ and $q_2^{\bar \eta}.
$
\end{remark}
Computing the resultant of the polynomials $q_1$ and 
$q_2$ with respect to $z$ one obtains a polynomial in $x$ of degree at most 10, which is the product of the five linear factors
$$
x,\:\: x-1\:\:, 2x-\al, \:\: 2x-\bt, \:\:4(2\al-1)x-(3\al^2+3\al\bt-\al-2\bt)
$$
and a factor of degree at most $5$ (see~\cite[genericsakuma.s]{code}).  
The last factor has degree $5$ and is irreducible in $\Q(\al, \bt)[x]$, if $\al$ and $ \bt$ are indeterminates over $\Q$. On the other hand, for certain values of $\alpha$ and $\beta$, this factor can be reducible: for example, it even completely splits in $\Q(\al, \bt)[x]$ when $\al=2\bt$ (see~\cite{FMS2}), or in the Norton-Sakuma case, when $(\al, \bt)=(1/4, 1/32)$ (see the proof of Theorem~\ref{thm} below). 

\begin{remark}\label{evenodd}
Under the hypotheses of Theorem~\ref{nec},  let 
\begin{enumerate}
\item []  $V_e:=\langle \langle a_0, a_2\rangle \rangle$ (the {\it even subalgebra} of $V$) 
\item [] $V_o:=\langle \langle  a_{-1}, a_1\rangle \rangle$ (the {\it odd  subalgebra} of $V$). 
\end{enumerate}
Since the automorphism  $\tau_{a_0}$ (resp. $\tau_{a_1}$) swaps $a_0$ and $a_2$ (resp. $a_{-1}$ and $a_1$),  $V_e$ and $V_o$ are symmetric (primitive $2$-generated axial algebras of Monster type $(\al, \bt)$). Using Yabe's classification of the symmetric case~\cite[Theorem~2]{Yabe} one can see that the last two coordinates $
\lambda_{a_0}(a_2)$ and $\lambda_{a_1}(a_{-1})
$ of $P_V$ vary in a finite and relatively small set (getting a bound of $16^2$, which can be further reduced using the theory of axets in~\cite{axet}). 
\end{remark}

%%%%%%%%%%%%%%%%%%%%%%%%%%%%%%%%%%%%%%%%%%%%%%%%%
\section{The generic case}\label{truegeneric}

As an application, we now classify $2$-generated primitive axial algebras of Monster type $(\al, \bt)$ over the field $\F:=\Q(\al, \bt)$, with $\al$ and $\bt$ independent indeterminates over $\Q$ (the {\it generic} case). We keep the notation of the previous sections, in particular, 
$V$ is an axial algebra of Monster type $(\al, \bt)$ over the field $\F$ generated by the two axes $a_0$ and $a_1$. Set
 $$
\bar \lambda_1:=\lambda_{a_0}(a_1), \:  \bar \lambda_1^\prime:=\lambda_{a_1}(a_0), \:
 \bar \lambda_2:=\lambda_{a_0}(a_2), \:\mbox{ and } \bar \lambda_2^\prime:=\lambda_{a_1}(a_{-1}). 
 $$
By Theorem~\ref{nec}, $V$ is isomorphic to a quotient of the algebra $U_P$ uniquely determined by the quadruple $P:= (\bar \lambda_1, \bar \lambda_1^\prime, \bar \lambda_2,  \bar \lambda_2^\prime)\in \Var(p_1, p_2, p_3, p_4)$.
By Remark~\ref{casosimmetrico} and  Remark~\ref{evenodd}, $\bar \lambda_2$ and   $\bar \lambda_2^\prime$ have to be the first coordinates of a point lying in $\Var(q_1, q_2)$. Set
$$
  \nu:= 
\frac{(3\al^2+3\al\bt-\al-2\bt)}{4(2\al-1)}.$$

\begin{lemma}\label{solutions}
With the above notation 
\begin{enumerate}
\item ${\Var} (q_1, q_2)=\{(1,1), (0,1), (\tfrac{\bt}{2}, \tfrac{\bt}{2}),(\tfrac{\al}{2}, 1 ), (\nu, \nu)\}$. 
\item In particular, 
%\begin{equation}\label{L}
$\{\bar \lambda_2, \bar \lambda_2^\prime\} \subseteq  \{ 1, 0, \tfrac{\al}{2}, \tfrac{\bt}{2}, \nu \}$. 
%\end{equation}
\end{enumerate}
\end{lemma}
\begin{proof}
The first assertion is obtained by solving computationally the system 
\begin{align*}
\left \{ \begin{array}{c}
q_1(x,z)=0\\
q_2(x,z)=0
\end{array}
\right .
\end{align*}
in~\cite[genericsakuma.s]{code}. More specifically, we first computed the resultant with respect to $z$ of the two polynomials and then the  solutions  of that resultant. The second assertion follows from the above discussion,  by taking the first coordinates of the points in ${\Var} (q_1, q_2)$.
\end{proof}

\begin{lemma}\label{solgen}
With the above notation,  let $Q=(x_0, y_0, z_0, t_0)\in \mathbb A^4$ and let 
$\{z_0, t_0\}\subseteq  \{ 1, 0, \tfrac{\al}{2}, \tfrac{\bt}{2}, \nu \}$. Then $Q\in {\Var}(p_1, p_2, p_3, p_4)$  if and only if 
$$
Q\in \{(1,1,1,1),\:\: (0,0,1,1), \:\: (\tfrac{\bt}{2}, \tfrac{\bt}{2}, \tfrac{\bt}{2}, \tfrac{\bt}{2}),\:\: ( \tfrac{\al}{2}, \tfrac{\al}{2}, 1,1 ), \:\: (\nu, \nu, \nu, \nu)\}.
$$
\end{lemma}
\begin{proof}
Using~\cite[genericsakuma.s]{code} one can check that the above five points lie in ${\Var}(p_1, p_2, p_3, p_4)$. Conversely, 
for every choice of $\{z_0, t_0\}\subseteq  \{ 1, 0, \tfrac{\al}{2}, \tfrac{\bt}{2}, \nu \}$ and for every $i\in \{2,3\}$, let  
 $r_{1,i}(x, z_0, t_0)$ be the resultant, with respect to $y$, between $p_1(x,y, z_0, t_0)$ and $p_i(x,y, z_0, t_0)$.
The common zeros  $(x_0, z_0, t_0)$ of $r_{1,2}(x, z_0, t_0)$ and $r_{1,3}(x, z_0, t_0)$, computed in~\cite[genericsakuma.s]{code},  are  those in the following set
$$
\{(1, 1, 1), \:(0, 1, 1), \:(\tfrac{\al}{2}, 1, 1), \:(0, 1, 0), \:(0, 1, \tfrac{\bt}{2}),
\:(0, 1, \tfrac{\al}{2}), \:(0, 1, \nu), \: (\tfrac{\bt}{2}, \tfrac{\bt}{2}, \tfrac{\bt}{2}), \:(\nu, \nu, \nu)\}.
$$
Finally, using~\cite[genericsakuma.s]{code}, we get that   the equation 
$$
p_3(x_0,y, z_0, t_0)=0
$$ 
in the indeterminate $y$ has no solution in $\F$, if 
$$
(x_0, z_0, t_0)\in \{0,1,0), (0, 1, \tfrac{\bt}{2}), (0, 1, \tfrac{\al}{2}), (0, 1, \nu)\},
$$ 
i.e. $z_0\neq t_0$. For each of the  remaining cases (again using~\cite[genericsakuma.s]{code}), we get a unique solution in $\F$, precisely
$$
\begin{array}{llllll}
p_3(1,y,1,1)=0&  \Longrightarrow  &y=1,& \Longrightarrow  &(x_0, y_0, z_0, t_0)=(1,1,1,1), \\
p_3(0,y,1,1)=0& \Longrightarrow  &y=0,& \Longrightarrow  &(x_0, y_0, z_0, t_0)=(0,0,1,1),\\
p_3(\tfrac{\al}{2},y,1,1)=0& \Longrightarrow  &y=\tfrac{\al}{2},& \Longrightarrow  &(x_0, y_0, z_0, t_0)=( \tfrac{\al}{2}, \tfrac{\al}{2}, 1,1 ),\\
p_3(\tfrac{\bt}{2},y,\tfrac{\bt}{2},\tfrac{\bt}{2})=0& \Longrightarrow  &y=\tfrac{\bt}{2},& \Longrightarrow &(x_0, y_0, z_0, t_0)=(\tfrac{\bt}{2}, \tfrac{\bt}{2}, \tfrac{\bt}{2}, \tfrac{\bt}{2}),\\
p_3(\nu, y, \nu, \nu)=0 &\Longrightarrow &y= \nu,& \Longrightarrow  &(x_0, y_0, z_0, t_0)=(\nu, \nu, \nu, \nu),
\end{array}
$$
proving the result.
\end{proof}

\begin{lemma}\label{simple}
The algebras $3C(\al)$, $3C(\bt)$, and  $3A(\al, \bt)$ over the field $\F$ are simple $2$-generated symmetric axial algebras of Monster type $(\al, \bt)$. 
\end{lemma}
\begin{proof}
For the algebras $3C(\al)$ and  $3C(\bt)$ the claim is proved in~\cite[Example~3.4]{HRS1}. 
  The algebra $3A(\al, \bt)$ is the algebra $3A^\prime_{\al, \bt}$ defined by Reheren in~\cite[Table~9]{R}. By~\cite[Theorem~8.1]{R}, it is a $2$-generated symmetric axial algebra of Monster type $(\al, \bt)$. By~\cite[Lemma~8.2]{R}, it admits a Frobenius form which is non degenerate over the field $\F$ and such that all the generating axes are non-singular with respect to this form. Hence, by Theorem~4.11 in~\cite{KMS}, every non trivial ideal contains at least one of the generating axes. Then, Corollary~4.6 in~\cite{KMS} yields that the algebra is simple.
\end{proof}

\noindent {\em Proof of Theorem~\ref{symmetric}.}
The algebras  $1A$, $2B$, $3C(\al)$, $3C(\bt)$, and $3A(\al, \bt)$ are $2$-generated symmetric axial algebras of Monster type $(\al, \bt)$: this is trivial for $1A$ and $2B$, and follows by Lemma~\ref{simple} in the remaining cases.

Conversely, let $V$ be an axial algebra of Monster type $(\al, \bt)$ over the field $\F$ generated by the two axes $a_0$ and $a_1$. Set
 $$
\bar \lambda_1:=\lambda_{a_0}(a_1), \:  \bar \lambda_1^\prime:=\lambda_{a_1}(a_0), \:
 \bar \lambda_2:=\lambda_{a_0}(a_2), \:\mbox{ and } \bar \lambda_2^\prime:=\lambda_{a_1}(a_{-1}). 
 $$
By Theorem~\ref{nec}, $V$ is isomorphic to a quotient of the algebra $U_P$ uniquely determined by the quadruple $P:= (\bar \lambda_1, \bar \lambda_1^\prime, \bar \lambda_2,  \bar \lambda_2^\prime)$, which belongs to $\mathcal \Var(p_1, p_2, p_3, p_4)$. By Remark~\ref{evenodd}, Remark~\ref{casosimmetrico} and Lemma~\ref{solgen}, $P$ is an element of the set
$$
\mathcal P:=\{(1,1,1,1),\:\: (0,0,1,1), \:\:(\tfrac{\bt}{2}, \tfrac{\bt}{2}, \tfrac{\bt}{2}, \tfrac{\bt}{2} ),\:\: (\tfrac{\al}{2}, \tfrac{\al}{2}, 1,1  ), (\nu, \nu, \nu, \nu)\}.
$$
By Corollary~\ref{tensor} and Proposition~\ref{span}, $V$ is linearly spanned as an $\F$-vector space by the set 
$$B:=\{a_{-2}, a_{-1}, a_{0}, a_{1}, a_{2}, s_{0,1}, s_{0,2}, s_{1,2}\}.
$$
 Define
$$
d_0:=s_{2,3}-s_{1,3}^{\tau_0}, \:\:
d_1:=d_0^f, \:\;   
d_2:={d_0}^{\tau_1},
$$
and, for $ i\in \{0,1,2\}$,
$$ 
D_i:={d_i}^{\tau_0}-d_i. 
$$
The basic idea of the proof is to extract, for each choice of  $P\in \mathcal P$, from the set $B$ a basis for $V$ and show that the structure constants relative to these bases  are the same as those of the target algebras. To achieve that, we write in two different ways the vectors $d_i$ and $D_i$, for $i\in \{0,1,2\}$, as linear combinations of the elements of $B$. 
On one side, by Lemma~\ref{action}, all vectors  $d_i, D_i$ for $i\in \{0,1,2\}$ must be equal to the zero vector. On the other, following Remark~\ref{struct}, we can express these vectors also as non-trivial linear combinations of the elements of $B$. The coefficients of these linear combinations have been computed in~\cite[genericsakuma.s]{code}\footnote{In the sequel, for the convenience of the reader, we have cited explicitly each time the code~\cite[genericsakuma.s]{code} has been used.}. In particular, for every $P\in \mathcal P$, the coefficient of $a_{-2}$ in $D_0$ is non zero, hence, using~\cite[genericsakuma.s]{code}, we can express $a_{-2}$ as a linear combination of $a_{-1}$, $a_{0}$, $a_{1}$,$a_{2}$, $s_{0,1}$, $s_{0,2}$, and $s_{1,2}$. Similarly, since the coefficient of $s_{1,2}$ in $d_0$ is non zero, using~\cite[genericsakuma.s]{code}, we can express $s_{1,2}$ as a linear combination of $a_{-1}$, $a_{0}$, $a_{1}$, $a_{2}$, $s_{0,1}$, and $s_{0,2}$. So $V$ is linearly spanned by 
$$
B_1:=\{a_{-1}, a_{0}, a_{1}, a_{2}, s_{0,1}, s_{0,2}\}.
$$
In particular, Remark~\ref{struct} (by~\cite[genericsakuma.s]{code}) gives an expression of the product $s_{0,1}s_{0,2}$ as a linear combination of the vectors in $B_1$.

Assume 
$$
P\in  \{(\tfrac{\bt}{2}, \tfrac{\bt}{ 2}, \tfrac{\bt}{ 2}, \tfrac{\bt}{ 2} ),  (\nu, \nu, \nu, \nu) \}.
$$ 
Arguing as above, from the identity $d_2=0$ by~\cite[genericsakuma.s]{code} we get $a_{-1}=a_2$ and consequently, by Equation~(\ref{s}) and Lemma~\ref{invariances},
\begin{equation}\label{identity}
s_{0,2}=a_0a_2-\bt(a_0+a_2)=a_0a_{-1}-\bt(a_0+a_{-1})=s_{0,1}.
\end{equation}
So, in this case, $V$ is linearly spanned by 
$$\{a_{0}, a_{1}, a_{2}, s_{0,1}\}.
$$
If $P= (\nu, \nu, \nu, \nu )$, then by Equation~(\ref{s}), Lemma~\ref{primo}, and Lemma~\ref{a0s2f}, by~\cite[genericsakuma.s]{code} we see that $V$ satisfies the same multiplication table as $3A(\al, \bt)$, whence, by Lemma~\ref{simple}, $V\cong 3A(\al, \bt)$.
 If $P=(\frac{\bt}{2}, \frac{\bt}{ 2}, \frac{\bt}{2}, \frac{\bt}{2})$, then from Equation~(\ref{identity}), by~\cite[genericsakuma.s]{code}, we get $s_{0,1}s_{0,2}-s_{0,1}s_{0,1}=0$. Using the expression of $s_{0,1}s_{0,2}$ as a linear combination of the elements of $B_1$, by~\cite[genericsakuma.s]{code}, we get
$$
s_{0,1}=-\frac{\bt}{2}(a_0+a_1+a_2),
$$ 
proving that  $V$ satisfies the same multiplication table as the algebra $3C(\bt)$. Since, again by Lemma~\ref{simple},  $3C(\bt)$ is simple, we get $V\cong 3C(\bt)$. 

Assume  
$$P \in 
\{(1,1,1,1), (0,0,1,1),  (\tfrac{\al}{2}, \tfrac{\al}{2}, 1,1 ) \}.
$$ 
From the identity $D_2=0$, by~\cite[genericsakuma.s]{code}, we get 
\begin{equation}\label{identity4}
a_{-1}=a_1,
\end{equation}
 and consequently, by Equation~(\ref{s}),
  \begin{equation}\label{identity2}
s_{1,2}=a_{-1}a_1-\bt(a_{-1}+a_1)=a_1-2\bt a_1=(1-2\bt)a_1
\end{equation} 
and, by Lemma~\ref{invariances},
\begin{equation}\label{identity3} 
s_{0,2}=s_{1,2}^f=(1-2\bt)a_1^f=(1-2\bt)a_0.
\end{equation}
So, in this case, $V$ is linearly spanned by 
$$\{a_{0}, a_{1}, a_{2}, s_{0,1}\}.
$$
If $P \in 
\{ (0,0,1,1),  (\tfrac{\al}{2}, \tfrac{\al}{2}, 1,1 ) \},
$
then,  from the identity $d_2=0$, by~\cite[genericsakuma.s]{code}, we deduce $a_2=a_0$.  It follows that, if $P=\left (\frac{\al}{2}, \frac{\al}{2}, 1,1 \right )$, then $V$ satisfies the same multiplication table as the algebra $3C(\al)$ and so, by Lemma~\ref{simple}, $V\cong 3C(\al)$. If $P=(0,0,1,1)$,  then, since  $a_0s_{1,2}=(2\bt-1)a_0a_1$, by the formula for the product $a_0s_{1,2}$ in Lemma~\ref{a0s2f} we get $s_{0,1}=-\bt(a_0+a_1)$, whence $a_0 a_1=0$. Hence in this case $V$ is isomorphic to the  algebra $2B$.
Finally, suppose $P=(1,1,1,1)$. Then, from identity $D_0=0$, Equations~(\ref{identity4}), ~(\ref{identity2}),~(\ref{identity3}), and~\cite[genericsakuma.s]{code}, we get $a_{-2}=a_2$. Hence, as in the previous case, from the identity $a_0s_{1,2}=(2\bt-1)a_0a_1$ we can express $s_{0,1}$ as a linear combination of $a_0$, $a_1$ and $a_2$. Then, again using~\cite[genericsakuma.s]{code}, from the identity  $d_2-d_2^{\tau_1}=0$ we get $a_2=a_0$ and from    $d_2=0$ we get $a_0=a_1$, that is $V$ is the algebra $1A$.
\hfill$\square$
\medskip

\begin{proof}[Proof of Corollary~\ref{3-transposition}]
Let $W$ be as in the statement of Corollary~\ref{3-transposition}, let $a$ and $b$ be two axes of $W$, let $\tau_a$ and $\tau_b$ be the corresponding Miyamoto involutions, and let $V$ be the subalgebra of $W$ generated by $a$ and $b$. Then $V$ falls into one of the five cases of Theorem~\ref{symmetric}. We may assume the type of $V$ is not $1A$, otherwise the result is trivial, so $a\neq b$. Let $x:=(\tau_a \tau_b)^n$ where $n=1$, if $V$ is of type $2B$ or $3C(\al)$, and $n=3$ otherwise.  Then $x$ acts trivially on $V$: this follows immediately for the algebras $2B$ and $3C(\al)$, since in these algebras the $\bt$-eigenspace for the adjoint action of every axis is trivial, by~\cite[Example~(3.3)]{HRS1} for the algebra $3C(\bt)$, and  by~\cite[Lemma~8.3]{R} for the algebra $3A(\al, \bt)$. In particular $x$ commutes with $\tau_a$. On the other hand $\tau_a$ acts as the inversion on $\langle \tau_a\tau_b\rangle$, in particular it inverts $x$, thus  
$x$ has order at most $2$. 
If $n=1$ we are done. Assume $n=3$. 
Let $c:=a^{\tau_b}$, then $c\in V$ and  $b=a^{\tau_c}$ (this follows from~\cite[Example~(3.3)]{HRS1}, in case $V$ has type $3C(\bt)$, and  from~\cite[Table 9]{R}, in case $V$ has type $3A(\al, \bt)$). So 
$$\tau_b = \tau_a^{\tau_c}\mbox{ and } \tau_c=\tau_a^{\tau_b}\in \langle \tau_a, \tau_b\rangle.$$ 
In particular $\langle \tau_a, \tau_b\rangle $ cannot be  a dihedral group of order $12$ since in such a group two generating involutions are never conjugate. Whence $x=1$ also in this case, so $\tau_a\tau_b$ has order at most $3$.
\end{proof}

%%%%%%%%%%%%%%%%%%%%%%%%%%%%%%%%%%%%%%%%%%%
\section{Proof of Theorem~\ref{thm}}
Let $\F$ be a field of characteristic $0$ and let $V$ be a primitive axial algebra of Monster type $(\tfrac{1}{4}, \tfrac{1}{32})$ over the field $\F$ generated by the two axes $a_0$ and $a_1$. We proceed as in Section~\ref{truegeneric}. Set
 $$
\bar \lambda_1:=\lambda_{a_0}(a_1), \:  \bar \lambda_1^\prime:=\lambda_{a_1}(a_0), \:
 \bar \lambda_2:=\lambda_{a_0}(a_2), \:\mbox{ and } \bar \lambda_2^\prime:=\lambda_{a_1}(a_{-1}). 
 $$
By Theorem~\ref{nec}, $V$ is isomorphic to a quotient of the algebra $U_P$ uniquely determined by the quadruple $P:= (\bar \lambda_1, \bar \lambda_1^\prime, \bar \lambda_2,  \bar \lambda_2^\prime)\in \Var(p_1, p_2, p_3, p_4)$.
By Remark~\ref{casosimmetrico} and  Remark~\ref{evenodd}, $\bar \lambda_2$ and   $\bar \lambda_2^\prime$ have to be the first coordinates of a point lying in ${\Var}(q_1, q_2)$. 

\begin{lemma}\label{solutionsm}
With the above notation, let  $Q=(x_0,z_0)\in \mathbb A^2$. If $Q\in  {\Var} (q_1, q_2)$, then 
$$
x_0\in \{1, 0, \tfrac{1}{8}, \tfrac{1}{32}, \tfrac{1}{64}, \tfrac{13}{2^8}, \tfrac{3}{2^7}, \tfrac{5}{2^8}\}.
$$
\end{lemma}
\begin{proof}
The assertion follows since the resultant (computed with Singular in~\cite[majorana.s]{code}) with respect to $z$ of the polynomials $q_1$ and $q_2$ has degree $9$ and splits in $\Q[x]$ as the product of a constant and the linear factors
$$
x,\:\: x-1,\: x-\tfrac{1}{8},\:\: \left (x-\tfrac{1}{64}\right )^2,\:\: x-\tfrac{13}{2^8},\:\: x-\tfrac{1}{32}, \:\: x-\tfrac{3}{2^7},\: \: x-\tfrac{5}{2^8}.
$$
\end{proof}

\begin{lemma}\label{solgen2}
With the above notation,  let $Q=(x_0, y_0, z_0, t_0)\in \mathbb A^4$ and assume 
$\{z_0, t_0\}\subseteq  \{ 1, 0, \tfrac{1}{8}, \tfrac{1}{32}, \tfrac{1}{64}, \tfrac{13}{2^8}, \tfrac{3}{2^7}, \tfrac{5}{2^8} \}$. Then $Q\in {\Var}(p_1, p_2,p_3, p_4)$  if and only if 
\begin{align*}
Q\in \{&(1,1,1,1),\:\: (0,0,1,1), \:\:\left (\tfrac{1}{8}, \tfrac{1}{8}, 1,1\right ),\:\: \left (\tfrac{1}{64}, \tfrac{1}{64}, \tfrac{1}{64}, \tfrac{1}{64} \right ),\:\:\left (\tfrac{13}{2^8},  \tfrac{13}{2^8}, \tfrac{13}{2^8}, \tfrac{13}{2^8}\right ),\\
&\left (\tfrac{1}{32}, \tfrac{1}{32}, 0,0\right), \:\: \left (\tfrac{1}{64}, \tfrac{1}{64}, \tfrac{1}{8}, \tfrac{1}{8}\right ), \:\:\left (\tfrac{3}{2^7}, \tfrac{3}{2^7}, \tfrac{3}{2^7}, \tfrac{3}{2^7}\right ), \:\: \left  (\tfrac{5}{2^8}, \tfrac{5}{2^8}, \tfrac{13}{2^8}, \tfrac{13}{2^8}\right)\}.
\end{align*}
\end{lemma}
\begin{proof}
The result follows by solving computationally (see~\cite[majorana.s]{code}) the system 
$$
\begin{cases}
p_1(x,y,z_0, t_0)=0 & \\
p_2(x,y,z_0, t_0)=0 & \\
p_3(x,y,z_0, t_0)=0 & \\
p_4(x,y,z_0, t_0)=0 & \\
\end{cases}
$$
for each $\{z_0, t_0\}\subseteq  \{ 1, 0, \tfrac{1}{8}, \tfrac{1}{32}, \tfrac{1}{64}, \tfrac{13}{2^8}, \tfrac{3}{2^7}, \tfrac{5}{2^8} \}$. 
\end{proof}

By~\cite{IPSS10, HRS}, for each $P$ in the  list of Lemma~\ref{solgen2} , $U_{P}$ has a quotient $U_P/I_P$ isomorphic to a Norton-Sakuma algebra $A_P$. On the other hand, computing the dimension of $U_P$ with arguments similar to those used in the proof of Theorem~\ref{symmetric} (see also~\cite[Proof of Lemma 8.6]{HRS}), we see that  $dim_\F(U_P)=dim_\F(A_P)$, whence $U_P$ is isomorphic to one of the nine Norton-Sakuma algebras. 
By Corollary~4.13 in~\cite{KMS}, the algebra $A_P$ is simple, provided it is not of type $2B$, so, in that case, $V$ is isomorphic to $A_P$.  If $A_P$ is of type $2B$, then $A_P$ has dimension $2$ and the result follows trivially since its proper quotients are isomorphic to the Norton-Sakuma algebra of type $1A$.
%%%%%%%%%%%%%%%
%%%%%%%%%%%%%%%%%%%%%%%%%%%%%%%%%%%%%%%%%%%%
\section{Acknowledgements}
This work was supported by the ``National Group for Algebraic and Geometric Structures, and their Applications" (GNSAGA - INdAM).

%%%%%%%%%%%%%%%%%%%%%%%%%%%%%%%%%%%%%%%%%%%%%%%%%%%%%%%%%%%%%

\end{document}